\documentclass[11pt,leqno]{article}
\usepackage{amsmath,amsthm,amsfonts,amssymb}
\usepackage[colorlinks=true, pdfstartview=FitV, linkcolor=blue, citecolor=red, urlcolor=blue]{hyperref}
\hypersetup{breaklinks=true}

\usepackage{geometry}
\newcommand{\N}{{\mathbb N}}
\newcommand{\Z}{{\mathbb Z}}
\newcommand{\R}{{\mathbb R}}
\newcommand{\C}{{\mathbb C}}

\newcommand{\A}{{\mathbb A}}
\newcommand{\E}{{\mathbb E}}
\newcommand{\F}{{\mathbb F}}
\newcommand{\B}{{\mathbb B}}
\newcommand{\G}{{\mathbb G}}
\newcommand{\PP}{{\mathbb P}}

\newcommand{\eps}{{\varepsilon}}
\newcommand{\tauetabar}{{(\underline{\tau},\underline{\eta})}}
\newcommand{\utau}{\underline{\tau}}
\newcommand{\ueta}{\underline{\eta}}

\newcommand\cU{{\mathcal U}}
\newcommand\cL{{\mathcal L}}

\newcommand\cB{{\mathcal B}}

\newcommand\cV{{\mathcal V}}

\newcommand\cC{{\mathcal C}}

\newcommand\cN{{\mathcal N}}

\newcommand\cF{{\mathcal F}}

\newcommand\cM{{\mathcal M}}
\newcommand\cT{{\mathcal T}}

\newcommand{\uxi}{\underline{\xi}}
\newcommand{\uomega}{\underline{\omega}}

\newcommand\adots{\mathinner{\mkern2mu\raise1pt\hbox{.}
\mkern3mu\raise4pt\hbox{.}\mkern1mu\raise7pt\hbox{.}}}

\newtheorem{theo}{Theorem}[section]
\newtheorem{prop}[theo]{Proposition}
\newtheorem{cor}[theo]{Corollary}
\newtheorem{lem}[theo]{Lemma}
\newtheorem{defn}[theo]{Definition}

\newtheorem{rem}[theo]{Remark}

\newtheorem{assumption}[theo]{Assumption}

\numberwithin{equation}{section}

\title{The Mach stem equation and amplification\\
in strongly nonlinear geometric optics}

\author{Jean-Fran\c{c}ois {\sc Coulombel}\thanks{CNRS and Universit\'e de Nantes, Laboratoire de math\'ematiques 
Jean Leray (UMR CNRS 6629), 2 rue de la Houssini\`ere, BP 92208, 44322 Nantes Cedex 3, France. Email: 
{\tt jean-francois.coulombel@univ-nantes.fr}. Research of J.-F. C. was supported by ANR project BoND, 
ANR-13-BS01-0009-01.} $\,$ \& 
Mark {\sc Williams}\thanks{University of North Carolina, Mathematics Department, CB 3250, Phillips Hall, 
Chapel Hill, NC 27599. USA. Email: {\tt williams@email.unc.edu}. Research of M.W. was partially supported by 
NSF grants number DMS-0701201 and DMS-1001616.}}

\begin{document}

\maketitle

\begin{abstract}
We study highly oscillating solutions to a class of weakly well-posed hyperbolic initial boundary value problems. 
Weak well-posedness is associated with an amplification phenomenon of oscillating waves on the boundary. 
In the previous works \cite{CGW3,CW5}, we have rigorously justified a {\sl weakly nonlinear} regime for 
{\sl semilinear} problems. In that case, the forcing term on the boundary has amplitude $O(\eps^2)$ and 
oscillates at a frequency $O(1/\eps)$. The corresponding exact solution, which has been shown to exist on 
a time interval that is independent of $\eps \in (0,1]$, has amplitude $O(\eps)$. In this paper, we deal with 
the exact same scaling, namely $O(\eps^2)$ forcing term on the boundary and $O(\eps)$ solution, for 
{\sl quasilinear} problems. In analogy with \cite{CGM1}, this corresponds to a {\sl strongly nonlinear} regime, 
and our main result proves solvability for the corresponding WKB cascade of equations, which yields 
existence of approximate solutions on a time interval that is independent of $\eps \in (0,1]$. Existence 
of exact solutions close to approximate ones is a stability issue which, as shown in \cite{CGM1}, highly 
depends on the hyperbolic system and on the boundary conditions; we do not address that question here.

This work encompasses previous formal expansions in the case of weakly stable shock waves \cite{MR} 
and two-dimensional compressible vortex sheets \cite{AM}. In particular, we prove well-posedness for 
the leading amplitude equation (the ``Mach stem equation") of \cite{MR} and generalize its derivation 
to a large class of hyperbolic boundary value problems and to periodic forcing terms. The latter case is 
solved under a crucial nonresonant assumption and a small divisor condition.
\end{abstract}

\bigskip
\tableofcontents

\newpage
\section{Introduction}
\label{intro}

\subsection{General presentation}

\emph{\quad} This article is devoted to the analysis of high frequency solutions to quasilinear hyperbolic 
initial boundary value problems. Up to now, the rigorous construction of such solutions is known in only a 
few situations and highly depends on the well-posedness properties of the boundary value problem one 
considers. In the case where the so-called {\sl uniform Kreiss-Lopatinskii condition} is satisfied, the existence 
of {\sl exact} oscillating solutions on a fixed time interval has been proved by one of the authors in \cite{W1}, 
see also \cite{williams1} for semilinear problems. The asymptotic behavior of exact solutions as the wavelength 
tends to zero is described in \cite{CGW1} for wavetrains and in \cite{CW4} for pulses. The main difference 
between the two problems is that in the wavetrains case, resonances can occur between a combination of 
three phases, giving rise to integro-differential terms in the equation that governs the leading amplitude of 
the solution\footnote{This is not specific to the boundary conditions and is also true for the Cauchy problem 
in the whole space, see, e.g., \cite{HMR,JMR1}.}. Resonances do not occur at the leading 
order\footnote{Interactions between pulses associated with different phases need to be considered only 
when dealing with the construction of correctors to the leading amplitude.} for pulses, which makes the 
leading amplitude equation easier to deal with in that case.

In this article, we pursue our study of {\sl weakly well-posed} problems and consider situations where 
the uniform Kreiss-Lopatinskii condition breaks down. Let us recall that in that case, high frequency 
oscillations can be amplified when reflected on the boundary. As far as we know, this phenomenon 
was first identified by Majda and his collaborators, see for instance \cite{MR,AM,MA} in connection 
with the formation of specific wave patterns in compressible fluid dynamics. Asymptotic expansions 
in the spirit of \cite{AM} are also performed in the recent work \cite{wangyu}. In various situations 
(depending on the scaling of the source terms and on the number of phases), these authors managed 
to derive an equation that governs the leading amplitude of the solution. Solving the leading amplitude 
equation in \cite{MR} and constructing exact and/or approximate oscillating solutions was left open. As 
far as we know, the rigorous justification of such expansions has not been considered in the literature 
so far. The present article follows previous works where we have given a rigorous justification of the 
amplification phenomenon: first for {\sl linear} problems in \cite{CG}, and then for {\sl semilinear} 
problems in \cite{CGW3,CW5}. These previous works considered either linear problems, or a weakly 
nonlinear regime of oscillating solutions for which the existence and asymptotic behavior of {\sl exact} 
oscillating solutions can be studied on a fixed time interval.

The regime considered in \cite{MR,AM}, and that we shall also consider in this article, goes beyond 
the one considered in \cite{CGW3,CW5}. In analogy with \cite{CGM1}, this regime will be referred to 
as that of {\sl strong oscillations}. We extend the analysis of \cite{MR,AM} to a general framework, not 
restricted to the system of gas dynamics, and explain why the problem of vortex sheets considered 
in \cite{AM} and the analogous one in \cite{wangyu} yield a much simpler equation than the problem 
of shock waves in \cite{MR}. We also clarify the causality arguments used in \cite{MR,AM} to discard 
some of the terms in the (formal) asymptotic expansion of the highly oscillating solution. We need 
however to make a crucial assumption in order to analyze this asymptotic expansion, namely we 
need to assume that no resonance occurs between the phases. This is no major concern for pulses 
because interactions are not visible at the leading order, and this may be the reason why this aspect 
was not mentioned in \cite{MR}. Resonances can have far worse consequences when dealing with 
wavetrains, and what saves the day in \cite{AM,wangyu} is that there are too few phases to allow for 
resonances. This explains why the amplitude equation in \cite{AM} and the corresponding one in 
\cite{wangyu} reduce to the standard Burgers equation. When the system admits at least three phases 
(two incoming and one outgoing), and even in the absence of resonances, the amplification phenomenon 
gives rise, as in \cite{MR}, to integro-differential terms in the equation for the amplitude that determines 
the trace of the leading profile. We refer to the latter equation as ``the Mach stem equation", and show 
how it arises more generally in weakly stable (WR class\footnote{The WR class is described in Assumption 
\ref{assumption3} below.}) hyperbolic boundary problems with a strongly nonlinear scaling, both in the 
wavetrain setting, where the equation we derive appears to be completely new, and in the pulse setting, 
where the equation coincides with the one derived in \cite{MR}.

Our main results  establish the well-posedness of the Mach stem equation in both settings, and then use 
those solvability results to construct {\sl approximate} highly oscillating solutions on a {\sl fixed} time interval 
to the underlying hyperbolic boundary value problems. In the wavetrain case we are able to construct 
approximate solutions of arbitrarily high order under a crucial nonresonant condition and a small divisor 
condition; in the pulse setting we construct approximate solutions up to the point at which further ``correctors" 
are too large to be regarded as correctors.

In Appendix \ref{appB} we compute the formal large period limit of the Mach stem equation for wavetrains, 
and find a surprising discrepancy (described further below) between that limit and the Mach stem equation 
for pulses derived in \cite{MR}.

The construction of {\sl exact} oscillating solutions close to approximate ones is a stability issue that is far 
from obvious in such a strong scaling. We refer to \cite{CGM1} for indications on possible instability issues 
and postpone the stability problem in our context to a future work. In any case, it is very likely that no general 
answer can be given and that stability vs instability of the family of approximate solutions will depend on the 
system and/or on the boundary conditions, see for instance \cite{CGM2} for further results in this direction.

The precise Mach stem formation mechanism described in \cite{MR} for reacting shock fronts comes from 
wave breaking (blow-up) in the "Mach stem equation". The numerical simulations in \cite{MR2} suggest that 
the latter equation displays a similar blow-up phenomenon as the corresponding one for the Burgers equation. 
Our results show that, in a precise functional setting, the Mach stem equation is a {\sl semilinear} perturbation 
of the Burgers equation, which might suggest that the hint in \cite{MR2} is true. However, the semilinear 
perturbation in the Mach stem equation takes the form of a bilinear Fourier multiplier which makes the 
rigorous justification of a blow-up result difficult. We also postpone this rigorous justification to a future 
work.

\subsection*{Notation}

Throughout this article, we let ${\mathcal M}_{n,N}({\mathbb K})$ denote the set of $n \times N$ matrices with 
entries in ${\mathbb K} = \R \text{ or }\C$, and we use the notation ${\mathcal M}_N({\mathbb K})$ when $n=N$. 
We let $I$ denote the identity matrix, without mentioning the dimension. The norm of a (column) vector $X \in \C^N$ 
is $|X| := (X^* \, X)^{1/2}$, where the row vector $X^*$ denotes the conjugate transpose of $X$. If $X,Y$ are two 
vectors in $\C^N$, we let $X \cdot Y$ denote the quantity $\sum_j X_j \, Y_j$, which coincides with the usual 
scalar product in $\R^N$ when $X$ and $Y$ are real. We often use Einstein's summation convention in 
order to make some expressions easier to read.

The letter $C$ always denotes a positive constant that may vary from line to line or within the same line. 
Dependance of the constant $C$ on various parameters is made precise throughout the text. The sign 
$\lesssim$ means $\le$ up to a multiplicative constant.

\subsection{The equations and main assumptions}

In the space domain $\R^d_+ := \{ x=(y,x_d) \in \R^{d-1} \times \R \, : \, x_d>0 \}$, we consider the quasilinear 
evolution problem with oscillating source term:
\begin{equation}
\label{0}
\begin{cases}
\partial_t u_\eps +\sum^{d}_{j=1} A_j (u_\eps) \, \partial_j u_\eps =0 \, ,& t \le T\, , \, x \in \R^d_+ \, ,\\
b(u_\eps |_{x_d=0}) =\eps^2 \, G \left( t,y,\dfrac{\varphi_0(t,y)}{\eps} \right) \, ,& t \le T\, , \, y \in \R^{d-1} \, ,\\
u_\eps,G |_{t<0} =0 \, ,
\end{cases}
\end{equation}
where the $A_j$'s belong to ${\mathcal M}_N(\R)$ and depend in a ${\mathcal C}^\infty$ way on $u$ in 
a neighborhood of $0$ in $\R^N$, $b$ is a ${\mathcal C}^\infty$ mapping from a neighborhood of $0$ in 
$\R^N$ to $\R^p$ (the integer $p$ is made precise below), and the source term $G$ is valued in $\R^p$. 
It is also assumed that $b(0)=0$, so that the solution starts from the rest state $0$ in negative times and 
is ignited by the small oscillating source term $\eps^2 \, G$ on the boundary in positive times. The two main 
underlying questions of nonlinear geometric optics are:
\begin{enumerate}
   \item Proving existence of solutions to \eqref{0} on a fixed time interval (the time $T>0$ should be 
   independent of the wavelength $\eps \in (0,1]$).
   
   \item Studying the asymptotic behavior of the sequence $u_\eps$ as $\eps$ tends to zero. If we let 
   $u^{\rm app}_\eps$ denote an approximate solution on $[0,T']$, $T' \le T$, constructed by the methods 
   of nonlinear geometric optics (that is, solving eikonal equations for phases and suitable transport 
   equations for profiles), how well does $u^{\rm app}_\eps$ approximate $u_\eps$ for $\eps$ small ? 
   For example, is it true that
\begin{equation*}
\lim_{\eps\to 0} \, \, \| u_\eps-u^{\rm app}_\eps \|_{L^\infty ([0,T'] \times \R^d_+)} \to 0 \, ?
\end{equation*}
\end{enumerate}

The above questions are dealt with in a different way according to the functional setting chosen for the 
source term $G$ in \eqref{0}. More precisely, we distinguish between:
\begin{itemize}
   \item Wavetrains, for which $G$ is a function defined on $(-\infty,T_0] \times \R^{d-1} \times \R$ that is 
   $\Theta$-periodic with respect to its last argument (denoted $\theta_0$ later on).
   
   \item Pulses, for which $G$ is a function defined on $(-\infty,T_0] \times \R^{d-1} \times \R$ that has 
   at least polynomial decay at infinity with respect to its last argument.
\end{itemize}

The answer to the above two questions highly depends on the well-posedness of the linearized system at 
the origin:
\begin{equation}
\label{1}
\begin{cases}
\partial_t v +\sum^{d}_{j=1} A_j (0) \, \partial_j v =f \, ,& t \le T\, , \, x \in \R^d_+ \, ,\\
{\rm d}b(0) \cdot v|_{x_d=0} =g \, ,& t \le T\, , \, y \in \R^{d-1} \, ,\\
v,f,g |_{t<0} =0 \, .
\end{cases}
\end{equation}
The first main assumption for the linearized problem \eqref{1} deals with hyperbolicity.

\begin{assumption}[Hyperbolicity with constant multiplicity]
\label{assumption1}
There exist an integer $q \ge 1$, some real functions $\lambda_1,\dots,\lambda_q$ that are analytic on $\R^d
\setminus \{ 0 \}$ and homogeneous of degree $1$, and there exist some positive integers $\nu_1,\dots,\nu_q$
such that:
\begin{equation*}
\forall \, \xi=(\xi_1,\dots,\xi_d) \in \R^d \setminus \{ 0 \} \, ,\quad
\det \Big[ \tau \, I+\sum_{j=1}^d \xi_j \, A_j(0) \Big] =\prod_{k=1}^q \big( \tau+\lambda_k(\xi) \big)^{\nu_k} \, .
\end{equation*}
Moreover the eigenvalues $\lambda_1(\xi),\dots,\lambda_q(\xi)$ are semi-simple (their algebraic multiplicity
equals their geometric multiplicity) and satisfy $\lambda_1(\xi)<\dots<\lambda_q(\xi)$ for all $\xi \in \R^d
\setminus \{ 0\}$.
\end{assumption}

\noindent For reasons that will be fully explained in Sections \ref{sect2} and \ref{sect4}, we make a technical 
complementary assumption.

\begin{assumption}[Strict hyperbolicity or conservative structure]
\label{assumption1'}
In Assumption \ref{assumption1}, either all integers $\nu_1,\dots,\nu_q$ equal $1$ (which means that the 
operator $\partial_t +\sum_j A_j(0) \, \partial_j$ is strictly hyperbolic), or $A_1(u),\dots,A_d(u)$ are Jacobian 
matrices of some flux functions $f_1,\dots,f_d$ that depend in a ${\mathcal C}^\infty$ way on $u$ in a 
neighborhood of $0$ in $\R^N$. In the latter case, Assumption \ref{assumption1} holds for all $u$ close 
to the origin, namely for an open neighborhood ${\mathbb O}$ of  $0\in\R^N$, there exist an integer $q \ge 1$, 
some real functions $\lambda_1,\dots,\lambda_q$ that are $C^\infty$ on ${\mathbb O} \times \R^d \setminus 
\{ 0\}$ and homogeneous of degree $1$ and analytic in $\xi$, and there exist some positive integers 
$\nu_1,\dots,\nu_q$ such that:
\begin{equation*}
\det \Big[ \tau \, I+\sum_{j=1}^d \xi_j \, A_j(u) \Big] =\prod_{k=1}^q \big( \tau+\lambda_k(u,\xi) \big)^{\nu_k} \, ,
\end{equation*}
for $u \in {\mathbb O}$, $\tau \in \R$ and $\xi=(\xi_1,\dots,\xi_d) \in \R^d \setminus \{ 0 \}$. Moreover the eigenvalues 
$\lambda_1(u,\xi),\dots,\lambda_q(u,\xi)$ are semi-simple and satisfy $\lambda_1(u,\xi)<\dots<\lambda_q(u,\xi)$ 
for all $u \in {\mathbb O}$, $\xi=(\xi_1,\dots,\xi_d) \in \R^d \setminus \{ 0 \}$.
\end{assumption}

\noindent For simplicity, we restrict our analysis to noncharacteristic boundaries and therefore make the following:

\begin{assumption}[Noncharacteristic boundary]
\label{assumption2}
The matrix $A_d(0)$ is invertible and the Jacobian matrix $B:={\rm d}b(0)$ has maximal rank, its rank $p$ being 
equal to the number of positive eigenvalues of $A_d(0)$ (counted with their multiplicity). Moreover, the integer 
$p$ satisfies $1 \le p \le N-1$.
\end{assumption}

Energy estimates for solutions to \eqref{1} are based on the normal mode analysis, see, e.g., \cite[chapter 4]{BS}. 
We let $\tau-i\, \gamma \in \C$ and $\eta \in \R^{d-1}$ denote the dual variables of $t$ and $y$ in the Laplace 
and Fourier transform, and we introduce the symbol
\begin{equation*}
{\mathcal A}(\zeta):= -i \, A_d(0)^{-1} \left( (\tau-i\gamma) \, I +\sum_{j=1}^{d-1} \eta_j \, A_j(0) \right)
\, ,\quad \zeta:=(\tau-i\gamma,\eta) \in \C \times \R^{d-1} \, .
\end{equation*}
For future use, we also define the following sets of frequencies:
\begin{align*}
& \Xi := \Big\{ (\tau-i\gamma,\eta) \in \C \times \R^{d-1} \setminus (0,0) \, : \, \gamma \ge 0 \Big\} \, ,
& \Sigma := \Big\{ \zeta \in \Xi \, : \, \tau^2 +\gamma^2 +|\eta|^2 =1 \Big\} \, ,\\
& \Xi_0 := \Big\{ (\tau,\eta) \in \R \times \R^{d-1} \setminus (0,0) \Big\} = \Xi \cap \{ \gamma = 0 \} \, ,
& \Sigma_0 := \Sigma \cap \Xi_0 \, .
\end{align*}
Two key objects in our analysis are the hyperbolic region and the glancing set that are defined as follows.

\begin{defn}
\label{def1}
\begin{itemize}
 \item The hyperbolic region ${\mathcal H}$ is the set of all $(\tau,\eta) \in \Xi_0$ such that the matrix
       ${\mathcal A}(\tau,\eta)$ is diagonalizable with purely imaginary eigenvalues.

 \item Let ${\bf G}$ denote the set of all $(\tau,\xi) \in \R \times \R^d$ such that $\xi \neq 0$ and there exists
       an integer $k \in \{1,\dots,q\}$ satisfying
\begin{equation*}
\tau + \lambda_k(\xi) = \dfrac{\partial \lambda_k}{\partial \xi_d} (\xi) = 0 \, .
\end{equation*}
If $\pi ({\bf G})$ denotes the projection of ${\bf G}$ on the first $d$ coordinates (that is, $\pi (\tau,\xi) 
:=(\tau,\xi_1,\dots,\xi_{d-1})$ for all $(\tau,\xi)$), the glancing set ${\mathcal G}$ is ${\mathcal G} :=
\pi ({\bf G}) \subset \Xi_0$.
\end{itemize}
\end{defn}

\noindent We recall the following result that is due to Kreiss \cite{K} in the strictly hyperbolic case (when all 
integers $\nu_j$ in Assumption \ref{assumption1} equal $1$) and to M\'etivier \cite{Met} in our more general 
framework.

\begin{theo}[\cite{K,Met}]
\label{thm1}
Let Assumptions \ref{assumption1} and \ref{assumption2} be satisfied. Then for all $\zeta \in \Xi \setminus
\Xi_0$, the matrix ${\mathcal A}(\zeta)$ has no purely imaginary eigenvalue and its stable subspace $\E^s (\zeta)$
has dimension $p$. Furthermore, $\E^s$ defines an analytic vector bundle over $\Xi \setminus \Xi_0$ that can
be extended as a continuous vector bundle over $\Xi$.
\end{theo}

\noindent For all $(\tau,\eta) \in \Xi_0$, we let $\E^s(\tau,\eta)$ denote the continuous extension of $\E^s$ 
to the point $(\tau,\eta)$. Away from the glancing set ${\mathcal G} \subset \Xi_0$, $\E^s(\zeta)$ depends 
analytically on $\zeta$, see \cite{Met}. In particular, it follows from the analysis in \cite{Met}, see similar 
arguments in \cite{BRSZ,jfc2}, that the hyperbolic region ${\mathcal H}$ does not contain any glancing 
point, and $\E^s(\zeta)$ depends analytically on $\zeta$ in the neighborhood of any point of ${\mathcal H}$. 
We now make our weak stability condition precise (recall the notation $B:={\rm d}b(0)$).

\begin{assumption}[Weak Kreiss-Lopatinskii condition]
\label{assumption3}
\begin{itemize}
 \item For all $\zeta \in \Xi \setminus \Xi_0$, $\text{\rm Ker} B \cap \E^s (\zeta) = \{ 0\}$.

 \item The set $\Upsilon := \{ \zeta \in \Sigma_0 \, : \, \text{\rm Ker} B \cap \E^s (\zeta) \neq \{ 0\} \}$ is
       nonempty and included in the hyperbolic region ${\mathcal H}$.

 \item There exists a neighborhood ${\mathcal V}$ of $\Upsilon$ in $\Sigma$, a real valued ${\mathcal C}^\infty$
       function $\sigma$ defined on ${\mathcal V}$, a basis  $E_1(\zeta),\dots,E_p(\zeta)$ of $\E^s(\zeta)$ that
       is of class ${\mathcal C}^\infty$ with respect to $\zeta \in {\mathcal V}$, and a matrix $P(\zeta) \in
       \text{\rm GL}_p (\C)$ that is of class ${\mathcal C}^\infty$ with respect to $\zeta \in {\mathcal V}$,
       such that
\begin{equation*}
\forall \, \zeta \in {\mathcal V} \, ,\quad B \, \begin{pmatrix}
E_1(\zeta) & \dots & E_p(\zeta) \end{pmatrix}
= P(\zeta) \, \text{\rm diag } \big( \gamma +i\, \sigma (\zeta),1,\dots,1 \big) \, .
\end{equation*}
\end{itemize}
\end{assumption}

As explained in \cite{CG,CGW3,CW5}, Assumption \ref{assumption3} is a more convenient description of 
the so-called WR class of \cite{BRSZ}. Let us recall that this class consists of hyperbolic boundary value 
problems for which the uniform Kreiss-Lopatinskii condition breaks down "at first order" in the hyperbolic 
region\footnote{Let us also recall that the uniform Kreiss-Lopatinskii condition is satisfied when 
$\text{\rm Ker} B \cap \E^s (\zeta) = \{ 0\}$ for all $\zeta \in \Xi$, and not only for $\zeta \in \Xi \setminus \Xi_0$.}. 
This class is meaningful for nonlinear problems because it is stable by perturbations of the matrices $A_j(0)$ 
and of the boundary conditions $B$.

Our final assumption deals with the phase $\varphi_0$ occuring in \eqref{0}.

\begin{assumption}[Critical phase]
\label{assumption4}
The phase $\varphi_0$ in \eqref{0} is defined by
\begin{equation*}
\varphi_0 (t,y) :=\underline{\tau} \, t +\underline{\eta} \cdot y \, ,
\end{equation*}
with $\tauetabar \in \Upsilon$. In particular, there holds $\tauetabar \in {\mathcal H}$.
\end{assumption}

The shock waves problem considered in \cite{MR} enters the framework defined by Assumptions \ref{assumption1}, 
\ref{assumption1'}, \ref{assumption2}, \ref{assumption3} and \ref{assumption4} with the additional difficulty that the 
space domain has a free boundary. The vortex sheets problem considered in \cite{AM} and the analogous one in 
\cite{wangyu} violate Assumption \ref{assumption2} but these problems share all main features which we consider 
here. We restrict our analysis to fixed noncharacteristic boundaries mostly for convenience and simplicity of notation.

Our main results deal with the existence of {\sl approximate} solutions to \eqref{0}. This is the reason why we 
only make assumptions on the linearized problem at the origin \eqref{1}, and not on the full nonlinear problem 
\eqref{0}.

\subsection{Main result for wavetrains}

In Part \ref{part1}, we consider the nonlinear problem \eqref{0} with a source term $G$ that is $\Theta$-periodic 
with respect to its last argument $\theta_0$. As evidenced in several previous works, the asymptotic behavior of 
the solution $u_\eps$ to \eqref{0} is described in terms of the characteristic phases whose trace on the boundary 
equals $\varphi_0$. We thus consider the pairwise distinct roots (and all the roots are real) $\underline{\omega}_1,
\dots,\underline{\omega}_M$ to the dispersion relation
\begin{equation*}
\det \Big[ \underline{\tau} \, I+\sum_{j=1}^{d-1} \underline{\eta}_j \, A_j(0) +\omega \, A_d(0) \Big] = 0 \, .
\end{equation*}
To each $\underline{\omega}_m$ there corresponds a unique integer $k_m \in \{ 1,\dots,q\}$ such that
$\underline{\tau} + \lambda_{k_m} (\underline{\eta},\underline{\omega}_m)=0$. We can then define the following
real phases and their associated group velocity:
\begin{equation}
\label{phases}
\forall \, m =1,\dots,M \, ,\quad \varphi_m (t,x):= \varphi(t,y)+\underline{\omega}_m \, x_d \, ,\quad 
{\bf v}_m := \nabla \lambda_{k_m} (\underline{\eta},\underline{\omega}_m) \, .
\end{equation}
We let $\Phi :=(\varphi_1,\dots,\varphi_M)$ denote the collection of phases. Each group velocity ${\bf v}_m$ 
is either incoming or outgoing with respect to the space domain $\R^d_+$: the last coordinate of ${\bf v}_m$ 
is nonzero. This property holds because $\tauetabar$ does not belong to the glancing set ${\mathcal G}$.

\begin{defn}
\label{def2}
The phase $\varphi_m$ is incoming if the group velocity ${\bf v}_m$ is incoming ($\partial_{\xi_d} \lambda_{k_m}
(\underline{\eta},\underline{\omega}_m) >0$), and it is outgoing if the group velocity ${\bf v}_m$ is outgoing
($\partial_{\xi_d} \lambda_{k_m} (\underline{\eta},\underline{\omega}_m) <0$).
\end{defn}

\noindent In all what follows, we let ${\mathcal I}$ denote the set of indices $m \in \{ 1,\dots,M\}$ such 
that $\varphi_m$ is incoming, and ${\mathcal O}$ denote the set of indices $m \in \{ 1,\dots,M\}$ such 
that $\varphi_m$ is outgoing. Under Assumption \ref{assumption2}, both ${\mathcal I}$ and ${\mathcal O}$ 
are nonempty, as follows from \cite[Lemma 3.1]{CG} which we recall later on.

The proof of our main result for wavetrains, that is Theorem \ref{theowavetrains} below, heavily relies on 
the nonresonance assumption below. For later use, we introduce the following notation: if $0 \le k \le M$, 
we let $\Z^{M;k}$ denote the subset of all $\alpha \in \Z^M$ such that at most $k$ coordinates of $\alpha$ 
are nonzero. For instance $\Z^{M;1}$ is the union of the sets $\Z \, {\bf e}_m$, $m=1,\dots,M$, where 
$({\bf e}_1,\dots,{\bf e}_M)$ denotes the canonical basis of $\R^M$. We also introduce the notation:
\begin{equation}
\label{defopL}
L(\tau,\xi) :=\tau \, I +\sum_{j=1}^d \xi_j \, A_j(0) \, ,\quad L(\partial) :=\partial_t +\sum_{j=1}^d A_j(0) \, \partial_j \, .
\end{equation}
The nonresonance assumption reads as follows.

\begin{assumption}[Nonresonance and small divisor condition]
\label{assumption5}
The phases are nonresonant, that is for all $\alpha \in \Z^M \setminus \Z^{M;1}$, there holds 
$\det L({\rm d} (\alpha \cdot \Phi)) \neq 0$, where $\alpha \cdot \Phi := \alpha_m \, \varphi_m$.

Furthermore, there exists a constant $c>0$ and a real number $\gamma$ such that for all $\alpha \in 
\Z^M \setminus \Z^{M;1}$ that satisfies $\alpha_m =0$ for all $m \in {\mathcal O}$, there holds
\begin{equation*}
|\det L({\rm d} (\alpha \cdot \Phi))| \ge c \, |\alpha|^{-\gamma} \, .
\end{equation*}
\end{assumption}

Let us note that the small divisor condition is only required for $\alpha$ with nonzero components 
$\alpha_m$ which correspond to incoming phases. If there is only one incoming phase, then there is 
no such $\alpha$ with at least two nonzero coordinates, and we do not need any small divisor condition. 
The reason for this simplification will be explained in Sections \ref{sect2} and \ref{sect3}. Our main result 
reads as follows.

\begin{theo}
\label{theowavetrains}
Let Assumptions \ref{assumption1}, \ref{assumption1'}, \ref{assumption2}, \ref{assumption3}, \ref{assumption4}, 
\ref{assumption5} be satisfied, let $T_0>0$ and consider $G \in {\mathcal C}^\infty ((-\infty,T_0]_t; H^{+\infty} 
(\R^{d-1}_y \times (\R /(\Theta \, \Z))_{\theta_0}))$ that vanishes for $t<0$. Then there exists $0<T \le T_0$ 
and there exists a unique sequence of profiles $(\cU_n)_{n \ge 0}$ in ${\mathcal C}^\infty ((-\infty,T]; H^{+\infty} 
(\R^d_+ \times (\R /(\Theta \, \Z))^M))$ that satisfies the WKB cascade \eqref{BKWint}, \eqref{BKWbord} below, 
and $\cU_n|_{t<0}=0$ for all $n \in \N$. In particular, each profile $\cU_k$ has its $\theta$-spectrum included in 
the set
\begin{equation*}
\Z^M_{\mathcal I} := \big\{ \alpha \in \Z^M \, / \, \forall \, m \in {\mathcal O} \, , \, \alpha_m=0 \big\} \, ,
\end{equation*}
which means that no outgoing signal is generated at any order.

Furthermore, if for all integers $N_1, N_2 \ge 0$, we define the approximate solution
\begin{equation*}
u_\eps^{{\rm app},N_1,N_2}(t,x) := \sum_{n=0}^{N_1+N_2} \eps^{1+n} \, \cU_n \left( t,x,\dfrac{\Phi(t,x)}{\eps} \right) \, ,
\end{equation*}
then
\begin{equation*}
\begin{cases}
\partial_t u_\eps^{{\rm app},N_1,N_2} +{\displaystyle \sum^{d}_{j=1}} A_j (u_\eps^{{\rm app},N_1,N_2}) \, 
\partial_j u_\eps^{{\rm app},N_1,N_2} = O(\eps^{N_1+1}) \, ,& t \le T\, , \, x \in \R^d_+ \, ,\\
b(u_\eps^{{\rm app},N_1,N_2}|_{x_d=0}) =\eps^2 \, G \left( t,y,\dfrac{\varphi_0(t,y)}{\eps} \right) +O(\eps^{N_1+2}) \, ,& 
t \le T\, , \, y \in \R^{d-1} \, ,\\
u_\eps^{{\rm app},N_1,N_2}|_{t<0} =0 \, ,
\end{cases}
\end{equation*}
where the $O(\eps^{N_1+1})$ in the interior equation and $O(\eps^{N_1+2})$ in the boundary conditions are measured 
respectively in the ${\mathcal C} ((-\infty,T];H^{N_2}(\R^d_+)) \cap L^\infty((-\infty,T] \times \R^d_+)$ and ${\mathcal C} 
((-\infty,T];H^{N_2}(\R^{d-1})) \cap L^\infty((-\infty,T] \times \R^{d-1})$ norms.
\end{theo}

Of course, the approximate solutions provided by Theorem \ref{theowavetrains} become interesting only 
for $N_1 \ge 1$, that is, when the remainder $O(\eps^{N_1+2})$ on the boundary becomes smaller than 
the source term $\eps^2 \, G$.

The spectrum property in Theorem \ref{theowavetrains} is a {\sl rigorous} justification of the causality arguments 
used in \cite{AM,wangyu}. Theorem \ref{theowavetrains} will be proved in Part \ref{part1} of this article. In Section 
\ref{sect2}, we shall derive the so-called leading amplitude (Mach stem) equation from which the leading profile 
$\cU_0$ is constructed. Section \ref{sect3} is devoted to proving well-posedness for this evolution equation. As far 
as we know, the bilinear Fourier multiplier that we shall encounter in this equation had not appeared earlier in the 
geometric optics context and our main task is to prove a {\sl tame boundedness} estimate for this multiplier. Section 
\ref{sect4} is devoted to the construction of the correctors $\cU_n$, $n \ge 1$, and to completing the proof of 
Theorem \ref{theowavetrains}. We refer to Appendix \ref{appA} for a discussion of the two-dimensional isentropic 
Euler equations, with specific emphasis on Assumption \ref{assumption5}.

\subsection{Main result for pulses}

We keep the same notation \eqref{phases} for the phases, but now consider the nonlinear problem \eqref{0} with 
a source term $G$ that has "polynomial decay" with respect to its last argument $\theta_0$. This behavior is made 
precise by introducing the following weighted Sobolev spaces:
\begin{equation*}
\Gamma^k (\R^d):= \left\{ u \in L^2(\R^{d-1}_y \times \R_\theta) \, : \, \theta^\alpha \, \partial_{y,\theta}^\beta u \in 
L^2 (\R^d) \quad \text{\rm if }Ê\alpha +|\beta| \le k \right\} \, .
\end{equation*}
Our second main result reads as follows.

\begin{theo}
\label{theopulses}
Let Assumptions \ref{assumption1}, \ref{assumption1'}, \ref{assumption2}, \ref{assumption3}, \ref{assumption4} 
be satisfied. Let $k_0$ denote the smallest integer satisfying $k_0>(d+1)/2$, and let $K_0,K_1$ be two integers 
such that $K_0>8+(d+2)/2$, $K_1 -K_0 \ge 2\, k_0+2$. Let $T_0>0$ and consider
\begin{equation*}
G \in \cap_{\ell=0}^{K_0-1} {\mathcal C}^\ell ((-\infty,T_0]_t; \Gamma^{K_1-\ell}(\R^d_{y,\theta})) \, ,
\end{equation*}
that vanishes for $t<0$. Then there exists $0<T \le T_0$ and there exist profiles $\cU_0,\cU_1,\cU_2 $ vanishing in 
$t<0$ that satisfy the WKB cascade \eqref{3p}, \eqref{4p} below. If we define the approximate solution
\begin{equation*}
u_\eps^{\rm app}(t,x) := \sum_{n=0}^2 \eps^{1+n} \, \cU_n \left( t,x,\dfrac{\varphi_0(t,y)}{\eps},\dfrac{x_d}{\eps} \right) \, ,
\end{equation*}
then
\begin{equation*}
\begin{cases}
\partial_t u_\eps^{\rm app} +{\displaystyle \sum^{d}_{j=1}} A_j (u_\eps^{\rm app}) \, \partial_j u_\eps^{\rm app} 
= O(\eps^3) \, ,& t \le T\, , \, x \in \R^d_+ \, ,\\
b(u_\eps^{\rm app}|_{x_d=0}) =\eps^2 \, G \left( t,y,\dfrac{\varphi_0(t,y)}{\eps} \right) +O(\eps^3) \, ,& 
t \le T\, , \, y \in \R^{d-1} \, ,\\
u_\eps^{\rm app}|_{t<0} =0 \, ,
\end{cases}
\end{equation*}
where the $O(\eps^3)$ in the interior equation and in the boundary conditions are measured respectively in the 
${\mathcal C} ((-\infty,T];L^2(\R^d_+)) \cap L^\infty((-\infty,T] \times \R^d_+)$ and ${\mathcal C} ((-\infty,T]; 
L^2(\R^{d-1})) \cap L^\infty((-\infty,T] \times \R^{d-1})$ norms.

The exact regularity and decay properties of the profiles are given in \eqref{a9}.
\end{theo}

The approximate solution provided by Theorem \ref{theopulses} is constructed, as in \cite{MR}, according to the 
following line of thought: we expect that the exact solution $u_\eps$ to \eqref{0} admits an asymptotic expansion 
of the form
\begin{equation*}
u_\eps \sim \eps \, \sum_{k \ge 0} \eps^k \, \cU_k \left(t,x,\dfrac{\varphi_0(t,y)}{\eps},\dfrac{x_d}{\eps} \right) \, ,
\end{equation*}
that is either finite up to an order $K \ge 2$, or infinite. We plug this ansatz and try to identify each profile 
$\cU_k$. The corrector $\eps \, \cU_1$ is expected to be negligible with respect to $\cU_0$, and so on for 
higher indices. Hence the identification of the profiles is based on some {\sl boundedness assumption} for 
the correctors to the leading profile. Of course, such assumptions have to be verified a posteriori when 
constructing $\cU_0$, $\cU_1$ and so on. For instance, Theorem \ref{theowavetrains} is based on the 
assumption that one can decompose $u_\eps$ with profiles in $H^\infty$, and we give a rigorous construction 
of such profiles for which the corrector
\begin{equation*}
\eps^{1+n} \, \cU_n \left( t,x,\dfrac{\Phi(t,x)}{\eps} \right) \, ,
\end{equation*}
is indeed an $O(\eps^{1+n})$ in $L^\infty ((-\infty,T] \times \R^d_+)$.

In Sections \ref{sect6} and \ref{sect7}, we give a rigorous construction of the leading profile $\cU_0$ and of 
the first two correctors $\cU_1$, $\cU_2$ that satisfy all the boundedness and integrability properties on which 
the derivation of the leading amplitude equation relies. In particular in section \ref{sect6} we explain why, assuming 
that the first and second correctors $\cU_1$, $\cU_2$ satisfy some boundedness and integrability properties in 
the stretched variables $(t,x,\theta_0,\xi_d)$, the leading profile $\cU_0$ is necessarily determined by an amplitude 
equation that is entirely similar to the one in \cite{MR}. The analysis of Section \ref{sect6} clarifies some of the 
causality arguments used in \cite{MR}. This makes the arguments of \cite{MR} consistent, and one of our 
achievements is to prove in section \ref{sect7} local well-posedness for the leading amplitude equation derived 
in \cite{MR} (which we call the Mach stem equation).

However, the drawback of this approach is that, surprisingly, it is not consistent with the formal large period 
limit for wavetrains. More precisely, it seems rather reasonable to expect that the pulse problem is obtained 
by considering the analogous problem for wavetrains with a period $\Theta$ and by letting $\Theta$ tend to 
infinity. In particular, the reader can check that the leading amplitude equations derived in \cite{CW4}, resp. 
\cite{CW5}, for {\sl quasilinear uniformly stable} pulse problems, resp. {\sl semilinear weakly stable} pulse 
problems, coincide with the large period limit of the analogous equations obtained in \cite{CGW1}, resp. 
\cite{CGW3}, for wavetrains, even though the latter equations include interaction integrals to account for 
resonances. One could therefore adopt a different point of view and first derive the profile equations for 
pulses by considering the limit $\Theta \rightarrow +\infty$ for wavetrains, and then study the property of the 
corresponding approximate solution. Surprisingly, the two approaches do not give the same leading profile 
$\cU_0$, as we shall explain in Appendix \ref{appB}. It seems very difficult at this stage to decide which of 
the two approximate solutions should be the most ``physically relevant" since we do not have a {\sl nonlinear} 
stability result that would claim that the {\sl exact} solution $u_\eps$ to \eqref{0} is close to one of these two 
approximate solutions on a fixed time interval independent of $\eps$ small enough. The clarification of this 
surprising phenomenon is left to a future work.

\newpage
\part{Highly oscillating wavetrains}
\label{part1}

\section{Construction of approximate solutions: the leading amplitude}
\label{sect2}

\subsection{Some decompositions and notation}
\label{notation}

We recall here some useful results from \cite{CG} and introduce some notation. Recall that the matrix 
${\mathcal A}\tauetabar$ is diagonalizable with eigenvalues $i\, \underline{\omega}_m$, $m=1,\dots,M$.  
The eigenspace of ${\mathcal A}\tauetabar$ for $i\, \underline{\omega}_m$ coincides with the kernel 
of $L({\rm d}\varphi_m)$.

\begin{lem}[\cite{CG}]
\label{lem1}
The (extended) stable subspace $\E^s \tauetabar$ admits the decomposition
\begin{equation}
\label{decomposition1}
\E^s \tauetabar = \oplus_{m \in {\mathcal I}} \, \text{\rm Ker } L({\rm d} \varphi_m) \, ,
\end{equation}
and each vector space in the decomposition \eqref{decomposition1} is of real type (that is, it admits a basis 
of real vectors).
\end{lem}

\begin{lem}[\cite{CG}]
\label{lem2}
The following decompositions hold
\begin{equation}
\label{decomposition2}
\C^N = \oplus_{m=1}^M \, \text{\rm Ker } L({\rm d} \varphi_m) 
= \oplus_{m=1}^M \, A_d(0) \, \text{\rm Ker } L({\rm d} \varphi_m) \, ,
\end{equation}
and each vector space in the decompositions \eqref{decomposition2} is of real type.

We let $P_1,\dots,P_M$, resp. $Q_1,\dots,Q_M$, denote the projectors associated with the first, resp. second, 
decomposition in \eqref{decomposition2}. Then for all $m=1,\dots,M$, there holds $\text{\rm Im } L({\rm d} \varphi_m) 
= \text{\rm Ker } Q_m$.
\end{lem}

\noindent Using Lemma \ref{lem2}, we may introduce the partial inverse $R_m$ of $L({\rm d}\varphi_m)$, 
which is uniquely determined by the relations
\begin{equation*}
\forall \, m=1,\dots,M \, ,\quad R_m \, L({\rm d}\varphi_m) =I-P_m \, ,\quad 
L({\rm d}\varphi_m) \, R_m=I-Q_m \, ,\quad P_m \, R_m=0 \, ,\quad R_m \, Q_m =0 \, .
\end{equation*}

When the system is strictly hyperbolic, which is the case considered in Sections \ref{sect2}, \ref{sect3} 
and most of Section \ref{sect4}, each vector space $\text{\rm Ker } L({\rm d} \varphi_m)$ is one-dimensional 
and $M=N$. The case of conservative hyperbolic systems with constant multiplicity will be dealt with in Paragraph 
\ref{sect5}. In the case of a strictly hyperbolic system, we choose, for all $m=1,\dots,N$, a real vector $r_m$ that 
spans $\text{\rm Ker } L({\rm d} \varphi_m)$. We also choose real row vectors $\ell_1,\dots,\ell_N$, that satisfy
\begin{equation*}
\forall \, m=1,\dots,N \, ,\quad \ell_m \, L({\rm d}\varphi_m) =0 \, ,
\end{equation*}
together with the normalization $\ell_m \, A_d(0) \, r_{m'} =\delta_{mm'}$. With this choice, the partial inverse 
$R_m$ and the projectors $P_m$, $Q_m$ are given by
\begin{equation*}
\forall \, X \in \C^N \, ,\quad R_m \, X=\sum_{m' \neq m} 
\dfrac{\ell_{m'} \, X}{\underline{\omega}_m -\underline{\omega}_{m'}} \, r_{m'} \, ,\quad 
P_m \, X = (\ell_m \, A_d(0) \, X) \, r_m \, ,\quad 
Q_m \, X = (\ell_m \, X) \, A_d(0) \, r_m \, .
\end{equation*}

According to Assumption \ref{assumption3}, $\text{\rm Ker} B \cap \E^s \tauetabar$ is one-dimensional and 
is therefore spanned by some vector $e=\sum_{m \in {\mathcal I}} e_m$, $e_m \in \text{\rm Span } r_m$ 
(here we have used Lemma \ref{lem1}). The vector space $B \, \E^s \tauetabar$ is $(p-1)$-dimensional 
and is of real type. We can therefore write it as the kernel of a real linear form
\begin{equation}
\label{defb}
B\, \E^s \tauetabar = \big\{ X \in \C^p \, , \, \underline{b} \, X = 0 \big\} \, ,
\end{equation}
for a suitable nonzero row vector $\underline{b}$. Eventually, we can introduce the partial inverse of the 
restriction of $B$ to the vector space $\E^s \tauetabar$. More precisely, we choose a supplementary vector 
space of Span $e$ in $\E^s \tauetabar$:
\begin{equation}
\label{decomposition3}
\E^s \tauetabar = \check{\E}^s \tauetabar \oplus \text{\rm Span } e \, .
\end{equation}
The matrix $B$ then induces an isomorphism from $\check{\E}^s \tauetabar$ to the hyperplane 
$B\, \E^s \tauetabar$.

\subsection{Strictly hyperbolic systems of three equations}
\label{sect2example}

For simplicity of notation, we first explain the derivation of the leading amplitude equation in the case of a 
$3 \times 3$ strictly hyperbolic system. We keep the notation introduced in the previous paragraph, and 
we make the following assumption.

\begin{assumption}
\label{assumption6}
The phases $\varphi_1,\varphi_3$ are incoming and $\varphi_2$ is outgoing.
\end{assumption}

\noindent Assumption \ref{assumption6} corresponds to the case $p=2$ in Assumption \ref{assumption2} 
(up to reordering the phases). The only other possibility that is compatible with Assumption \ref{assumption2} 
is $p=1$, and two phases are outgoing. This case would yield the standard Burgers equation for determining 
the leading amplitude (see Paragraph \ref{paragraphAM} below for a detailed discussion), so we focus on 
$p=2$ which incorporates the main new difficulty.

Let us now derive the WKB cascade for highly oscillating solutions to \eqref{0}. The solution $u_\eps$ to 
\eqref{0} is assumed to have an asymptotic expansion of the form
\begin{equation}
\label{BKW}
u_\eps \sim \eps \, \sum_{k \ge 0} \eps^k \, \cU_k \left(t,x,\dfrac{\Phi(t,x)}{\eps} \right) \, ,
\end{equation}
where we recall that $\Phi$ denotes the collection of phases $(\varphi_1,\varphi_2,\varphi_3)$, and the profiles 
$\cU_k$ are assumed to be $\Theta$-periodic with respect to each of their last three arguments $\theta_1, 
\theta_2, \theta_3$. Plugging the ansatz \eqref{BKW} in \eqref{0} and identifying powers of $\eps$, we obtain 
the following first three relations for the $\cU_k$'s (see Section \ref{sect4} for the complete set of relations up 
to any order):
\begin{equation}
\label{3}
\begin{split}
&{\rm (a)}\quad \cL(\partial_{\theta}) \, \cU_0 =0 \, ,\\
&{\rm (b)}\quad \cL(\partial_{\theta}) \, \cU_1+L(\partial) \, \cU_0 +\cM (\cU_0,\cU_0) =0 \, ,\\
&{\rm (c)}\quad \cL(\partial_{\theta}) \, \cU_2 +L(\partial) \, \cU_1 +\cM (\cU_0,\cU_1) +\cM (\cU_1,\cU_0) 
+\cN_1 (\cU_0,\cU_0) +\cN_2 (\cU_0,\cU_0,\cU_0) =0 \, ,\\
\end{split}
\end{equation}
where the differential operators $\cL,\cM,\cN_1,\cN_2$ are defined by:
\begin{equation}
\label{3a}
\begin{split}
&\cL(\partial_\theta) := L({\rm d}\varphi_m) \, \partial_{\theta_m} \, ,\\
&\cM (v,w) := \partial_j \varphi_m \, ({\rm d}A_j(0) \cdot v) \, \partial_{\theta_m} w \, ,\\
&\cN_1 (v,w) := ({\rm d}A_j(0) \cdot v) \, \partial_j w \, ,\\
&\cN_2 (v,v,w) := \dfrac{1}{2} \, \partial_j \varphi_m \, ({\rm d}^2A_j(0) \cdot (v,v)) \, \partial_{\theta_m} w \, .\\
\end{split}
\end{equation}
The equations \eqref{3} in the domain $(-\infty,T] \times \R^d_+ \times (\R /\Theta \, \Z)^3$ are supplemented 
with the boundary conditions obtained by plugging \eqref{BKW} in the boundary conditions of \eqref{0}, which 
yields (recall $B ={\rm d}b(0)$):
\begin{equation}\label{4}
\begin{split}
&{\rm (a)}\quad B\, \cU_0 =0 \, ,\\
&{\rm (b)}\quad B\, \cU_1 +\dfrac{1}{2} \, {\rm d}^2b(0) \cdot (\cU_0,\cU_0) =G(t,y,\theta_0) \, ,\\
\end{split}
\end{equation}
where functions on the left hand side of \eqref{4} are evaluated at $x_d=0$ and $\theta_1=\theta_2=\theta_3 
=\theta_0$. In order to get $u_\eps|_{t<0}=0$, as required in \eqref{0}, we also look for profiles ${\mathcal U}_k$ 
that vanish for $t<0$.

The derivation of the leading amplitude equation is split in several steps, which we decompose below in 
order to highlight the (slight) differences in Paragraph \ref{sect5} for the case of systems with constant multiplicity.
\bigskip

\noindent \underline{Step 1: ${\mathcal U}_0$ has mean zero.}

According to Assumption \ref{assumption5}, the phases $\varphi_m$ are nonresonant. Equation \eqref{3}(a) 
thus yields the polarization condition for the leading amplitude $\cU_0$. More precisely, we expand the 
amplitude $\cU_0$ in Fourier series with respect to the $\theta_m$'s, and \eqref{3}(a) shows that only the 
characteristic modes $\Z^{3,1}$ occur in $\cU_0$. More precisely, we can write
\begin{equation}
\label{5}
\cU_0(t,x,\theta_1,\theta_2,\theta_3) =\underline{\cU}_0(t,x) +\sum_{m=1}^3 \sigma_m(t,x,\theta_m) \, r_m \, ,
\end{equation}
with unknown scalar functions $\sigma_m$ depending on a single periodic variable $\theta_m$ and of 
mean zero with respect to this variable.

Let us now consider Equation \eqref{3}(b), and integrate with respect to $\theta_1,\theta_2,\theta_3$. 
Using the expression \eqref{5} of $\cU_0$, we obtain
\begin{equation}
\label{eq:moyenne1}
L(\partial) \, \underline{\cU}_0 =0 \, ,
\end{equation}
because the quadratic term $\cM (\cU_0,\cU_0)$ has zero mean with respect to $(\theta_1,\theta_2,\theta_3)$. 
Indeed, $\cM (\cU_0,\cU_0)$ splits as the sum of terms that have one of the following three forms :
\begin{equation*}
\star) \partial_{\theta_m} \sigma_m \, f_m(t,x)\, ,\quad 
\star) \, \sigma_m \, \partial_{\theta_m} \sigma_m \, \tilde{r}_m \, ,\quad 
\star) \, \sigma_{m_1} \, \partial_{\theta_{m_2}} \sigma_{m_2} \, \tilde{r}_{m_1 m_2} \, (m_1 \neq m_2) \, ,
\end{equation*}
where $\tilde{r}_m, \tilde{r}_{m_1 m_2}$ are constant vectors (whose precise expression is useless), and each 
of these terms has zero mean with respect to $(\theta_1,\theta_2,\theta_3)$. Equation \eqref{eq:moyenne1} is 
supplemented by the boundary condition obtained by integrating \eqref{4}(a), that is,
\begin{equation}
\label{eq:moyenne2}
B\, \underline{\cU}_0 |_{x_d=0}=0 \, .
\end{equation}
By the well-posedness result of \cite{C}, we get $\underline{\cU}_0 \equiv 0$. The goal is now to identify the 
amplitudes $\sigma_m$'s in \eqref{5}.
\bigskip

\noindent \underline{Step 2: ${\mathcal U}_0$ has no outgoing mode.}

We first start by showing $\sigma_2 \equiv 0$. We first integrate \eqref{3}(b) with respect to $(\theta_1,\theta_3)$ 
and apply the row vector $\ell_2$ (which amounts to applying $Q_2$), obtaining
\begin{equation*}
\ell_2 \, L(\partial) (\sigma_2 \, r_2) +\ell_2 \, \left( \dfrac{1}{\Theta^2} \, \int_0^\Theta \! \! \! \int_0^\Theta 
\cM (\cU_0,\cU_0) \, {\rm d}\theta_1 \, {\rm d}\theta_3 \right) =0 \, .
\end{equation*}
Since there is no resonance among the phases, integration of the quadratic term $\cM (\cU_0,\cU_0)$ with 
respect to $(\theta_1,\theta_3)$ only leaves the self-interaction term $\sigma_2 \, \partial_{\theta_2} \sigma_2$, 
and the classical Lax lemma \cite{L} for the linear part\footnote{In fact, $\ell_2 \, L(\partial) (\cdot \, r_2)$ equals 
$\ell_2 \, r_2$ times the transport operator $\partial_t +{\bf v}_2 \cdot \nabla_x$, and $\ell_2 \, r_2$ does not 
vanish.} $\ell_2 \, L(\partial) (\sigma_2 \, r_2)$ gives the scalar equation
\begin{equation*}
\partial_t \sigma_2 +{\bf v}_2 \cdot \nabla_x \sigma_2 +c_2 \, \sigma_2 \, \partial_{\theta_2} \sigma_2 =0 \, ,\quad 
c_2 :=\dfrac{\partial_j \varphi_2 \, \ell_2 \, ({\rm d}A_j(0) \cdot r_2) \, r_2}{\ell_2 \, r_2} \, .
\end{equation*}
Since ${\bf v}_2$ is outgoing and $\sigma_2$ vanishes for $t<0$, we obtain $\sigma_2 \equiv 0$.

The above derivation of the interior equation for $\sigma_2$ can be performed word for word for the other 
scalar amplitudes $\sigma_1,\sigma_3$, because $\cM (\cU_0,\cU_0)$ also has zero mean with respect to 
$(\theta_1,\theta_2)$ and $(\theta_2,\theta_3)$. We thus get
\begin{equation}
\label{eq:Burgersm}
\partial_t \sigma_m +{\bf v}_m \cdot \nabla_x \sigma_m +c_m \, \sigma_m \, \partial_{\theta_m} \sigma_m =0 \, ,\quad 
m=1,3 \, ,\quad c_m :=\dfrac{\partial_j \varphi_m \, \ell_m \, ({\rm d}A_j(0) \cdot r_m) \, r_m}{\ell_m \, r_m} \, ,
\end{equation}
but we now need to determine the trace of $\sigma_m$ on the boundary $\{ x_d=0 \}$.

Since only $\sigma_1, \sigma_3$ appear in the decomposition \eqref{5}, the leading amplitude $\cU_0$ takes 
values in the stable subspace $\E^s \tauetabar$ (Lemma \ref{lem1}), and the boundary condition \eqref{4}(a) 
yields
\begin{equation*}
\sigma_1 (t,y,0,\theta_0) \, r_1 =a(t,y,\theta_0) \, e_1 \, ,\quad \sigma_3 (t,y,0,\theta_0) \, r_3 =a(t,y,\theta_0) \, e_3 \, ,
\end{equation*}
for a single unknown scalar function $a$ of zero mean with respect to its last argument $\theta_0$ 
(recall that $e=e_1+e_3$ spans the vector space $\E^s \tauetabar \cap \text{\rm Ker } B$).
\bigskip

\noindent \underline{Step 3: ${\mathcal U}_1$ has no outgoing mode.}

The derivation of the equation that governs the evolution of $a$ comes from analyzing the equations for the 
first corrector $\cU_1$. Since \eqref{3}(c) is more intricate than the corresponding equation in \cite{CGW3}, 
the analysis is starting here to differ from what we did in our previous work \cite{CGW3}. The first corrector 
$\cU_1$ reads
\begin{equation*}
\cU_1(t,x,\theta_1,\theta_2,\theta_3) =\underline{\cU}_1 (t,x) +\sum_{m=1}^3 \cU_1^m (t,x,\theta_m) 
+\cU_1^{\rm nc}(t,x,\theta_1,\theta_2,\theta_3) \, ,
\end{equation*}
where $\underline{\cU}_1$ represents the mean value with respect to $(\theta_1,\theta_2,\theta_3)$, each 
$\cU_1^m$ incorporates the $\theta_m$-oscillations and has mean zero, and the spectrum of the "noncharacteristic" 
part $\cU_1^{\rm nc}$ is a subset of $\Z^3 \setminus \Z^{3,1}$ due to the nonresonant Assumption \ref{assumption5}. 
More precisely, $\cU_1^{\rm nc}$ is obtained by expanding \eqref{3}(b) in Fourier series and retaining only the 
noncharacteristic modes $\Z^3 \setminus \Z^{3,1}$. From the expression \eqref{5} of $\cU_0$ (recall 
$\underline{\cU}_0 \equiv 0$ and $\sigma_2 \equiv 0$), we get
\begin{equation}
\label{u1nc}
\cL(\partial_\theta) \, \cU_1^{\rm nc} =-\sigma_1 \, \partial_{\theta_3} \sigma_3 \, \partial_j \varphi_3 \, 
({\rm d}A_j(0) \cdot r_1) \, r_3 -\sigma_3 \, \partial_{\theta_1} \sigma_1 \, \partial_j \varphi_1 \, 
({\rm d}A_j(0) \cdot r_3) \, r_1 \, .
\end{equation}
In particular, the spectrum of $\cU_1^{\rm nc}$ is a subset of the integers $\alpha \in \Z^3$ that satisfy $\alpha_2=0$ 
and $\alpha_1 \, \alpha_3 \neq 0$, so $\cU_1^{\rm nc}$ has zero mean when integrated with respect to 
$(\theta_1,\theta_3)$.

Equation \eqref{3}(b) also shows that the component $\cU_1^2$ that carries the $\theta_2$-oscillations 
of $\cU_1$ satisfies $L({\rm d}\varphi_2) \, \partial_{\theta_2} \cU_1^2=0$, so that $\cU_1^2$ can be 
written as $\cU_1^2 =\tau_2(t,x,\theta_2) \, r_2$ for an unknown scalar function $\tau_2$ of zero mean 
with respect to $\theta_2$.

Let us now consider Equation \eqref{3}(c). Since $\cU_0$ only has oscillations in $\theta_1$ and $\theta_3$, 
and since there is no resonance among the phases, none of the terms $\cN_1 (\cU_0,\cU_0)$, $\cN_2 
(\cU_0,\cU_0,\cU_0)$ has oscillations in $\theta_2$ only. Looking also closely at each term in $\cM (\cU_0,\cU_1), 
\cM (\cU_1,\cU_0)$, we find that both expressions have zero mean with respect to $(\theta_1,\theta_3)$, 
because the only way to generate a $\theta_2$-oscillation would be to have a nonzero mode of the form 
$(\alpha_1,\alpha_2,0)$ or $(0,\alpha_2,\alpha_3)$ with $\alpha_2 \neq 0$ in $\cU_1^{\rm nc}$, but there 
is no such mode according to \eqref{u1nc}. We thus derive the outgoing transport equation
\begin{equation*}
\partial_t \tau_2 +{\bf v}_2 \cdot \nabla_x \tau_2 =0 \, ,
\end{equation*}
from which we get $\tau_2 \equiv 0$.
\bigskip

\noindent \underline{Step 4:  computation of the nonpolarized components of ${\mathcal U}_1^1, 
{\mathcal U}_1^3$, and compatibility condition.}

At this stage, we know that the first corrector $\cU_1$ reads
\begin{equation*}
\cU_1 =\cU_1(t,x,\theta_1,\theta_3) =\underline{\cU}_1 (t,x) +\cU_1^1 (t,x,\theta_1) +\cU_1^3 (t,x,\theta_3) 
+\cU_1^{\rm nc}(t,x,\theta_1,\theta_3) \, ,
\end{equation*}
with $\cU_1^{\rm nc}$ determined by \eqref{u1nc}. Moreover, the nonpolarized part of $\cU_1^{1,3}$ is 
obtained by considering Equation \eqref{3}(b) and retaining only the $\theta_{1,3}$ Fourier modes. We get
\begin{equation*}
L({\rm d}\varphi_m) \, \partial_{\theta_m} \cU_1^m =-L(\partial) \, (\sigma_m \, r_m) 
-\sigma_m \, \partial_{\theta_m} \sigma_m \, \partial_j \varphi_m \, ({\rm d}A_j(0) \cdot r_m) \, r_m \, ,\quad 
m=1,3 \, ,
\end{equation*}
so $(I-P_m) \, \cU_1^m$, $m=1,3$, is the only zero mean function that satisfies
\begin{equation}
\label{u113}
(I-P_m) \, \partial_{\theta_m} \cU_1^m =-R_m \, L(\partial) \, (\sigma_m \, r_m) 
-\sigma_m \, \partial_{\theta_m} \sigma_m \, \partial_j \varphi_m \, R_m \, ({\rm d}A_j(0) \cdot r_m) \, r_m \, ,\quad 
m=1,3 \, .
\end{equation}

We now consider the boundary condition \eqref{4}(b), which we rewrite equivalently as:
\begin{align*}
B \, \underline{\cU}_1|_{x_d=0} 
+B \, P_1 \, \cU_1^1|_{x_d=0,\theta_1=\theta_0} +B \, P_3 \, \cU_1^3|_{x_d=0,\theta_3=\theta_0} & \\
+B\, (I-P_1) \, \cU_1^1|_{x_d=0,\theta_1=\theta_0} +B \, (I-P_3) \, \cU_1^3|_{x_d=0,\theta_3=\theta_0} 
+B \, \cU_1^{\rm nc}|_{x_d=0,\theta_1=\theta_3=\theta_0} & \\
+\dfrac{1}{2} \, ({\rm d}^2b(0) \cdot (e,e)) \, a^2 &=G(t,y,\theta_0) \, .
\end{align*}
We differentiate the latter equation with respect to $\theta_0$ and apply the row vector $\underline{b}$, so 
that the first line vanishes. We are left with
\begin{align*}
\underline{b} \, B\, (I-P_1) \, (\partial_{\theta_1} \cU_1^1)|_{x_d=0,\theta_1=\theta_0} 
+\underline{b} \, B \, (I-P_3) \, (\partial_{\theta_3} \cU_1^3)|_{x_d=0,\theta_3=\theta_0} 
+\underline{b} \, B \, \partial_{\theta_0} (\cU_1^{\rm nc}|_{x_d=0,\theta_1=\theta_3=\theta_0}) & \\
+\dfrac{1}{2} \, \underline{b} \, ({\rm d}^2 b(0) \cdot (e,e)) \, \partial_{\theta_0} (a^2) & 
=\underline{b} \, \partial_{\theta_0} G \, .
\end{align*}
The first two terms in the first row are computed by using \eqref{u113}, and \cite[Proposition 3.5]{CG}. 
We get
\begin{equation}
\label{eqa1}
\upsilon \, \partial_{\theta_0} (a^2) -X_{\rm Lop} a 
+\underline{b} \, B \, \partial_{\theta_0} (\cU_1^{\rm nc}|_{x_d=0,\theta_1=\theta_3=\theta_0}) =\underline{b} \, 
\partial_{\theta_0} G \, ,
\end{equation}
where the constant $\upsilon$ and the vector field $X_{\rm Lop}$ are defined by
\begin{align}
\upsilon &:= \dfrac{1}{2} \, \underline{b} \, ({\rm d}^2 b(0) \cdot (e,e)) 
-\dfrac{1}{2} \, \underline{b} \, B \, R_1 \, \partial_j \varphi_1 \, ({\rm d}A_j(0) \cdot e_1) \, e_1 
-\dfrac{1}{2} \, \underline{b} \, B \, R_3 \, \partial_j \varphi_3 \, ({\rm d}A_j(0) \cdot e_3) \, e_3 \, ,\label{defupsilon}\\
X_{\rm Lop} &:=\underline{b} \, B \, (R_1 \, e_1 +R_3\, e_3) \, \partial_t 
+\sum_{j=1}^{d-1} \underline{b} \, B \, (R_1 \, A_j(0) \, e_1 +R_3\, A_j(0) \, e_3) \, \partial_j 
=\iota \, (\partial_\tau \sigma \tauetabar \, \partial_t +\partial_{\eta_j} \sigma \tauetabar \, 
\partial_j) \, ,\label{defXLop}
\end{align}
with $\iota$ a nonzero real constant, and $\sigma$ defined in Assumption \ref{assumption3}. It is also shown 
in \cite[Proposition 3.5]{CG} that the partial derivative $\partial_\tau \sigma \tauetabar$ does not vanish, so 
that, up to a nonzero constant, $X_{\rm Lop}=\partial_t +{\bf w} \cdot \nabla_y$ for some vector ${\bf w} \in 
\R^{d-1}$ (which represents the group velocity associated with the characteristic set of the Lopatinskii 
determinant).
\bigskip

\noindent \underline{Step 5: computation of ${\mathcal U}_1^{\rm nc}$ and conclusion.}

The final step in the analysis is to compute the derivative $\partial_{\theta_0} (\cU_1^{\rm nc} 
|_{x_d=0,\theta_1=\theta_3=\theta_0})$ arising in \eqref{eqa1} in terms of the amplitude $a$. 
Restricting \eqref{u1nc} to the boundary $\{ x_d=0\}$ gives
\begin{align*}
\cL(\partial_\theta) \, \cU_1^{\rm nc}|_{x_d=0} =&-a (t,y,\theta_1) \, (\partial_{\theta_0} a) (t,y,\theta_3) 
\, \partial_j \varphi_3 \, ({\rm d}A_j(0) \cdot e_1) \, e_3 \\
&-a (t,y,\theta_3) \, (\partial_{\theta_0} a) (t,y,\theta_1) \, \partial_j \varphi_1 \, ({\rm d}A_j(0) \cdot e_3) \, e_1 \, .
\end{align*}
Let us expand $a$ in Fourier series with respect to $\theta_0$ (recall that $a$ has mean zero):
\begin{equation*}
a(t,y,\theta_0) =\sum_{k \in \Z^*} a_k(t,y) \, {\rm e}^{2 \, i \, \pi \, k \, \theta_0/\Theta} \, .
\end{equation*}
Then the Fourier series of $\cU_1^{\rm nc}$ reads
\begin{equation*}
\cU_1^{\rm nc} (t,x,\theta_1,\theta_3) =\sum_{k_1,k_3 \in \Z^*} u_{k_1,k_3}(t,x) \, 
{\rm e}^{2 \, i \, \pi \, (k_1 \, \theta_1+k_3 \, \theta_3)/\Theta} \, ,
\end{equation*}
with
\begin{align}
u_{k_1,k_3}(t,y,0) =-L(k_1 \, {\rm d}\varphi_1+k_3 \, {\rm d}\varphi_3)^{-1} \, 
(k_1 \, E_{1,3}+k_3 \, E_{3,1}) \, ,\notag \\
E_{1,3} :=\partial_j \varphi_1 \, ({\rm d}A_j(0) \cdot e_3) \, e_1 \, ,\quad 
E_{3,1} :=\partial_j \varphi_3 \, ({\rm d}A_j(0) \cdot e_1) \, e_3 \, .\label{defE1E3}
\end{align}
Plugging the latter expression in \eqref{eqa1}, we end up with the evolution equation that governs 
the leading amplitude $a$ on the boundary:
\begin{equation}
\label{eqa2}
\upsilon \, \partial_{\theta_0} (a^2) -X_{\rm Lop} a +\partial_{\theta_0} Q_{\rm per}[a,a] =\underline{b} \, \partial_{\theta_0} G \, ,
\end{equation}
with
\begin{equation*}
Q_{\rm per}[a,\widetilde{a}] :=-\sum_{k \in \Z} \left( \sum_{\underset{k_1 \, k_3 \neq 0}{k_1+k_3=k,}} \underline{b} \, B \, 
L(k_1 \, {\rm d}\varphi_1+k_3 \, {\rm d}\varphi_3)^{-1} \, (k_1 \, E_{1,3}+k_3 \, E_{3,1}) \, a_{k_1} \, \widetilde{a}_{k_3} 
\right) \, {\rm e}^{2 \, i \, \pi \, k \, \theta_0/\Theta} \, .
\end{equation*}

Equation \eqref{eqa2} is a closed equation for the leading amplitude $a$ on the boundary. It involves the 
vector field $X_{\rm Lop}$ associated with a characteristic of the Lopatinskii determinant, a Burgers term 
$\partial_{\theta_0} a^2$ and a new quadratic nonlinearity $\partial_{\theta_0} Q_{\rm per}[a,a]$. The operator 
$Q_{\rm per}$ takes the form of a bilinear Fourier multiplier. Its above expression may be simplified a little 
bit by computing the decomposition of $L(k_1 \, {\rm d}\varphi_1+k_3 \, {\rm d}\varphi_3)^{-1}$ on the basis 
$r_1,r_2,r_3$, and by recalling the property $\underline{b} \, B \, r_1=\underline{b} \, B \, r_3=0$ (so only 
the component of $L(k_1 \, {\rm d}\varphi_1+k_3 \, {\rm d}\varphi_3)^{-1}$ on $r_2$ matters). We obtain:
\begin{equation}
\label{defQper}
Q_{\rm per}[a,\widetilde{a}] :=-\underline{b} \, B \, r_2 \, \sum_{k \in \Z} \left( 
\sum_{\underset{k_1 \, k_3 \neq 0}{k_1+k_3=k,}} 
\dfrac{k_1 \, \ell_2 \, E_{1,3}+k_3 \, \ell_2 \, E_{3,1}}{k_1 \, (\underline{\omega}_1-\underline{\omega}_2) 
+k_3 \, (\underline{\omega}_3-\underline{\omega}_2)} \, a_{k_1} \, \widetilde{a}_{k_3} \right) \, 
{\rm e}^{2 \, i \, \pi \, k \, \theta_0/\Theta} \, .
\end{equation}

Anticipating slightly our discussion in Section \ref{sect3}, well-posedness of \eqref{eqa2} will be linked to 
arithmetic properties of the phases $\varphi_m$, and this is one reason for the small divisor condition in 
Assumption \ref{assumption5}. This is in sharp contrast with the theory of {\sl weakly nonlinear} geometric 
optics for both the Cauchy problem (see \cite{HMR,JMR2,JMR1} and references therein) and for uniformly 
stable boundary value problems (see \cite{williams4,W1,CGW1}), where arithmetic properties of the phases 
do not enter the discussion on the leading profile for the high frequency limit.

\subsection{Extension to strictly hyperbolic systems of size $N$}

The above derivation of Equation \eqref{eqa2} can be extended without any difficulty to the case of a 
hyperbolic system of size $N$ provided that Assumption \ref{assumption5} is satisfied. In that case, the 
number of phases equals $N$.

Steps 1 and 2 of the above analysis extend almost word for word, to the price of changing some notation. 
Namely, the first relations of the WKB cascade \eqref{3}, \eqref{4} shows that the leading amplitude $\cU_0$ 
reads
\begin{equation*}
\cU_0(t,x,\theta_1,\dots,\theta_N) =\underline{\cU}_0(t,x) +\sum_{m=1}^N \sigma_m(t,x,\theta_m) \, r_m \, ,
\end{equation*}
with unknown scalar functions $\sigma_m$ depending on a single periodic variable $\theta_m$ and of 
mean zero with respect to this variable. The quadratic expression $\cM (\cU_0,\cU_0)$ still has zero 
mean with respect to $(\theta_1,\dots,\theta_N)$ so the nonoscillating part $\underline{\cU}_0$ satisfies 
\eqref{eq:moyenne1} and \eqref{eq:moyenne2}, and therefore vanishes. Furthermore, each function 
$\sigma_m$ satisfies the Burgers equation \eqref{eq:Burgersm}, which reduces the leading amplitude 
$\cU_0$ to
\begin{equation}
\label{decompU0}
\cU_0(t,x,\theta_1,\dots,\theta_N) =\sum_{m \in {\mathcal I}} \sigma_m(t,x,\theta_m) \, r_m \, ,
\end{equation}
where ${\mathcal I}$ denotes the set of incoming phases. The boundary condition \eqref{4}(a) 
then yields
\begin{equation*}
\forall \, m \in {\mathcal I} \, ,\quad \sigma_m (t,y,0,\theta_0) \, r_m =a(t,y,\theta_0) \, e_m \, ,
\end{equation*}
for a single unknown scalar function $a$ of zero mean with respect to its last argument $\theta_0$ 
($e=\sum_{m \in {\mathcal I}} e_m$ spans the vector space $\E^s \tauetabar \cap \text{\rm Ker } B$).

Step 3 of the above discussion is unchanged, showing that in the first corrector $\cU_1$, each profile 
$\cU_1^m$ vanishes when the index $m$ corresponds to an outgoing phase. The noncharacteristic part 
$\cU_1^{\rm nc}$ is obtained by using the relation
\begin{equation*}
\cL(\partial_\theta) \, \cU_1^{\rm nc} =-\sum_{\underset{m_1, m_2 \in {\mathcal I}}{m_1<m_2}} \sigma_{m_1} 
\, \partial_{\theta_{m_2}} \sigma_{m_2} \, \partial_j \varphi_{m_2} \, 
({\rm d}A_j(0) \cdot r_{m_1}) \, r_{m_2} +\sigma_{m_2} \, \partial_{\theta_{m_1}} \sigma_{m_1} \, 
\partial_j \varphi_{m_1} \, ({\rm d}A_j(0) \cdot r_{m_2}) \, r_{m_1} \, .
\end{equation*}
which is the analogue of \eqref{u1nc}.

Step 4 is also unchanged because $\cU_1$ has no outgoing mode, and when $m$ corresponds to an incoming 
phase, $(I-P_m) \, \partial_{\theta_m} \cU_1^m$ is given by \eqref{u113}. Eventually, the boundary condition 
\eqref{4}(b) gives the compatibility condition
\begin{equation}
\label{eqa3}
\upsilon \, \partial_{\theta_0} (a^2) -X_{\rm Lop} a +\partial_{\theta_0} Q_{\rm per}[a,a] 
=\underline{b} \, \partial_{\theta_0} G \, ,
\end{equation}
with
\begin{align}
\upsilon &:= \dfrac{1}{2} \, \underline{b} \, ({\rm d}^2 b(0) \cdot (e,e)) -\dfrac{1}{2} \, 
\sum_{m \in {\mathcal I}} \underline{b} \, B \, R_m \, \partial_j \varphi_m \, ({\rm d}A_j(0) \cdot e_m) \, e_m \, , 
\label{defupsilonstrict} \\
X_{\rm Lop} &:=\sum_{m \in {\mathcal I}} \underline{b} \, B \, R_m\, e_m \, \partial_t 
+\sum_{j=1}^{d-1} \sum_{m \in {\mathcal I}} \underline{b} \, B \, R_m \, A_j(0) \, e_m \, \partial_j 
=\iota \, (\partial_\tau \sigma \tauetabar \, \partial_t +\partial_{\eta_j} \sigma \tauetabar \, \partial_j) \, ,\notag
\end{align}
where $\iota$ is a nonzero real constant and the function $\sigma$ is defined in Assumption \ref{assumption3}. 
(Again, \cite[Proposition 3.5]{CG} shows that the partial derivative $\partial_\tau \sigma \tauetabar$ does not 
vanish.) The new expression of the bilinear Fourier multiplier $Q_{\rm per}$ reads:
\begin{multline}
\label{defQper'}
Q_{\rm per}[a,\widetilde{a}] :=-\sum_{m \in {\mathcal O}} \underline{b} \, B \, r_m \, 
\sum_{\underset{m_1, m_2 \in {\mathcal I}}{m_1<m_2}} \\
\sum_{k \in \Z} \left( \sum_{\underset{k_{m_1} \, k_{m_2} \neq 0}{k_{m_1}+k_{m_2}=k,}} 
\dfrac{k_{m_1} \, \ell_m \, E_{m_1,m_2}+k_{m_2} \, \ell_m \, E_{m_2,m_1}}{k_{m_1} \, 
(\underline{\omega}_{m_1}-\underline{\omega}_m) 
+k_{m_2} \, (\underline{\omega}_{m_2}-\underline{\omega}_m)} \, a_{k_{m_1}} \, \widetilde{a}_{k_{m_2}} \right) 
\, {\rm e}^{2 \, i \, \pi \, k \, \theta_0/\Theta} \, ,
\end{multline}
with
\begin{equation}
\label{defEm1m2}
\forall \, m_1,m_2 \in {\mathcal I}\, ,\quad 
E_{m_1,m_2} :=\partial_j \varphi_{m_1} \, ({\rm d}A_j(0) \cdot e_{m_2}) \, e_{m_1} \, .
\end{equation}
The definition \eqref{defQper'} reduces to \eqref{defQper} when $N=3$ and Assumption \ref{assumption6} is satisfied.

\subsection{The case with a single incoming phase}
\label{paragraphAM}

In this short paragraph, we explain why the computations in \cite{AM,wangyu} lead to the standard Burgers 
equation for determining the leading amplitude, and does not incorporate any quadratic nonlinearity of the 
form \eqref{defQper} we have found under Assumption \ref{assumption6}.

The vortex sheets problem considered in \cite{AM} and the analogous one in \cite{wangyu} differ from the 
framework that we consider here by the fact that the (free) boundary is characteristic. Nevertheless, one 
can reproduce a similar normal modes analysis for trying to detect violent or neutral instabilities. The 
{\sl two-dimensional supersonic} regime considered in \cite{AM} precludes violent instabilities, but a similar 
situation to the one encoded in Assumption \ref{assumption3} occurs\footnote{The reader will find in \cite{CS} 
a detailed analysis of the roots of the associated Lopatinskii determinant, showing that they are located in the 
hyperbolic region and simple.}. The situation in \cite{wangyu} for three dimensional steady flows is similar, 
and the corresponding Lopatinskii determinant is computed in \cite{wangyuJDE}.

The two-dimensional Euler equations form a system of three equations ($N=3$), but due to the characteristic 
boundary (the corresponding matrix $A_d(0)$ has a kernel of dimension $1$), the number of phases 
$\varphi_m$ on either side of the vortex sheet equals $2$. One of them is incoming, and the other is 
outgoing. In such a situation, there are too few incoming phases to create a nontrivial component 
$\cU_1^{\rm nc}$ for the first corrector $\cU_1$, so that the bilinear Fourier multiplier $Q_{\rm per}$ 
vanishes. Though our argument is somehow formal, the reader can follow the computations in \cite{AM} 
or in \cite{wangyu} and check that they follow the exact same procedure that we have described in our 
general framework.

\section{Analysis of the leading amplitude equation}
\label{sect3}

Our goal in this section is to prove a well-posedness result for the leading amplitude equation \eqref{eqa2}. 
Up to dividing by nonzero constants, and using the shorter notation $\theta$ instead of $\theta_0$, the equation 
takes the form
\begin{equation}
\label{eqampli1}
\partial_t a +{\bf w} \cdot \nabla_y a +c \, a \, \partial_\theta a +\mu \, \partial_\theta Q_{\rm per}[a,a] =g \, ,
\end{equation}
where ${\bf w}$ is a fixed vector in $\R^{d-1}$, $c,\mu$ are real constants, and $Q_{\rm per}$ is a bilinear Fourier 
multiplier with respect to the periodic variable $\theta$:
\begin{equation}
\label{defQper''}
Q_{\rm per}[a,a] :=\sum_{k \in \Z} \left( \sum_{\underset{k_1 \, k_3 \neq 0}{k_1+k_3=k,}} 
\dfrac{k_1 \, \ell_2 \, E_{1,3}+k_3 \, \ell_2 \, E_{3,1}}{k_1 \, (\underline{\omega}_1-\underline{\omega}_2) 
+k_3 \, (\underline{\omega}_3-\underline{\omega}_2)} \, a_{k_1} \, a_{k_3} \right) \, 
{\rm e}^{2 \, i \, \pi \, k \, \theta/\Theta} \, .
\end{equation}
The source term $g$ in \eqref{eqampli1} belongs to $H^{+\infty} ((-\infty,T_0]_t \times \R^{d-1}_y \times 
(\R /(\Theta \, \Z))_\theta)$, $T_0>0$, and vanishes for $t<0$. Furthermore, it has mean zero with respect 
to the variable $\theta$. Recall that in \eqref{defQper''}, $a_k$ denotes the $k$-th Fourier coefficient of $a$ 
with respect to $\theta$ (which is a function of $(t,y)$).

Recall that for strictly hyperbolic systems of size $N$, \eqref{defQper'} should be substituted for \eqref{defQper''} 
in the definition of $Q_{\rm per}$, while in the particular case $p=1$ (one single incoming phase), \eqref{eqampli1} 
reduces to the standard Burgers equation for which our main well-posedness result, Theorem \ref{thmKinkModes} 
below, is well-known. For simplicity, we thus encompass all cases by studying \eqref{eqampli1}, \eqref{defQper''} 
and leave to the reader the very minor modifications required for the general case.

\subsection{Preliminary reductions}

We first introduce the nonzero parameters:
\begin{equation*}
\delta_1 :=\dfrac{\underline{\omega}_1-\underline{\omega}_2}{\underline{\omega}_3-\underline{\omega}_1} \, ,\quad 
\delta_3 :=\dfrac{\underline{\omega}_3-\underline{\omega}_2}{\underline{\omega}_3-\underline{\omega}_1} \, ,
\end{equation*}
that satisfy $\delta_3=1+\delta_1$, and we observe that $Q_{\rm per}[a,a]$ in \eqref{defQper''} can be written as
\begin{equation*}
Q_{\rm per}[a,a]=-\dfrac{\Theta \, \ell_2 \, E_{1,3}}{2\, \pi \, (\underline{\omega}_3-\underline{\omega}_1)} \, 
\F_{\rm per}(\partial_\theta a,a) 
-\dfrac{\Theta \, \ell_2 \, E_{3,1}}{2\, \pi \, (\underline{\omega}_3-\underline{\omega}_1)} \, 
\F_{\rm per}(a,\partial_\theta a) \, ,
\end{equation*}
where the bilinear operator $\F_{\rm per}$ is defined by:
\begin{equation}
\label{defFper}
\F_{\rm per}(u,v) :=\sum_{k \in \Z} \left( \sum_{\underset{k_1 \, k_3 \neq 0}{k_1+k_3=k,}} 
\dfrac{i \, u_{k_1} \, v_{k_3}}{k_1 \, \delta_1 +k_3 \, \delta_3} \right) \, {\rm e}^{2 \, i \, \pi \, k \, \theta/\Theta} \, .
\end{equation}

The bilinear operator $\F_{\rm per}$ satisfies the following two properties:
\begin{align}
\text{\rm (Differentiation)} \quad &\partial_\theta (\F_{\rm per}(u,v)) =\F_{\rm per}(\partial_\theta u,v) 
+\F_{\rm per}(u,\partial_\theta v) \, ,\label{propFper1} \\
\text{\rm (Integration by parts)} \quad &\F_{\rm per}(u,\partial_\theta v) =-\dfrac{2\, \pi}{\Theta \, \delta_3} \, u\, v 
-\dfrac{\delta_1}{\delta_3} \, \F_{\rm per}(\partial_\theta u,v) \, ,\quad \text{\rm if } u_0=v_0=0 \, .\label{propFper2}
\end{align}
Using the properties \eqref{propFper1}, \eqref{propFper2}, we can rewrite equation \eqref{eqampli1} as
\begin{equation}
\label{eqampli2}
\partial_t a +{\bf w} \cdot \nabla_y a +c \, a \, \partial_\theta a +\mu \, \F_{\rm per}(\partial_\theta a,\partial_\theta a) 
=g \, ,
\end{equation}
with new (harmless) constants $c$, $\mu$ for which we keep the same notation. Our goal is to solve 
equation \eqref{eqampli2}, that is equivalent to \eqref{eqampli1}, by a standard fixed point argument. The 
main ingredient in the proof is to show that the nonlinear term $\F_{\rm per} (\partial_\theta a, \partial_\theta a)$ 
acts as a {\sl semilinear} term in the scale of Sobolev spaces.

\subsection{Tame boundedness of the bilinear operator $\F_{\rm per}$}

The operator $\F_{\rm per}$ is not symmetric but changing the roles of $\delta_1$ and $\delta_3$, the roles 
of the first and second argument of $\F_{\rm per}$ in the estimates below can be exchanged. This will be used 
at one point in the analysis below. We let $H^\nu :=H^\nu(\R^{d-1} \times (\R/\Theta \, \Z))$ denote the standard 
Sobolev space of index $\nu \in \N$. The norm is denoted $\| \cdot \|_{H^\nu}$. Functions are assumed to take 
real values. In the proof of Theorem \ref{propFper} below, we shall also make use of fractional Sobolev 
spaces on the torus $\R/\Theta \, \Z$ or on the whole space $\R^{d-1}$. These are defined by means of 
the Fourier transform, see, e.g., \cite{BS,chemin}. Our main boundedness result for the operator $\F_{\rm per}$ 
reads as follows.

\begin{theo}
\label{propFper}
There exists an integer $\nu_0>1+d/2$ such that, for all $\nu \ge \nu_0$, there exists a constant $C_\nu$ satisfying
\begin{equation}
\label{estimpropFper}
\forall \, u,v \in H^\nu \, ,\quad \| \F_{\rm per}(\partial_\theta u,\partial_\theta v) \|_{H^\nu} \le 
C_\nu \, \big( \| u \|_{H^{\nu_0}} \, \| v \|_{H^\nu} +\| u \|_{H^\nu} \, \| v \|_{H^{\nu_0}} \big) \, .
\end{equation}
\end{theo}

Estimate \eqref{estimpropFper} is tame because the integer $\nu_0$ is fixed and the right hand side of the 
inequality depends linearly on the norms $\| u \|_{H^\nu},\| v \|_{H^\nu}$. This will be used in the proof of 
Theorem \ref{thmKinkModes} below for propagating the regularity of the initial condition for \eqref{eqampli2} 
on a fixed time interval.

\begin{proof}
We first observe that, provided that $\F_{\rm per}(u,v)$ makes sense, then $\F_{\rm per}(u,v)$ takes real values. 
This simply follows from observing that
\begin{equation*}
(\F_{\rm per}(u,v))_{-k} =\overline{(\F_{\rm per}(u,v))_k} \, ,
\end{equation*}
provided that $u$ and $v$ take real values. We are now going to prove a convenient new formulation of Assumption 
\ref{assumption5}.

\begin{lem}
\label{lem3}
There exist a constant $C>0$ and a real number $\gamma_0 \ge 0$ such that
\begin{equation*}
\forall \, k_1,k_3 \in \Z \setminus \{ 0 \} \, ,\quad \dfrac{1}{|k_1 \, \delta_1 +k_3 \, \delta_3|} \le 
C \, \min (|k_1|^{\gamma_0},|k_3|^{\gamma_0}) \, .
\end{equation*}
\end{lem}

\begin{proof}[Proof of Lemma \ref{lem3}]
Since in our framework, we have ${\mathcal I}=\{ 1,3\}$ and ${\mathcal O} =\{ 2\}$, we can apply Assumption 
\ref{assumption5} to any $(k_1,0,k_3) \in \Z^3$ with $k_1 \, k_3 \neq 0$. We compute
\begin{equation*}
L({\rm d} (k_1 \, \varphi_1 +k_3 \, \varphi_3)) \, r_2 =(k_1 \, (\underline{\omega}_1 -\underline{\omega}_2) 
+k_3 \, (\underline{\omega}_3-\underline{\omega}_2)) \, A_d(0) \, r_2 \, ,
\end{equation*}
and the quantity $k_1 \, (\underline{\omega}_1 -\underline{\omega}_2) +k_3 \, (\underline{\omega}_3-\underline{\omega}_2)$ 
cannot vanish for otherwise there would be a nonzero vector in the kernel of $L({\rm d} (k_1 \, \varphi_1 +k_3 \, \varphi_3))$. 
We thus derive the bound
\begin{equation*}
\dfrac{1}{|k_1 \, (\underline{\omega}_1 -\underline{\omega}_2) +k_3 \, (\underline{\omega}_3-\underline{\omega}_2)|} 
\le C \, \| L({\rm d} (k_1 \, \varphi_1 +k_3 \, \varphi_3))^{-1} \| \, ,
\end{equation*}
for a suitable constant $C$ that does not depend on $k_1,k_3$. The norm of $L({\rm d} (k_1 \, \varphi_1 +k_3 \, 
\varphi_3))^{-1}$ is estimated by combining the lower bound given in Assumption \ref{assumption5} for the 
determinant, and an obvious polynomial bound for the transpose of the comatrix. We have thus shown that 
there exists a constant $C>0$ and a real parameter $\gamma_0$ (which can be chosen nonnegative without 
loss of generality), that do not depend on $k_1,k_3$, such that
\begin{equation*}
\dfrac{1}{|k_1 \, (\underline{\omega}_1 -\underline{\omega}_2) +k_3 \, (\underline{\omega}_3-\underline{\omega}_2)|} 
\le C \, |(k_1,k_3)|^{\gamma_0} \, .
\end{equation*}
Up to changing the constant $C$, we can rephrase this estimate in terms of the rescaled parameters 
$\delta_1,\delta_3$:
\begin{equation}
\label{petitsdiv}
\dfrac{1}{|k_1 \, \delta_1 +k_3 \, \delta_3|} \le C \, |(k_1,k_3)|^{\gamma_0} \, ,
\end{equation}
and it only remains to substitute the minimum of $|k_1|,|k_3|$ for the norm $|(k_1,k_3)|$ in the right hand side 
of \eqref{petitsdiv}.

There are two cases. Either $|k_1 \, \delta_1 +k_3 \, \delta_3|>|\delta_1|>0$, and in that case, it is sufficient 
to choose $C \ge 1/|\delta_1|$ (we use $\gamma_0 \ge 0$). Or $|k_1 \, \delta_1 +k_3 \, \delta_3| \le |\delta_1|$, 
and we have
\begin{equation*}
|k_3| \le \dfrac{1}{|\delta_3|} \, |k_1 \, \delta_1 +k_3 \, \delta_3| +\dfrac{1}{|\delta_3|} \, |k_1 \, \delta_1| 
\le 2 \, \dfrac{|\delta_1|}{|\delta_3|} \, |k_1| \, ,
\end{equation*}
because $k_1$ is nonzero. Up to choosing a new constant $C$, \eqref{petitsdiv} reduces to
\begin{equation*}
\dfrac{1}{|k_1 \, \delta_1 +k_3 \, \delta_3|} \le C \, |k_1|^{\gamma_0} \, ,
\end{equation*}
and we can prove the analogous estimate with $k_3$ instead of $k_1$ by the same arguments. This completes 
the proof of Lemma \ref{lem3}
\end{proof}

The proof of Theorem \ref{propFper} relies on the following straightforward extension of \cite[Lemma 1.2.2]{rauchreed}. 
The proof of Lemma \ref{lemRR} is exactly the same as that of \cite[Lemma 1.2.2]{rauchreed}, and is therefore omitted.

\begin{lem}
\label{lemRR}
Let ${\mathbb K} : \R^{d-1} \times \Z \times \R^{d-1} \times \Z \rightarrow \C$ be a locally integrable measurable 
function such that, either
\begin{equation*}
\sup_{(\xi,k) \in \R^{d-1} \times \Z} \, \int_{\R^{d-1}} \sum_{\ell \in \Z} \, |{\mathbb K}(\xi,k,\eta,\ell)|^2 \, {\rm d}\eta 
<+\infty \, ,
\end{equation*}
or
\begin{equation*}
\sup_{(\eta,\ell) \in \R^{d-1} \times \Z} \, \int_{\R^{d-1}} \sum_{\ell \in \Z} \, |{\mathbb K}(\xi,k,\eta,\ell)|^2 \, {\rm d}\xi 
<+\infty \, .
\end{equation*}
Then the map
\begin{equation*}
(f,g) \longmapsto \int_{\R^{d-1}} \sum_{\ell \in \Z} \, {\mathbb K}(\xi,k,\eta,\ell) \, f(\xi-\eta,k-\ell) \, g(\eta,\ell) 
\, {\rm d}\eta \, ,
\end{equation*}
is bounded on $L^2(\R^{d-1} \times \Z) \times L^2(\R^{d-1} \times \Z)$ with values in $L^2(\R^{d-1} \times \Z)$.
\end{lem}

To prove boundedness of the bilinear operator $\F_{\rm per}(\partial_\theta \cdot,\partial_\theta \cdot)$, we shall 
apply Lemma \ref{lemRR} in the Fourier variables. More precisely, for functions $u,v$ in the Schwartz space 
${\mathcal S}(\R^{d-1} \times (\R/\Theta \, \Z))$, there holds\footnote{Here we use the notation $c_k$ for the 
$k$-th Fourier coefficient with respect to the variable $\theta$, and the "hat" notation for the Fourier transform 
with respect to the variable $y$.}:
\begin{equation*}
\widehat{c_k(\F_{\rm per}(\partial_\theta u,\partial_\theta v))} (\xi) ={\rm C}^{\rm st} \, 
\int_{\R^{d-1}} \sum_{\ell \in \Z, \ell \not \in \{ 0,k\}} \dfrac{(k-\ell) \, \ell}{(k-\ell) \, \delta_1 +\ell \, \delta_3} \, 
\widehat{c_{k-\ell}(u)}(\xi-\eta) \, \widehat{c_\ell(v)}(\eta) \, {\rm d}\eta \, .
\end{equation*}
Omitting from now on the constant multiplicative factor, we consider the symbol
\begin{equation*}
{\bf K} (k,\ell) :=\begin{cases}
(k-\ell) \, \ell/((k-\ell) \, \delta_1 +\ell \, \delta_3) &\text{\rm if } \ell \not \in \{ 0,k\} \, ,\\
0 &\text{\rm otherwise.}
\end{cases}
\end{equation*}
We wish to bound the $H^\nu$ norm:
\begin{equation*}
\int_{\R^{d-1}} \sum_{k \in \Z} \, \langle (\xi,k) \rangle^{2\, \nu} \, 
|\widehat{c_k(\F_{\rm per}(\partial_\theta u,\partial_\theta v))} (\xi) |^2 \, {\rm d}\xi \, ,
\end{equation*}
for $\nu \in \N$ large enough ($\langle \cdot \rangle$ stands as usual for the Japanese bracket).

Given the parameter $\gamma_0 \ge 0$ in Lemma \ref{lem3}, we fix an integer $\nu_0>\gamma_0+2+d/2$. 
We consider two functions $\chi_1,\chi_2$ on $\R^d \times \R^d$ such that $\chi_1+\chi_2 \equiv 1$, and
\begin{align*}
&\chi_1(\xi,k,\eta,\ell) =0 \quad \text{\rm if } \langle (\eta,\ell) \rangle \ge (2/3) \, \langle (\xi,k) \rangle \, ,\\
&\chi_2(\xi,k,\eta,\ell) =0 \quad \text{\rm if } \langle (\eta,\ell) \rangle \le (1/3) \, \langle (\xi,k) \rangle \, .
\end{align*}
We first consider the quantity
\begin{equation}
\label{tameterm1}
\int_{\R^{d-1}} \sum_{\ell \in \Z} \, \chi_1(\xi,k,\eta,\ell) \, \langle (\xi,k) \rangle^\nu \, {\bf K} (k,\ell) \, 
\widehat{c_{k-\ell}(u)}(\xi-\eta) \, \widehat{c_\ell(v)}(\eta) \, {\rm d}\eta \, ,
\end{equation}
which we rewrite as
\begin{equation*}
\int_{\R^{d-1}} \sum_{\ell \in \Z} \, 
\dfrac{\chi_1(\xi,k,\eta,\ell) \, \langle (\xi,k) \rangle^\nu \, {\bf K} (k,\ell)}{\langle (\xi-\eta,k-\ell) \rangle^\nu \, 
\langle (\eta,\ell) \rangle^{\nu_0}} \, \Big( \langle (\xi-\eta,k-\ell) \rangle^\nu \, \widehat{c_{k-\ell}(u)}(\xi-\eta) \Big) 
\, \Big( \langle (\eta,\ell) \rangle^{\nu_0} \, \widehat{c_\ell(v)}(\eta)\Big) \, {\rm d}\eta \, .
\end{equation*}
On the support of $\chi_1$, there holds
\begin{equation*}
\langle (\xi-\eta,k-\ell) \rangle \ge \langle (\xi,k) \rangle -\langle (\eta,\ell) \rangle \ge 
\dfrac{1}{3} \, \langle (\xi,k) \rangle \, ,
\end{equation*}
and therefore
\begin{equation*}
\int_{\R^{d-1}} \sum_{\ell \in \Z} \, \left| \dfrac{\chi_1(\xi,k,\eta,\ell) \, \langle (\xi,k) \rangle^\nu \, {\bf K} (k,\ell)}
{\langle (\xi-\eta,k-\ell) \rangle^\nu \, \langle (\eta,\ell) \rangle^{\nu_0}} \right|^2 \, {\rm d}\eta 
\le C \, \int_{\R^{d-1}} \sum_{\ell \in \Z} \, \dfrac{|{\bf K} (k,\ell)|^2}{\langle (\eta,\ell) \rangle^{2\, \nu_0}} \, 
{\rm d}\eta \, .
\end{equation*}
We now use Lemma \ref{lem3} to derive the bound
\begin{equation*}
\left| \dfrac{(k-\ell) \, \ell}{(k-\ell) \, \delta_1 +\ell \, \delta_3} \right| =\dfrac{1}{|\delta_1|} \, 
\left| \ell -\dfrac{\delta_3 \, \ell^2}{(k-\ell) \, \delta_1 +\ell \, \delta_3} \right| \le C \, |\ell|^{\gamma_0+2} \, ,
\end{equation*}
from which we get
\begin{equation*}
\sup_{(\xi,k) \in \R^{d-1} \times \Z} \, \int_{\R^{d-1}} \sum_{\ell \in \Z} \left| 
\dfrac{\chi_1(\xi,k,\eta,\ell) \, \langle (\xi,k) \rangle^\nu \, {\bf K} (k,\ell)}{\langle (\xi-\eta,k-\ell) \rangle^\nu \, 
\langle (\eta,\ell) \rangle^{\nu_0}} \right|^2 \, {\rm d}\eta 
\le C \, \int_{\R^{d-1}} \sum_{\ell \in \Z} \, \dfrac{|\ell |^{2\, (\gamma_0+2)}}{\langle (\eta,\ell) \rangle^{2\, \nu_0}} 
\, {\rm d}\eta <+\infty \, ,
\end{equation*}
thanks to our choice of $\nu_0$. Applying Lemma \ref{lemRR} to the quantity in \eqref{tameterm1}, we obtain
\begin{equation*}
\int_{\R^{d-1}} \sum_{k \in \Z} \, \left| 
\int_{\R^{d-1}} \sum_{\ell \in \Z} \chi_1(\xi,k,\eta,\ell) \, \langle (\xi,k) \rangle^\nu \, {\bf K} (k,\ell) \, 
\widehat{c_{k-\ell}(u)}(\xi-\eta) \, \widehat{c_\ell(v)}(\eta) \, {\rm d}\eta \right|^2 \, {\rm d}\xi \le C\, 
\| u \|_{H^\nu}^2 \, \| v \|_{H^{\nu_0}}^2 \, .
\end{equation*}
Similar arguments yield the bound
\begin{equation*}
\int_{\R^{d-1}} \sum_{k \in \Z} \, \left| 
\int_{\R^{d-1}} \sum_{\ell \in \Z} \chi_2(\xi,k,\eta,\ell) \, \langle (\xi,k) \rangle^\nu \, {\bf K} (k,\ell) \, 
\widehat{c_{k-\ell}(u)}(\xi-\eta) \, \widehat{c_\ell(v)}(\eta) \, {\rm d}\eta \right|^2 \, {\rm d}\xi \le C\, 
\| u \|_{H^{\nu_0}}^2 \, \| v \|_{H^\nu}^2 \, ,
\end{equation*}
and the combination of the two previous inequalities gives the expected estimate
\begin{equation*}
\int_{\R^{d-1}} \sum_{k \in \Z} \, \langle (\xi,k) \rangle^{2\, \nu} \, 
|\widehat{c_k(\F_{\rm per}(\partial_\theta u,\partial_\theta v))} (\xi) |^2 \, {\rm d}\xi \le 
C \, \big( \| u \|_{H^{\nu_0}} \, \| v \|_{H^\nu} +\| u \|_{H^\nu} \, \| v \|_{H^{\nu_0}} \big) \, .
\end{equation*}

\end{proof}

\subsection{The iteration scheme}

In view of the boundedness property proved in Theorem \ref{propFper}, Equation \eqref{eqampli2} is a 
semilinear perturbation of the Burgers equation (the transport term ${\bf w} \cdot \nabla_y$ is harmless), 
and it is absolutely not surprising that we can solve \eqref{eqampli2} by using the standard energy method 
with a fixed point iteration. This well-posedness result can be summarized in the following Theorem.

\begin{theo}
\label{thmKinkModes}
Let $\nu_0$ be defined as in Theorem \ref{propFper}, and let $\nu \ge \nu_0$. Then for all $R>0$, there exists 
a time $T>0$ such that for all data $a_0 \in H^\nu(\R^{d-1} \times (\R/\Theta \, \Z))$ satisfying $\| a_0 \|_{H^{\nu_0}} 
\le R$, there exists a unique solution $a \in {\mathcal C}([0,T];H^\nu)$ to the Cauchy problem:
\begin{equation*}
\begin{cases}
\partial_t a +{\bf w} \cdot \nabla_y a +c \, a \, \partial_\theta a +\mu \, \F_{\rm per} (\partial_\theta a,\partial_\theta a) =0 \, ,& \\
a|_{t=0} =a_0 \, . &
\end{cases}
\end{equation*}
In particular, if $a_0 \in H^{+\infty}(\R^{d-1} \times (\R/\Theta \, \Z))$, then $a \in {\mathcal C}([0,T];H^{+\infty})$ 
where the time $T>0$ only depends on $\| a_0 \|_{H^{\nu_0}}$.
\end{theo}

\begin{proof}
The proof follows the standard strategy for quasilinear symmetric systems, see for instance \cite[chapter 10]{BS} 
or \cite[chapter 16]{taylor}, and we solve the Cauchy problem by the iteration scheme
\begin{equation*}
\begin{cases}
\partial_t a^{n+1} +{\bf w} \cdot \nabla_y a^{n+1} +c \, a^n \, \partial_\theta a^{n+1} +\mu \, \F_{\rm per} 
(\partial_\theta a^n,\partial_\theta a^n) =0 \, ,& \\
a^{n+1}|_{t=0} =a_{0,n+1} \, , &
\end{cases}
\end{equation*}
where $(a_{0,n})$ is a sequence of, say, Schwartz functions that converges towards $a_0$ in $H^\nu$, and 
the scheme is initialized with the choice $a^0 \equiv a_{0,0}$. Given the radius $R$ for the ball in $H^{\nu_0}$, we 
can choose some time $T>0$, that only depends on $R$ and $\nu$, such that the sequence $(a^n)$ is bounded 
in ${\mathcal C}([0,T];H^\nu)$. The uniform bound in ${\mathcal C}([0,T];H^\nu)$ is proved by following the 
exact same ingredients as in the case of the Burgers equation. Contraction in ${\mathcal C}([0,T];L^2)$ is 
obtained by computing the equation for the difference $r^{n+1} :=a^{n+1}-a^n$, which reads
\begin{equation}
\label{eqreste}
\partial_t r^{n+1} +{\bf w} \cdot \nabla_y r^{n+1} +c \, a^n \, \partial_\theta r^{n+1} =-c \, r^n \, \partial_\theta a^n 
-\mu \, \F_{\rm per}(\partial_\theta r^n,\partial_\theta a^n) -\mu \, \F_{\rm per}(\partial_\theta a^{n-1},\partial_\theta r^n) \, . 
\end{equation}
The error terms on the right hand-side are written as
\begin{align*}
\F_{\rm per}(\partial_\theta r^n,\partial_\theta a^n) &=-\dfrac{2\, \pi}{\Theta \, \delta_1} \, r^n \, \partial_\theta a^n 
-\dfrac{\delta_3}{\delta_1} \, \F_{\rm per}(r^n,\partial^2_{\theta \theta} a^n) \, ,\\
\F_{\rm per}(\partial_\theta a^{n-1},\partial_\theta r^n) &=-\dfrac{2\, \pi}{\Theta \, \delta_3} \, r^n \, \partial_\theta a^{n-1} 
-\dfrac{\delta_1}{\delta_3} \, \F_{\rm per}(\partial^2_{\theta \theta} a^{n-1},r^n) \, ,
\end{align*}
where we have used \eqref{propFper2}.

The final ingredient in the proof is a continuity estimate of the form
\begin{equation}
\label{FpercontL2}
\| \F_{\rm per}(u,v) \|_{L^2} \le C \, \min \big( \| u \|_{H^{\nu_0-2}} \, \| v \|_{L^2}, \| u \|_{L^2} \, \| v \|_{H^{\nu_0-2}} \big) \, ,
\end{equation}
which we now prove for completeness. We apply the Fubini and Parseval-Bessel Theorems to obtain
\begin{equation*}
\| \F_{\rm per}(u,v) \|_{L^2}^2 =\Theta \, \int_{\R^{d-1}} \sum_{k \in \Z} \left| \sum_{\underset{k_1 \, k_3 \neq 0}{k_1+k_3=k,}} 
\dfrac{1}{k_1 \, \delta_1 +k_3 \, \delta_3} \, u_{k_1} \, v_{k_3} \right|^2 \, {\rm d}y \, ,
\end{equation*}
and then apply the $\ell^1 \star \ell^2 \rightarrow \ell^2$ continuity estimate to derive
\begin{equation*}
\| \F_{\rm per}(u,v) \|_{L^2}^2 \le C\, \int_{\R^{d-1}} \left( \sum_{k \in \Z} |k|^{\gamma_0}| \, |u_k| \right) \, 
\sum_{k \in \Z} |v_k|^2 \, {\rm d}y \, .
\end{equation*}
We then apply the Cauchy-Schwarz inequality and derive the estimate
\begin{equation*}
\| \F_{\rm per}(u,v) \|_{L^2}^2 \le C \, \Big( \sup_{y \in \R^{d-1}} \| u (y,\cdot) \|_{H^{\gamma_0+1}(\R/\Theta \, \Z)}^2 
\Big) \, \| v \|_{L^2}^2 \, ,
\end{equation*}
which yields \eqref{FpercontL2} because the integer $\nu_0$ in Theorem \ref{propFper} can be chosen larger than 
$(d-1)/2+\gamma_0+3$. (The "symmetric" estimate is obtained by exchanging the roles of $u$ and $v$.)

At this stage, we multiply Equation \eqref{eqreste} by $r^{n+1}$ and perform integration by parts to derive
\begin{equation*}
\sup_{t \in [0,T]} \| r^{n+1} \|_{L^2}^2 \le \| r^{n+1}|_{t=0} \|_{L^2}^2 +C_0 \, T \, \sup_{t \in [0,T]} \| r^{n+1} \|_{L^2}^2 
+C_0 \, T \, \sup_{t \in [0,T]} \| r^{n+1} \|_{L^2} \, \sup_{t \in [0,T]} \| r^n \|_{L^2} \, ,
\end{equation*}
where the constant $C_0$ is independent of $n$ and follows from the uniform bound for $\sup_{t \in [0,T]} \| a^n 
\|_{H^\nu}$. By classical interpolation arguments, $(a^n)$ converges towards $a$ weakly in $L^\infty([0,T];H^\nu)$ 
and strongly in ${\mathcal C}([0,T];H^{\nu'})$, $\nu'<\nu$. Continuity of $a$ with values in $H^\nu$ is recovered by 
the standard arguments, see, e.g., \cite[Proposition 1.4]{taylor}.

If $\nu>\nu_0$, it remains to show that the time $T$ only depends on the norm $\| a_0 \|_{H^{\nu_0}}$, and this is 
where the tame estimate of Theorem \ref{propFper} enters the game. More precisely, we follow the same strategy 
as in \cite[Corollary 1.6]{taylor}, and show that the $H^\nu$-norm of the solution $a$ satisfies a differential inequality 
of the form
\begin{equation*}
\dfrac{{\rm d} \| a(t) \|_{H^\nu}^2}{{\rm d}t} \le C_\nu \Big( \| a(t) \|_{H^{\nu_0}}^2 \Big) \, \| a(t) \|_{H^\nu}^2 \, ,
\end{equation*}
where $C_\nu$ is an increasing function of its argument. In particular, boundedness of $a(t)$ in $H^{\nu_0}$ on an 
interval $[0,T')$, $T'>0$, implies a unique extension of the solution $a \in {\mathcal C}([0,T');H^\nu)$ beyond the time 
$T'$, which means that the time $T$ of existence for $a$ only depends on $\| a_0 \|_{H^{\nu_0}}$.
\end{proof}

\subsection{Construction of the leading profile}

Theorem \ref{thmKinkModes} is the cornerstone of the construction of the leading profile $\cU_0$. Solvability 
of \eqref{eqa3} for $a$ is summarized in the following result. Recall that the smoothness assumption for $G$ 
was made in Theorem \ref{theowavetrains}.

\begin{cor}
\label{cor1}
There exists $T>0$, and $a \in {\mathcal C}^\infty ((-\infty,T];H^{+\infty}(\R^{d-1} \times (\R/\Theta \, \Z)))$ solution 
to \eqref{eqa3} with $a|_{t<0}=0$. Furthermore, $a$ has mean value zero with respect to the variable $\theta$.
\end{cor}

\begin{proof}
Equation \eqref{eqa3} is easier to solve than the pure Cauchy problem in Theorem \ref{thmKinkModes} because 
we can apply Duhamel's formula starting from the initial condition $a_0=0$. From the assumption of Theorem 
\ref{theowavetrains}, we have $G \in {\mathcal C}^\infty ((-\infty,T_0]; H^{+\infty}(\R^{d-1} \times (\R/\Theta \, \Z)))$ 
with $T_0>0$ and $G|_{t<0}=0$, so we can find $0<T \le T_0$ and $a \in {\mathcal C}((-\infty,T];H^{+\infty}(\R^{d-1} 
\times (\R/\Theta \, \Z)))$ solution to \eqref{eqa3} with $a|_{t<0}=0$. Here the time $T$ depends on a fixed norm 
of $G$. Then Equation \eqref{eqa3} yields $a \in {\mathcal C}^\infty ((-\infty,T];H^{+\infty}(\R^{d-1} \times 
(\R/\Theta \, \Z)))$ by the standard bootstrap argument and smoothness of $G$.

For every fixed $y$, the mean value
\begin{equation*}
\underline{a}(t,y) :=\dfrac{1}{\Theta} \, \int_0^\Theta a(t,y,\theta) \, {\rm d}\theta \, ,
\end{equation*}
satisfies the homogeneous transport equation
\begin{equation*}
\partial_t \underline{a} +{\bf w} \cdot \nabla_y \underline{a} =0 \, ,
\end{equation*}
with zero initial condition, and therefore vanishes.
\end{proof}

\noindent After constructing $a$ on the boundary, we can achieve the construction of the leading profile $\cU_0$ 
in the whole domain.

\begin{cor}
\label{cor2}
Up to retricting $T>0$ in Corollary \ref{cor1}, for all $m \in {\mathcal I}$, there exists a unique solution $\sigma_m 
\in {\mathcal C}^\infty ((-\infty,T]; H^{+\infty}(\R^d_+ \times (\R/\Theta \, \Z)))$ to \eqref{eq:Burgersm} with 
$\sigma_m|_{t<0}=0$ and $\sigma_m|_{x_d=0} =\mathfrak{e}_m \, a$, where the real number $\mathfrak{e}_m$ is 
defined by $e_m =\mathfrak{e}_m \, r_m$. Furthermore, each $\sigma_m$ has mean value zero with respect to the 
variable $\theta_m$.
\end{cor}

The result follows from solving the boundary value problem for the Burgers equation \eqref{eq:Burgersm} with 
prescribed Dirichlet boundary condition on $\{ x_d=0\}$. This is a (very!) particular case of a quasilinear hyperbolic 
system with strictly dissipative boundary conditions for which well-posedness follows from the classical theory, see, 
e.g., \cite{BS}.

The leading profile $\cU_0$ is then given by \eqref{decompU0} and belongs to ${\mathcal C}^\infty ((-\infty,T]; 
H^{+\infty}(\R^d_+ \times (\R/\Theta \, \Z)^N))$. Furthermore, its spectrum with respect to the periodic variables 
$(\theta_1,\dots,\theta_N)$ is included in the set
\begin{equation}
\label{defZNI}
\Z^N_{\mathcal I} := \big\{ \alpha \in \Z^N \, / \, \forall \, m \in {\mathcal O} \, , \, \alpha_m=0 \big\} \, ,
\end{equation}
as claimed in Theorem \ref{theowavetrains}.

\section{Proof of Theorem \ref{theowavetrains}}
\label{sect4}

\subsection{The WKB cascade}

In this paragraph, we give a more detailed version of \eqref{3}-\eqref{4}. We plug again the ansatz 
\eqref{BKW} in \eqref{0} and derive the set of equations \eqref{BKWint}-\eqref{BKWbord} below. We 
recall that the operators $\cL,\cM$ are defined in \eqref{3a}, while $L(\partial)$ is defined in \eqref{defopL}. 
Then the WKB cascade in the interior reads:
\begin{equation}
\label{BKWint}
\begin{split}
&{\rm (a)}\quad \cL(\partial_{\theta}) \, \cU_0 =0 \, ,\\
&{\rm (b)}\quad \cL(\partial_{\theta}) \, \cU_1+L(\partial) \, \cU_0 +\cM (\cU_0,\cU_0) =0 \, ,\\
&{\rm (c)}\quad \cL(\partial_{\theta}) \, \cU_{k+2} +L(\partial) \, \cU_{k+1} +\cM (\cU_0,\cU_{k+1}) 
+\cM (\cU_{k+1},\cU_0) +\F_k =0 \, ,\quad k \ge 0 \, ,\\
\end{split}
\end{equation}
with
\begin{align}
\forall \, k \ge 0 \, ,\quad \F_k :=& \, \partial_j \varphi_m \, \left( \sum_{\ell=2}^{k+2} \dfrac{1}{\ell !} \, 
\sum_{\kappa_1+\cdots+\kappa_\ell=k+2-\ell} {\rm d}^\ell A_j(0) \cdot (\cU_{\kappa_1},\dots,\cU_{\kappa_\ell}) \right) \, 
\partial_{\theta_m} \cU_0 \notag \\
&+\sum_{\ell=1}^{k+1} \A_j^{k+2-\ell} \, \partial_j \cU_{\ell-1} +\partial_j \varphi_m \, \sum_{\ell=1}^k 
\A_j^{k+2-\ell} \, \partial_{\theta_m} \cU_\ell \, ,\label{defFk} \\
\forall \, \nu \ge 1 \, ,\quad \A_j^\nu :=& \sum_{\ell=1}^\nu \dfrac{1}{\ell !} \, 
\sum_{\kappa_1+\cdots+\kappa_\ell=\nu-\ell} {\rm d}^\ell A_j(0) \cdot (\cU_{\kappa_1},\dots,\cU_{\kappa_\ell}) \, ,\notag
\end{align}
Observe that \eqref{BKWint}(c) coincides with \eqref{3}(c) for $k=0$. Furthermore, each matrix $\A_j^\nu$, 
$\nu \ge 1$, only depends on $\cU_0,\dots,\cU_{\nu-1}$, and therefore each source term $\F_k$, $k \ge 0$, 
only depends on $\cU_0,\dots,\cU_k$.

The set of boundary conditions for \eqref{BKWint} reads (recall $B ={\rm d}b(0)$):
\begin{equation}
\label{BKWbord}
\begin{split}
&{\rm (a)}\quad B\, \cU_0 =0 \, ,\\
&{\rm (b)}\quad B\, \cU_1 +\dfrac{1}{2} \, {\rm d}^2b(0) \cdot (\cU_0,\cU_0) =G(t,y,\theta_0) \, ,\\
&{\rm (c)}\quad B\, \cU_{k+2} +{\rm d}^2b(0) \cdot (\cU_0,\cU_{k+1}) +\G_k =0 \, ,\quad k \ge 0 \, ,
\end{split}
\end{equation}
with
\begin{equation}
\label{defGk}
\forall \, k \ge 0 \, ,\quad \G_k :=\sum_{\ell=3}^{k+3} \dfrac{1}{\ell !} \, \sum_{\kappa_1+\cdots+\kappa_\ell=k+3-\ell} 
{\rm d}^\ell b(0) \cdot (\cU_{\kappa_1},\dots,\cU_{\kappa_\ell}) \, .
\end{equation}
In \eqref{BKWbord} and \eqref{defGk}, all functions on the left hand side are evaluated at $x_d=0, \theta_1=\cdots=\theta_N 
=\theta_0$. The source term $\G_k$ in \eqref{BKWbord}(c) only depends on $\cU_0,\dots,\cU_k$.

We are looking for a sequence of profiles $(\cU_k)_{k \in \N}$ that satisfies \eqref{BKWint}-\eqref{BKWbord}, and 
$\cU_k|_{t<0}=0$ for all $k$.

\subsection{Construction of correctors}
\label{correctors-per}

Some notation will be useful in the arguments below. For any function $f$ that depends on $(t,x,\theta_1,\dots,\theta_N)$, 
with $\Theta$-periodicity with respect to each $\theta_m$, we decompose $f$ as
\begin{equation*}
f=\underline{f}(t,x) +\sum_{m=1}^N f^m (t,x,\theta_m) +f^{\rm nc} (t,x,\theta_1,\dots,\theta_N) \, ,
\end{equation*}
where $\underline{f}$ stands for the mean value of $f$ on the torus $(\R/\Theta \, \Z)^N$, each $f^m$ incorporates 
the $\theta_m$-modes of $f$ (in particular, the spectrum of $f^m$ is included in $\Z^{N;1}$ and $f_m$ has mean zero 
with respect to $\theta_m$), and the spectrum of $f^{\rm nc}$ is included in $\Z^N \setminus \Z^{N;1}$. Here, the 
spectrum only refers to the Fourier decomposition of $f$ with respect to $(\theta_1,\dots,\theta_N)$. The mappings 
$f \mapsto f^m$ and $f \mapsto f^{\rm nc}$ are continuous on ${\mathcal C}^\infty ((-\infty,T]; H^{+\infty}(\R^d_+ \times 
(\R/\Theta \, \Z)^N))$. Furthermore, if the spectrum of $f$ is included in $\Z^N_{{\mathcal I}}$, then $f^{\rm nc}$ belongs 
to the space of profiles $\PP^{\rm nc}$ defined in Lemma \ref{lem4} below.

The following observation is well-known in the theory of geometric optics, see for instance \cite{JMR2,williams4}, 
and relies on Assumption \ref{assumption5}.

\begin{lem}
\label{lem4}
The operator $\cL(\partial_{\theta})$ is a bounded isomorphism from $\PP^{\rm nc}$ into itself, where
\begin{equation*}
\PP^{\rm nc} := \Big\{ f \in {\mathcal C}^\infty ((-\infty,T]; H^{+\infty}(\R^d_+ \times (\R/\Theta \, \Z)^N)) \, / \, 
\text{\rm Spectrum } (f) \subset \Z^N_{{\mathcal I}} \setminus \Z^{N;1} \Big\} \, .
\end{equation*}
\end{lem}

\noindent Indeed, for $\alpha \in \Z^N_{{\mathcal I}} \setminus \Z^{N;1}$, the matrix $L({\rm d} (\alpha \cdot \Phi))$ 
is invertible and the norm of its inverse is bounded polynomially in $|\alpha|$ (the degree of the polynomial being 
fixed). We shall feel free to write $\cL(\partial_{\theta})^{-1} \, f^{\rm nc}$ when $f^{\rm nc}$ is an element of 
$\PP^{\rm nc}$.

Unsurprisingly, the construction of the sequence $(\cU_k)_{k \in \N}$ is based on an induction process. We formulate 
our induction assumption.

\begin{description}

\item[$({\bf H}_n)$] There exist profiles $\cU_0,\dots,\cU_n$ in ${\mathcal C}^\infty ((-\infty,T]; H^{+\infty} (\R^d_+ 
\times (\R/\Theta \, \Z)^N))$ that vanish for $t<0$, whose $(\theta_1,\dots,\theta_N)$-spectrum is included in 
$\Z^N_{{\mathcal I}}$, and that satisfies
\begin{itemize}
 \item \eqref{BKWint}(a) and \eqref{BKWbord}(a) if $n=0$,
 \item \eqref{BKWint}(a)-(b) and \eqref{BKWbord}(a)-(b) if $n=1$,
 \item \eqref{BKWint}(a)-(b), \eqref{BKWbord}(a)-(b), \eqref{BKWint}(c) and \eqref{BKWbord}(c) up to order $n-2$ 
          if $n \ge 2$,
  \item Compatibility condition in the interior:
\begin{equation}
\label{eqintn}
\forall \, m \in {\mathcal I} \, ,\quad \ell_m \, F_n^m =0 \, ,\quad \text{\rm and } \underline{F_n}=0 \, ,
\end{equation}
 \item Compatibility condition on the boundary:
\begin{multline}
\label{eqbordn}
\underline{b} \, \Big( -\sum_{m \in {\mathcal I}} B \, R_m \, F_n^m|_{x_d=0,\theta_m=\theta_0} 
-B \, \partial_{\theta_0} ((\cL(\partial_{\theta})^{-1} \, F_n^{\rm nc})|_{x_d=0,\theta_1=\dots=\theta_N=\theta_0}) \\
+\partial_{\theta_0} G_n(t,y,\theta_0) \Big) =0 \, ,
\end{multline}
          where in \eqref{eqintn} and \eqref{eqbordn}, we have set:
\begin{align}
F_n &:= \begin{cases}
L(\partial) \, \cU_0 +\cM (\cU_0,\cU_0) \, ,&\text{\rm if } n=0 \, ,\\
L(\partial) \, \cU_n +\cM (\cU_0,\cU_n) +\cM (\cU_n,\cU_0) +\F_{n-1} \, ,&\text{\rm if } n \ge 1 \, ,
\end{cases} \label{defFn} \\
G_n &:= \begin{cases}
\dfrac{1}{2} \, {\rm d}^2b(0) \cdot (\cU_0,\cU_0)|_{x_d=0,\theta_1=\dots=\theta_N=\theta_0} -G(t,y,\theta_0) 
\, ,&\text{\rm if } n=0 \, ,\\
{\rm d}^2b(0) \cdot (\cU_0,\cU_n)|_{x_d=0,\theta_1=\dots=\theta_N=\theta_0} +\G_{n-1} \, ,&\text{\rm if } n \ge 1 \, .
\end{cases} \label{defGn}
\end{align}
\end{itemize}

\end{description}

\noindent Recall that $\F_k$ and $\G_k$ are defined in \eqref{defFk} and \eqref{defGk}, so the above source terms 
$F_n,G_n$ only depend on $\cU_0,\dots,\cU_n$. Several properties of these source terms are made precise below 
which, in particular, will justify why we can apply the operator $\cL(\partial_{\theta})^{-1}$ to $F_n^{\rm nc}$.

The reader can verify that our construction of the leading profile $\cU_0$ in Sections \ref{sect2} and \ref{sect3} 
proves that $({\bf H}_0)$ holds for some $T>0$. In that case, \eqref{eqbordn} reduces to \eqref{eqa3}, which 
was solved in Section \ref{sect3}. The compatibility conditions \eqref{eqintn} in the interior give the decoupled 
Burgers equations \eqref{eq:Burgersm} for each incoming amplitude $\sigma_m$.

Our goal is to show that $({\bf H}_n)$ implies $({\bf H}_{n+1})$ with the same time $T>0$, which will imply that 
there exists a sequence of profiles $(\cU_k)_{k \in \N}$ in ${\mathcal C}^\infty ((-\infty,T]; H^{+\infty}(\R^d_+ 
\times (\R/\Theta \, \Z)^N))$ that satisfies \eqref{BKWint}-\eqref{BKWbord}, and $\cU_k|_{t<0}=0$ for all $k$. 
We decompose the analysis in several steps, as in Section \ref{sect2}, assuming from now on that $({\bf H}_n)$ 
holds for some integer $n \in \N$.
\bigskip

\noindent \underline{Step 1: properties of $F_n,G_n$, and definition of $\cU_{n+1}^{\rm nc}$, $(I-P_m) \, \cU_{n+1}^m$.}

From assumption $({\bf H}_n)$, the profiles $\cU_0,\dots,\cU_n$ belong to ${\mathcal C}^\infty ((-\infty,T]; H^{+\infty} 
(\R^d_+ \times (\R/\Theta \, \Z)^N))$, vanish for $t<0$, and their spectrum is a subset of $\Z^N_{{\mathcal I}}$. The 
space of such functions is an algebra, and therefore, we can verify from \eqref{defFk} and \eqref{defFn} that $F_n$ 
belongs to ${\mathcal C}^\infty ((-\infty,T]; H^{+\infty} (\R^d_+ \times (\R/\Theta \, \Z)^N))$, vanishes for $t<0$, and 
its spectrum is a subset of $\Z^N_{{\mathcal I}}$. Consequently, $F_n^{\rm nc}$ belongs to the space of profiles 
$\PP^{\rm nc}$ defined in Lemma \ref{lem4}.

Our goal is to construct a profile $\cU_{n+1}$ that satisfies
\begin{equation}
\label{eqintUn+1}
\cL (\partial_\theta) \, \cU_{n+1} +F_n =0 \, .
\end{equation}
We first observe that if $\alpha \in \Z^N$ is a noncharacteristic mode, that is, $\alpha \in \Z^N \setminus \Z^{N;1}$, and 
if moreover $\alpha$ has one nonzero coordinate $\alpha_m$ with $m \in {\mathcal O}$, then the $\alpha$-Fourier 
coefficient of $\cU_{n+1}$ vanishes. This implies that any solution to \eqref{eqintUn+1} has its spectrum included 
in $\Z^N_{{\mathcal I}}$. We thus define $\cU_{n+1}^{\rm nc} :=-\cL (\partial_\theta)^{-1} \, F_n^{\rm nc}$, so 
$\cU_{n+1}^{\rm nc}$ belongs to $\PP^{\rm nc}$ and vanishes for $t<0$.

For all $m=1,\dots,N$, we define $(I-P_m) \, \cU_{n+1}^m$ as the unique mean zero solution to
\begin{equation*}
(I-P_m) \, \partial_{\theta_m} \cU_{n+1}^m =-R_m \, F_n^m \, .
\end{equation*}
In particular, we can write $\cU_{n+1}^m =\sigma^m_{n+1} \, r_m$ for $m \in {\mathcal O}$ since $F_n$ has no 
outgoing mode. Due to the compatibility condition \eqref{eqintn} for $m \in {\mathcal I}$, we have
\begin{equation*}
\forall \, m =1,\dots,N \, ,\quad L({\rm d}\varphi_m) \, (I-P_m) \, \partial_{\theta_m} \cU_{n+1}^m +F_n^m =0 \, ,
\end{equation*}
which means that the components of $\cU_{n+1}$ that we have already defined satisfy
\begin{equation*}
\cL (\partial_\theta) \, \left( \sum_{m=1}^N (I-P_m) \, \cU_{n+1}^m +\cU_{n+1}^{\rm nc} \right) +F_n =0 \, ,
\end{equation*}
because $F_n$ has zero mean value. It is clear that the nonpolarized components $(I-P_m) \, \cU_{n+1}^m$ 
belong to ${\mathcal C}^\infty ((-\infty,T]; H^{+\infty} (\R^d_+ \times (\R/\Theta \, \Z)^N))$ and vanish for $t<0$, 
because the $F_n^m$'s do so.
\bigskip

\noindent \underline{Step 2: $\cU_{n+1}$ has no outgoing mode.}

Let $m \in {\mathcal O}$. We focus on \eqref{BKWint}(c) for $k=n$. Since $\F_n$ has no outgoing mode, the profile 
$\cU_{n+1}$ must necessarily satisfy
\begin{equation*}
\ell_m \, L(\partial) \, (\sigma^m_{n+1} \, r_m) +\ell_m \, (\cM (\cU_0,\cU_{n+1}) 
+\cM (\cU_{n+1},\cU_0))^m =0 \, .
\end{equation*}
However, the leading profile $\cU_0$ is given in \eqref{decompU0}, and since the only noncharacteristic modes of 
$\cU_{n+1}$ belong to $\Z^N_{{\mathcal I}}$, we observe that it is not possible to generate a $\theta_m$-mode in 
$\cM (\cU_0,\cU_{n+1})$ nor in $\cM (\cU_{n+1},\cU_0)$. This means that the amplitude $\sigma^m_{n+1}$ satisfies 
the outgoing transport equation
\begin{equation*}
\partial_t \sigma^m_{n+1} +{\bf v}_m \cdot \nabla_x \sigma^m_{n+1} =0 \, ,
\end{equation*}
and therefore vanishes.

Since the profile $\cU_{n+1}$ has no outgoing mode, it reads
\begin{equation*}
\cU_{n+1} =\underline{\cU_{n+1}} +\sum_{m \in {\mathcal I}} P_m \, \cU_{n+1}^m 
+\sum_{m \in {\mathcal I}} (I-P_m) \, \cU_{n+1}^m +\cU_{n+1}^{\rm nc} \, ,
\end{equation*}
and it only remains to determine the mean value $\underline{\cU_{n+1}}$ and the polarized components $P_m \, 
\cU_{n+1}^m$, $m \in {\mathcal I}$. Let us observe that such components have no influence on the fulfillment of 
\eqref{eqintUn+1}, no matter how we define them because they belong to the kernel of $\cL(\partial_\theta)$. 
Hence we shall no longer focus on \eqref{eqintUn+1}, but rather on the boundary conditions for $\cU_{n+1}$ 
and the interior compatibility condition at the next order.
\bigskip

\noindent \underline{Step 3: determining $\underline{\cU_{n+1}}$.}

Let us first derive the interior equation for $\underline{\cU_{n+1}}$. We introduce the notation $P_m \, \cU_{n+1}^m 
=\sigma_{n+1}^m \, r_m$ for $m \in {\mathcal I}$. We consider \eqref{BKWint}(c) for $k=n$, and take its mean value 
on the torus, observing first that both terms $\cM(\cU_0,\cU_{n+1}^{\rm nc})$ and $\cM(\cU_{n+1}^{\rm nc},\cU_0)$ 
have zero mean value. Hence we derive the equation
\begin{align*}
L(\partial) \, \underline{\cU_{n+1}} &+\sum_{m \in {\mathcal I}} \int_0^\Theta \sigma_m \, \partial_{\theta_m} 
\sigma_{n+1}^m \, \dfrac{{\rm d}\theta_m}{\Theta} \, \partial_j \varphi_m \, ({\rm d}A_j(0) \cdot r_m) \, r_m \\
&+\sum_{m \in {\mathcal I}} \int_0^\Theta \sigma_{n+1}^m \, \partial_{\theta_m} \sigma_m \, 
\dfrac{{\rm d}\theta_m}{\Theta} \, \partial_j \varphi_m \, ({\rm d}A_j(0) \cdot r_m) \, r_m 
+{\mathcal F}_n (t,x)=0 \, ,
\end{align*}
where the source term ${\mathcal F}_n$ is defined by:
\begin{equation*}
{\mathcal F}_n := \underline{\F_n} +\sum_{m \in {\mathcal I}} \underline{\cM(\cU_0,(I-P_m) \, \cU_{n+1}^m)} 
+\sum_{m \in {\mathcal I}} \underline{\cM((I-P_m) \, \cU_{n+1}^m,\cU_0)} \, .
\end{equation*}
We observe that each sum of integrals
\begin{equation*}
\int_0^\Theta \sigma_m \, \partial_{\theta_m} \sigma_{n+1}^m \, \dfrac{{\rm d}\theta_m}{\Theta} 
+\int_0^\Theta \sigma_{n+1}^m \, \partial_{\theta_m} \sigma_m \, \dfrac{{\rm d}\theta_m}{\Theta} 
\end{equation*}
vanishes, and therefore $\underline{\cU_{n+1}}$ must satisfy the system
\begin{equation*}
L(\partial) \, \underline{\cU_{n+1}} =-{\mathcal F}_n \, ,
\end{equation*}
in $(-\infty,T] \times \R^d_+$. The source term ${\mathcal F}_n$ belongs to ${\mathcal C}^\infty ((-\infty,T]; 
H^{+\infty} (\R^d_+))$ and vanishes for $t<0$.

The boundary conditions for $\underline{\cU_{n+1}}$ are obtained by taking the mean value of \eqref{BKWbord}(b) 
if $n=0$ or \eqref{BKWbord}(c) for $k=n-1$ if $n \ge 1$. In any case we find that $\underline{\cU_{n+1}}|_{x_d=0}$ 
must satisfy
\begin{equation*}
B \, \underline{\cU_{n+1}}|_{x_d=0} +B \, \underline{(\cU_{n+1}^{\rm nc}|_{x_d=0,\theta_1=\dots=\theta_N=\theta_0})} 
+\underline{G_n}(t,y)=0 \, .
\end{equation*}
Since $\cU_{n+1}^{\rm nc}$ has already been determined, we can apply the well-posedness result of \cite{C}, 
supplemented with the regularity result in \cite{morandosecchi}, and construct a solution $\underline{\cU_{n+1}} 
\in {\mathcal C}^\infty ((-\infty,T]; H^{+\infty} (\R^d_+))$ to the above equations (in the interior and on the boundary).
\bigskip

\noindent \underline{Step 4: determining $P_m \, \cU_{n+1}^m$. Part I.}

We keep the notation $P_m \, \cU_{n+1}^m =\sigma_{n+1}^m \, r_m$ for $m \in {\mathcal I}$. The evolution of 
$\sigma_{n+1}^m$ is obtained by imposing the compatibility condition:
\begin{equation*}
\ell_m \, L(\partial) \, \cU_{n+1}^m +\ell_m \, \Big( \cM(\cU_0,\cU_{n+1}) +\cM(\cU_{n+1},\cU_0) \Big)^m 
=-\ell_m \, \F_n^m \, .
\end{equation*}
Keeping on the left hand side only what is still unknown, we end up with:
\begin{multline*}
\partial_t \sigma_{n+1}^m +{\bf v}_m \cdot \nabla_x \sigma_{n+1}^m +c_m \, \big( 
\sigma_m \, \partial_{\theta_m} \, \sigma_{n+1}^m +\sigma_{n+1}^m \, \partial_{\theta_m} \, \sigma_m \big) \\
=-\dfrac{\ell_m}{\ell_m \, r_m} \left[ 
\F_n^m +L(\partial) \, (I-P_m) \, \cU_{n+1}^m +\cM(\underline{\cU_{n+1}},\sigma_m \, r_m) \right. \\
\left. +\cM(\sigma_m \, r_m,(I-P_m) \, \cU_{n+1}^m) +\cM((I-P_m) \, \cU_{n+1}^m,\sigma_m \, r_m) 
+\Big( \cM(\cU_0,\cU_{n+1}^{\rm nc}) +\cM(\cU_{n+1}^{\rm nc},\cU_0) \Big)^m \right] \, ,
\end{multline*}
where the constant $c_m$ is the one defined in \eqref{eq:Burgersm}. We have thus found that the $\sigma_{n+1}^m$'s 
must satisfy decoupled incoming transport equations in $\R^d_+ \times (\R/\Theta \, \Z)$ with infinitely smooth coefficients 
and source terms (all vanishing for $t<0$). The only task left is therefore to determine the trace of each $\sigma_{n+1}^m$, 
$m \in {\mathcal I}$, on $x_d=0$.

We recall the decomposition \eqref{decomposition3} introduced in Section \ref{sect2}. We thus introduce a decomposition
\begin{equation}
\label{decompbord}
\sum_{m \in {\mathcal I}} \sigma_{n+1}^m(t,y,0,\theta_0) \, r_m =a_{n+1}(t,y,\theta_0) \, e +\check{\cU}_{n+1}(t,y,\theta_0) \, ,
\end{equation}
where $a_{n+1}$ is an unknown scalar function (with zero mean value with respect to $\theta_0$), and 
$\check{\cU}_{n+1}$ takes its values in $\check{\E}^s \tauetabar$. Thanks to the compatibility condition 
\eqref{eqbordn}, the function $\check{\cU}_{n+1}$ is uniquely determined by solving
\begin{multline*}
B \, \check{\cU}_{n+1} -\partial_{\theta_0}^{-1} \, \left( 
\sum_{m \in {\mathcal I}} B \, R_m \, F_n^m|_{x_d=0,\theta_m=\theta_0} \right) \\
-B \, \Big( (\cL(\partial_{\theta})^{-1} \, F_n^{\rm nc})|_{x_d=0,\theta_1=\dots=\theta_N=\theta_0} 
-\underline{(\cL(\partial_{\theta})^{-1} \, F_n^{\rm nc})|_{x_d=0,\theta_1=\dots=\theta_N=\theta_0}} \Big) 
+(G_n-\underline{G_n})(t,y,\theta_0) =0 \, ,
\end{multline*}
where $\partial_{\theta_0}^{-1}$ denotes the inverse of $\partial_{\theta_0}$ when restricted to zero mean value 
functions. We get $\check{\cU}_{n+1} \in {\mathcal C}^\infty ((-\infty,T]; H^{+\infty} (\R^{d-1} \times (\R/\Theta \, \Z)))$, 
and $\check{\cU}_{n+1}|_{t<0}=0$.

Before going on, we observe that no matter how we define the scalar function $a_{n+1}$, our definitions so far of all 
components of $\cU_{n+1}$ give the relation
\begin{equation}
\label{eqbordUn+1}
B \, \cU_{n+1}|_{x_d=0,\theta_1=\dots=\theta_N=\theta_0} +G_n =0 \, ,
\end{equation}
which means that \eqref{BKWbord}(b) will be satisfied if $n=0$, or \eqref{BKWbord}(c) will be satisfied for $k=n-1$ if 
$n \ge 1$. At this stage, it only remains to determine the scalar function $a_{n+1}$, and the trace of $\sigma_{n+1}^m$ 
on $\{ x_d=0\}$ will be obtained by computing
\begin{equation*}
\sigma_{n+1}^m(t,y,0,\theta_0) \, r_m =a_{n+1}(t,y,\theta_0) \, e_m +\check{\cU}_{n+1,m}(t,y,\theta_0) \, ,
\end{equation*}
where $\check{\cU}_{n+1,m}$ stands for the component of $\check{\cU}_{n+1}$ on $r_m$ in the decomposition 
\eqref{decomposition1}.
\bigskip

\noindent \underline{Step 5: determining $P_m \, \cU_{n+1}^m$. Part II. Evolution equation for $a_{n+1}$.}

This step mimics the analysis in Section \ref{sect2}, and more specifically Steps 4 and 5 there. Let us assume 
that $\cU_{n+2}$ has no outgoing mode, which will be fully justified at the next order of the induction process. 
The nonpolarized components of the corrector $\cU_{n+2}^m$ are obtained by looking at \eqref{BKWint}(c) 
for $k=n$. For $m \in {\mathcal I}$, we thus define $(I-P_m) \, \cU_{n+2}^m$ as the solution to
\begin{multline*}
(I-P_m) \, \partial_{\theta_m} \cU_{n+2}^m +R_m \, L(\partial) (\sigma_{n+1}^m \, r_m) 
+\big( \sigma_m \, \partial_{\theta_m} \sigma_{n+1}^m +\sigma_{n+1}^m \, \partial_{\theta_m} \sigma_m \big) \, 
\partial_j \varphi_m \, R_m \, ({\rm d}A_j(0) \cdot r_m) \, r_m \\
=-R_m \, \F_n^m -R_m \, L(\partial) (I-P_m) \, \cU_{n+1}^m -\cM(\underline{\cU_{n+1}},\sigma_m \, r_m) \\
-\cM(\sigma_m \, r_m,(I-P_m) \, \cU_{n+1}^m) -\cM((I-P_m) \, \cU_{n+1}^m,\sigma_m \, r_m) 
-\Big( \cM(\cU_0,\cU_{n+1}^{\rm nc}) +\cM(\cU_{n+1}^{\rm nc},\cU_0) \Big)^m \, ,
\end{multline*}
where we should keep in mind that $\sigma_{n+1}^m$ is still not fully determined, but the right hand side of the 
equality has already been constructed. Since the partial inverses $R_m$ satisfy $R_m \, A_d(0) \, P_m =0$, 
the term $R_m \, L(\partial) (\sigma_{n+1}^m \, r_m)$ only involves tangential differentiation with respect to 
the boundary $\{ x_d=0\}$, so we can take the trace of the latter equation on the boundary. We then substitute 
$a_{n+1}(t,y,\theta_m) \, e_m +\check{\cU}_{n+1,m}(t,y,\theta_m)$ for $\sigma_{n+1}^m|_{x_d=0} \, r_m$. 
These operations yield
\begin{multline}
\label{substitution1}
\underline{b} \, B \, \sum_{m \in {\mathcal I}} (I-P_m) \, (\partial_{\theta_m} \cU_{n+2}^m)|_{x_d=0,\theta_m=\theta_0} \\
=-X_{\rm Lop} a_{n+1} +(2\, \upsilon -\underline{b} \, {\rm d}^2b(0) \cdot (e,e) ) \, \big( a \, \partial_{\theta_0} a_{n+1} 
+a_{n+1} \, \partial_{\theta_0} a \big) +g_{1,n}(t,y,\theta_0) \, ,
\end{multline}
with $\upsilon$ and $X_{\rm Lop}$ defined in \eqref{defupsilonstrict}, and $g_{1,n}$ is explicitly computable 
from all previously determined quantities.

We now determine $\cU_{n+2}^{\rm nc}$ by using \eqref{BKWint}(c) for $k=n$, computing the noncharacteristic 
components and by taking the trace on $x_d=0$. All these operations lead to the equation
\begin{align*}
\cL(\partial_\theta) \, \cU_{n+2}^{\rm nc}|_{x_d=0} = &-\sum_{m \in {\mathcal I}} \Big[ 
\cM (\cU_0,P_m \, \cU_{n+1}^m) +\cM(P_m \, \cU_{n+1}^m, \cU_0) \Big]^{\rm nc}|_{x_d=0} \\
&-\sum_{m \in {\mathcal I}} \Big[ 
\cM (\cU_0,(I-P_m) \, \cU_{n+1}^m) +\cM((I-P_m) \, \cU_{n+1}^m, \cU_0) \Big]^{\rm nc}|_{x_d=0} \\
&-\Big[ \cM (\cU_0,\cU_{n+1}^{\rm nc}) +\cM(\cU_{n+1}^{\rm nc},\cU_0) \Big]^{\rm nc}|_{x_d=0} 
-\F_n^{\rm nc}|_{x_d=0} -L(\partial) \, \cU_{n+1}^{\rm nc}|_{x_d=0} \, .
\end{align*}
We then use the decomposition:
\begin{equation*}
P_m \, \cU_{n+1}^m|_{x_d=0} =a_{n+1}(t,y,\theta_m) \, e_m +\check{\cU}_{n+1,m}(t,y,\theta_m) \, ,
\end{equation*}
where $\check{\cU}_{n+1,m}$ has already been determined, and we thus obtain the expression:
\begin{equation}
\label{substitution2}
\cU_{n+2}^{\rm nc}|_{x_d=0,\theta_1=\cdots=\theta_N=\theta_0} =\B_{\rm per}(a,a_{n+1}) +g_{2,n} \, ,
\end{equation}
where
\begin{multline*}
\B_{\rm per}[u,v] :=-\sum_{\underset{m_1, m_2 \in {\mathcal I}}{m_1 \neq m_2}} \sum_{k \in \Z} \\
\left( \sum_{\underset{k_{m_1} \, k_{m_2} \neq 0}{k_{m_1}+k_{m_2}=k,}} 
(u_{k_{m_1}} \, v_{k_{m_2}} +u_{k_{m_2}} \, v_{k_{m_1}}) \, 
L(k_{m_1} \, {\rm d}\varphi_{m_1} +k_{m_2} \, {\rm d}\varphi_{m_2})^{-1} \, (k_{m_2} \, E_{m_2,m_1}) \right) 
\, {\rm e}^{2 \, i \, \pi \, k \, \theta_0/\Theta} \, ,
\end{multline*}
the vectors $E_{m_1,m_2}$ are defined in \eqref{defEm1m2}, and $g_{2,n}$ is explicitly computable 
from all previously determined quantities.

We now consider \eqref{BKWbord}(c) for $k=n$, which we rewrite equivalently as
\begin{align}
B\, \underline{\cU_{n+2}}|_{x_d=0} 
+B \, \sum_{m \in {\mathcal I}} P_m \, \cU_{n+2}^m|_{x_d=0,\theta_m=\theta_0} & \notag \\
+B \, \sum_{m \in {\mathcal I}} (I-P_m) \, \cU_{n+2}^m|_{x_d=0,\theta_m=\theta_0} 
+B \, \cU_{n+2}^{\rm nc}|_{x_d=0,\theta_1=\cdots=\theta_N=\theta_0} +a\, a_{n+1} \, {\rm d}^2b(0) \cdot (e,e) & 
\notag \\
+{\rm d}^2b(0) \cdot \left( a\, e,\underline{\cU_{n+1}} +\check{\cU}_{n+1}
+\sum_{m \in {\mathcal I}} (I-P_m) \, \cU_{n+1}^m +\cU_{n+1}^{\rm nc} \right) +\G_n &=0 \, .\label{substitution3}
\end{align}
We differentiate the latter equation with respect to $\theta_0$, apply the row vector $\underline{b}$ and use 
\eqref{substitution1} and \eqref{substitution2} to derive the governing equation for $a_{n+1}$:
\begin{equation}
\label{eqan+1}
2\, \upsilon \, \big( a \, \partial_{\theta_0} a_{n+1} +a_{n+1} \, \partial_{\theta_0} a \big) -X_{\rm Lop} a_{n+1} 
+\partial_{\theta_0} \, Q_{\rm per}[a,a_{n+1}] +\partial_{\theta_0} \, Q_{\rm per}[a_{n+1},a] =g_n \, ,
\end{equation}
where $g_n$ incorporates all contributions from the source terms $g_{1,n},g_{2,n}$ and the one obtained after 
differentiating the last line of \eqref{substitution3} and applying $\underline{b}$. Moreover, $Q_{\rm per}$ is the 
operator defined in \eqref{defQper'}.

Observe that the above governing equation \eqref{eqan+1} for $a_{n+1}$ is a linearized version of \eqref{eqa3}. 
Well-posedness for \eqref{eqan+1} follows from the same arguments as those we have used in Section \ref{sect3}, 
namely from Theorem \ref{propFper} which shows that \eqref{eqan+1} is a transport equation for $a_{n+1}$ 
that is perturbed by a nonlocal "zero order" term. We thus construct a solution $a_{n+1} \in {\mathcal C}^\infty 
((-\infty,T]; H^{+\infty} (\R^{d-1} \times (\R/\Theta \, \Z)))$ to \eqref{eqan+1} that vanishes for $t<0$.
\bigskip

\noindent \underline{Step 6: conclusion.}

We have now determined each component of $\cU_{n+1}$, which gives a profile in ${\mathcal C}^\infty ((-\infty,T]; 
H^{+\infty} (\R^d_+ \times (\R/\Theta \, \Z)^N))$ with its spectrum in $\Z^N_{{\mathcal I}}$. Moreover, $\cU_{n+1}$ 
vanishes for $t<0$, and it satisfies \eqref{eqintUn+1}, \eqref{eqbordUn+1}. It is also a simple exercise to verify that, 
if we define $F_{n+1}$ according to \eqref{defFn}, our construction of $\cU_{n+1}$ gives the compatibility conditions
\begin{equation*}
\forall \, m \in {\mathcal I} \, ,\quad \ell_m \, F_{n+1}^m =0 \, ,\quad \text{\rm and } \quad \underline{F_{n+1}}=0 \, ,
\end{equation*}
which is nothing but \eqref{eqintn} at the order $n+1$. Step 5 above also shows that, with $G_{n+1}$ defined as in 
\eqref{defGn}, we have the compatibility condition:
\begin{equation*}
\underline{b} \, \Big( -\sum_{m \in {\mathcal I}} B \, R_m \, F_{n+1}^m|_{x_d=0,\theta_m=\theta_0} 
-B \, \partial_{\theta_0} ((\cL(\partial_{\theta})^{-1} \, F_{n+1}^{\rm nc})|_{x_d=0,\theta_1=\dots=\theta_N=\theta_0}) \\
+\partial_{\theta_0} G_{n+1}(t,y,\theta_0) \Big) =0 \, ,
\end{equation*}
which is nothing but \eqref{eqbordn} at the order $n+1$. We have therefore proved that  $({\bf H}_{n+1})$ holds, 
which completes the induction.

\subsection{Proof of Theorem \ref{theowavetrains}}

We now quickly complete the proof of Theorem \ref{theowavetrains}. The analysis in Paragraph \ref{correctors-per} 
shows that there exists a sequence of profiles $(\cU_n)_{n \ge 0}$ in ${\mathcal C}^\infty ((-\infty,T]; H^{+\infty} 
(\R^d_+ \times (\R /(\Theta \, \Z))^N))$ that satisfies the WKB cascade \eqref{BKWint}, \eqref{BKWbord}, and 
$\cU_n|_{t<0}=0$ for all $n \in \N$. Moreover, each profile $\cU_n$ has its $\theta$-spectrum included in 
$\Z^N_{\mathcal I}$. The uniqueness of such a sequence also follows from an induction argument, where 
we use the regularity of each profile to justify all computations that we have made in order to construct the 
$\cU_n$'s (successive differentiations, identification of Fourier coefficients, substitution $(\theta_1,\dots,\theta_N) 
\rightarrow \theta_0$ on the boundary etc.). Uniqueness of the sequence $(\cU_n)_{n \ge 0}$ then follows from 
the uniqueness of smooth solutions to all amplitude equations we have had to solve: Equation \eqref{eqa3} and 
its linearized version on the boundary, Burgers equation \eqref{eq:Burgersm} and its linearized version in the 
interior. All other operations, such as the determination of the non-characteristic components, are of "algebraic" 
type and obviously admit a single smooth solution.

To complete the proof of Theorem \ref{theowavetrains}, we thus only need to show that the approximate 
solutions built from the sequence $(\cU_n)_{n \ge 0}$ satisfy the error estimates claimed in Theorem 
\ref{theowavetrains}. We thus consider
\begin{equation*}
u_\eps^{{\rm app},N_1,N_2}(t,x) := \sum_{n=0}^{N_1+N_2} \eps^{1+n} \, \cU_n \left( t,x,\dfrac{\Phi(t,x)}{\eps} \right) \, .
\end{equation*}
Let us first consider the boundary conditions in \eqref{0}. Setting $u_\eps^{{\rm app},N_1,N_2}|_{x_d=0} 
=\eps \, v_\eps (t,y)$ for simplicity, we compute
\begin{align*}
b(u_\eps^{{\rm app},N_1,N_2}|_{x_d=0}) =& \sum_{n=1}^{N_1+N_2+1} \dfrac{\eps^n}{n!} \, 
\sum_{\nu_1,\dots,\nu_n=0}^{N_1+N_2} {\rm d}^n b(0) \cdot (v_\eps,\dots,v_\eps) \\
&+\dfrac{\eps^{N_1+N_2+2}}{(N_1+N_2+1)!} \, \int_0^1 (1-s)^{N_1+N_2+1} \, {\rm d}^{N_1+N_2+2} b(\eps \, s\, v_\eps) 
\cdot (v_\eps, \dots, v_\eps)\, {\rm d}s \, .
\end{align*}
Collecting powers of $\eps$ and using the WKB cascade \eqref{BKWbord}, we get
\begin{equation*}
b(u_\eps^{{\rm app},N_1,N_2}|_{x_d=0}) =\eps^2 \, \int_0^1 (1-s) \, {\rm d}^2 b(\eps \, s\, \cU_0) \cdot (\cU_0,\cU_0) 
\left( t,y,0,\dfrac{\varphi_0(t,y)}{\eps},\dots,\dfrac{\varphi_0(t,y)}{\eps} \right) \, {\rm d}s \, ,
\end{equation*}
if $N_1=N_2=0$, and
\begin{multline*}
b(u_\eps^{{\rm app},N_1,N_2}|_{x_d=0}) -\eps^2 \, G \left( t,y,\dfrac{\varphi_0(t,y)}{\eps} \right) \\
= \sum_{n=N_1+N_2+2}^{(N_1+N_2+1)^2} \eps^n \, \sum_{\ell=1}^{N_1+N_2+1} 
\dfrac{1}{\ell !} 
\sum_{\underset{\nu_1,\dots,\nu_\ell \le N_1+N_2}{\nu_1 +\dots +\nu_\ell=n-\ell,}} 
{\rm d}^\ell b(0) \cdot (\cU_{\nu_1},\dots,\cU_{\nu_\ell}) 
\left( t,y,0,\dfrac{\varphi_0(t,y)}{\eps},\dots,\dfrac{\varphi_0(t,y)}{\eps} \right) \\
+\dfrac{\eps^{N_1+N_2+2}}{(N_1+N_2+1)!} \, \int_0^1 (1-s)^{N_1+N_2+1} \, {\rm d}^{N_1+N_2+2} b(\eps \, s\, v_\eps) 
\cdot (v_\eps, \dots, v_\eps)\, {\rm d}s \, ,
\end{multline*}
if $N_1+N_2>0$.

Each profile $\cU_n$ belongs to ${\mathcal C}^\infty ((-\infty,T]; H^{+\infty} (\R^d_+ \times (\R /(\Theta \, \Z))^N))$ and 
vanishes for $t<0$, hence it also belongs to $L^\infty ((-\infty,T] \times (\R /(\Theta \, \Z))^N; L^2(\R^{d-1})) \cap L^\infty 
((-\infty,T] \times \R^{d-1} \times (\R /(\Theta \, \Z))^N)$ when restricted to the boundary $\{ x_d=0\}$. We thus have 
uniform bounds with respect to $\eps$ of the type
\begin{equation*}
\| {\rm d}^\ell b(0) \cdot (\cU_{\nu_1},\dots,\cU_{\nu_\ell}) (t,y,0,\varphi_0/\eps,\dots,\varphi_0/\eps) 
\|_{L^\infty ((-\infty,T];L^2(\R^{d-1}))} \le C \, ,
\end{equation*}
and similarly for the above integral remainders in Taylor's formula. When $N_1+N_2$ is positive, we thus get
\begin{equation*}
\| b(u_\eps^{{\rm app},N_1,N_2}|_{x_d=0}) -\eps^2 \, G(t,y,\varphi_0/\eps) \|_{L^\infty ((-\infty,T];L^2(\R^{d-1}))} 
\le C \, \eps^{N_1+N_2+2} \, ,
\end{equation*}
and we can derive the exact same $O(\eps^{N_1+N_2+2})$ estimate for the $L^\infty$ norm of the error at the boundary. 
When differentiating with respect to $y$, each partial derivative gives rise to a factor $1/\eps$, which yields
\begin{equation*}
\| b(u_\eps^{{\rm app},N_1,N_2}|_{x_d=0}) -\eps^2 \, G(t,y,\varphi_0/\eps) \|_{L^\infty ((-\infty,T];H^{N_2}(\R^{d-1}))} 
\le C \, \eps^{N_1+2} \, .
\end{equation*}
The case $N_1=N_2=0$ is similar, except that the error is as large as the source term $\eps^2 \, G$, which is not 
so interesting from a practical point of view.

We leave to the interested reader the verification of the error estimate in the interior domain, which involves a 
little more algebra but no additional technical difficulty.

\subsection{Extension to hyperbolic systems with constant multiplicity}
\label{sect5}

\emph{\quad} The extension of the derivation of the leading amplitude equation \eqref{eqa3} to hyperbolic systems 
with constant multiplicity is not entirely straightforward for two reasons, one related to  the zero mean property of 
$\cU_0$ and the other to the solvability of the interior profile equations.

When we analyzed the WKB cascade \eqref{BKWint}, \eqref{BKWbord}, the first point was to prove that the leading 
amplitude $\cU_0$ necessarily has zero mean. This property does not extend obviously to the case of hyperbolic 
systems with constant multiplicity (similar issues arise in \cite{CGW1}). Let us recall that when the system is hyperbolic 
with constant multiplicity, Lemma \ref{lem1} and Lemma \ref{lem2} hold. We then use the notation introduced in 
Paragraph \ref{notation} for the projectors $P_m, Q_m$ and the partial inverses $R_m$. The only difference with 
the strictly hyperbolic framework is that we do not use the row vectors $\ell_m$ here. We now analyze the equations 
\eqref{3}, \eqref{4} in this more general framework.

Equation \eqref{3}(a) shows that the leading profile $\cU_0$ can be decomposed as
\begin{equation*}
\cU_0(t,x,\theta_1,\dots,\theta_M) = \underline{\cU}_0(t,x) +\sum_{m=1}^M \cU_0^m(t,x,\theta_m) \, ,\quad 
P_m \, \cU_0^m =\cU_0^m \, ,
\end{equation*}
where each $\cU_0^m$ is $\Theta$-periodic and has zero mean with respect to $\theta_m$. The mean value 
$\underline{\cU}_0$ satisfies the homogeneous boundary condition \eqref{eq:moyenne2}, and we are going to 
show that \eqref{eq:moyenne1} still holds. Indeed, the mean value $\underline{\cU}_0$ satisfies
\begin{equation*}
L(\partial) \, \underline{\cU}_0 +\underline{\cM(\cU_0,\cU_0)} =0 \, ,
\end{equation*}
which, in view of the decomposition of $\cU_0$ and the definition \eqref{3a} of $\cM$, can be first simplified into
\begin{equation*}
L(\partial) \, \underline{\cU}_0 +\sum_{m=1}^M \underline{\cM(\cU_0^m,\cU_0^m)} =0 \, .
\end{equation*}
We then compute
\begin{equation}\label{symm}
\cM(\cU_0^m,\cU_0^m) =\partial_j \varphi_m \, ({\rm d}A_j(0) \cdot \cU_0^m) \, \partial_{\theta_m} \cU_0^m 
=\partial_j \varphi_m \, {\rm d}^2f_j(0) \cdot (\cU_0^m,\partial_{\theta_m} \cU_0^m) 
=\dfrac{1}{2} \, \partial_j \varphi_m \, \partial_{\theta_m} {\rm d}^2f_j(0) \cdot (\cU_0^m,\cU_0^m) \, ,
\end{equation}
where we have used Assumption \ref{assumption1'} and the symmetry of ${\rm d}^2f_j(0)$. Each 
$\cM(\cU_0^m,\cU_0^m)$ therefore has mean zero and $\underline{\cU}_0$ satisfies both \eqref{eq:moyenne1} 
and \eqref{eq:moyenne2} as in the strictly hyperbolic case. The result of Step 1 in Paragraph \ref{sect2example} 
thus extends to conservative hyperbolic systems with constant multiplicity.

We derive the interior equation for each $\cU_0^m$ by retaining only the $\theta_m$-oscillations in \eqref{3}(b) 
and applying the projector $Q_m$. Using \cite[Lemma 3.3]{CG} and the absence of resonances, we get
\begin{equation}
\label{eq:Burgerssyst}
(\partial_t +{\bf v}_m \cdot \nabla _x) \, Q_m \, \cU_0^m +\dfrac{1}{2} \, \partial_j \varphi_m \, \partial_{\theta_m} 
Q_m \, {\rm d}^2f_j(0) \cdot (\cU_0^m,\cU_0^m) =0 \, .
\end{equation}
Let us also recall that the restriction of $Q_m$ to $\text{\rm Im } P_m$ is injective, so that $Q_m \, \cU_0^m$ 
uniquely determines $\cU_0^m$ with $\cU_0^m =P_m \, \cU_0^m$.   In spite of the symmetry of $d^2f(0)$, 
\eqref{eq:Burgerssyst} is not obviously a symmetric hyperbolic problem for the unknown $Q_m \, \cU_0^m$, 
so its solvability is not immediately obvious. We encountered a similar difficulty in \cite{CGW1}, and we resolve 
it here in a similar way using the  expression $\partial_j \varphi_m \, ({\rm d}A_j(0) \cdot \cU_0^m) \, 
\partial_{\theta_m} \cU_0^m$ for the nonlinear term and the following lemma.

\begin{lem}\label{sh}
Let $w\in\R^N$ and write ${\rm d}\varphi_m=(\utau,\ueta,\uomega_m)=(-\lambda_{k_m}(\uxi),\uxi)$, where 
$\lambda_{k_m}(\uxi)=\lambda_{k_m}(u,\uxi)|_{u=0}$. Then there holds
\begin{equation}
\label{syhy}
\left( Q_m \, \sum^d_{j=1} \uxi_j \, {\rm d}_u A_j(0) \cdot w \right) \, P_m =\big( 
-{\rm d}_u\lambda_{k_m}(0,\uxi) \cdot w \big) \, Q_m \, P_m \, .
\end{equation}
\end{lem}

\begin{proof}
For $u$ near $0$ let $P_m(u)$ be the projector on
\begin{equation*}
\mathrm{Ker} \left(-\lambda_{k_m}(u,\uxi) \, I +\sum^d_{j=1} \uxi_j \, A_j(u) \right) \, ,
\end{equation*}
in the obvious decomposition of $\C^N$ (that is the analogue of the first direct sum in \eqref{decomposition2} for $u$ 
close to the origin); thus the projector $P_m$ in Lemma \ref{lem2} is $P_m=P_m(0)$. Differentiate the equation
\begin{equation*}
\left(-\lambda_{k_m}(u,\uxi) \, I+\sum^d_{j=1} \uxi_j \, A_j(u) \right) \, P_m(u)=0 \, ,
\end{equation*}
with respect to $u$ in the direction $w$, evaluate at $u=0$, and apply $Q_m$ on the left to obtain \eqref{syhy}.
\end{proof}

Taking $w=\cU_0^m$ in \eqref{syhy} and using $\cU^m_0=P_m \, \cU^m_0$, we see that \eqref{eq:Burgerssyst} 
is a symmetric hyperbolic system (essentially scalar) for the unknown $Q_m \, \cU^m_0$, so it can be solved with 
appropriate boundary and initial conditions just like the corresponding equations in the strictly hyperbolic case.

Equation \eqref{eq:Burgerssyst} shows that all outgoing modes in $\cU_0$ vanish, and there holds
\begin{equation}\label{polar}
\cU_0(t,x,\theta_1,\dots,\theta_M) = \sum_{m \in {\mathcal I}} \cU_0^m(t,x,\theta_m) \, ,\quad 
P_m \, \cU_0^m =\cU_0^m \, .
\end{equation}
In particular, there exists a scalar function $a$ that is $\Theta$-periodic with zero mean, such that
\begin{equation}\label{bcs}
\forall \, m \in {\mathcal I} \, ,\quad \cU_0^m(t,y,0,\theta_0) =a(t,y,\theta_0) \, e_m \, ,
\, \text{ where }e_m=P_m \, e\, .
\end{equation}
Our goal is to derive the amplitude equation that governs the evolution of $a$ on the boundary.

At this stage, the analysis of Steps 3, 4, 5 in Paragraph \ref{sect2example} applies almost word for word, 
with obvious modifications in order to take into account the possibly many incoming and outgoing phases. 
Namely, we can show that the first corrector $\cU_1$ has no outgoing mode. We can also determine the 
non-characteristic component $\cU_1^{\rm nc}|_{x_d=0}$ and the non-polarized components $(I-P_m) \, 
\cU_1^m|_{x_d=0}$ on the boundary in terms of the amplitude $a$. We end up with the exact same equation 
as \eqref{eqa3}, with a real constant $\upsilon$ and a vector field $X_{\rm Lop}$ as in \eqref{defupsilonstrict}. 
The new expression of the bilinear operator $Q_{\rm per}$ reads (compare with \eqref{defQper'}):
\begin{multline}\label{qoper}
Q_{\rm per}[a,\widetilde{a}] :=-\sum_{m \in {\mathcal O}} \, \sum_{\underset{m_1, m_2 \in {\mathcal I}}{m_1<m_2}} \\
\sum_{k \in \Z} \left( \sum_{\underset{k_{m_1} \, k_{m_2} \neq 0}{k_{m_1}+k_{m_2}=k,}} 
\dfrac{k_{m_1} \, \underline{b} \, B \, A_d(0)^{-1} \, Q_m \, E_{m_1,m_2} 
+k_{m_2} \, \underline{b} \, B \, A_d(0)^{-1} \, Q_m \, E_{m_2,m_1}}{k_{m_1} \, 
(\underline{\omega}_{m_1}-\underline{\omega}_m) +k_{m_2} \, (\underline{\omega}_{m_2}-\underline{\omega}_m)} 
\, a_{k_{m_1}} \, \widetilde{a}_{k_{m_2}} \right) \, {\rm e}^{2 \, i \, \pi \, k \, \theta_0/\Theta} \, ,
\end{multline}
with vectors $E_{m_1,m_2}$ as in \eqref{defEm1m2}.

The analysis of the WKB cascade \eqref{BKWint}, \eqref{BKWbord} proceeds as before, taking into account 
that incoming amplitudes $\cU_j^m$ are propagated in the interior domain by vector-valued Burgers type 
equations \eqref{eq:Burgerssyst} and appropriate linearizations at $\cU_0^m$, which can be solved using 
Lemma \ref{sh}.

\newpage
\part{Pulses}
\label{part2}

\section{Construction of approximate solutions}
\label{sect6}

\emph{\quad}We follow the approach of Section \ref{sect2} and first deal with strictly hyperbolic systems of three 
equations. We keep the notation of Paragraph \ref{sect2example}, and make Assumption \ref{assumption6}. Let 
us now derive the WKB cascade for pulse-like solutions to \eqref{0}. The solution $u_\eps$ to \eqref{0} is assumed 
to have an asymptotic expansion of the form
\begin{equation}
\label{BKWp}
u_\eps \sim \eps \, \sum_{k \ge 0} \eps^k \, \cU_k \left(t,x,\dfrac{\varphi_0(t,y)}{\eps},\dfrac{x_d}{\eps} \right) \, .
\end{equation}
We use the notation $\theta_0$ as a placeholder for the fast variable $\varphi_0/\eps$, and $\xi_d$ for $x_d/\eps$. 
Plugging the ansatz \eqref{BKWp} in \eqref{0} and identifying powers of $\eps$, we obtain the following first three 
relations for the $\cU_k$'s (observe the slight differences with Paragraph \ref{sect2example}, though we keep the 
same notation):
\begin{equation}
\label{3p}
\begin{split}
&{\rm (a)}\quad \cL(\partial_{\theta_0},\partial_{\xi_d}) \, \cU_0 =0 \, ,\\
&{\rm (b)}\quad \cL(\partial_{\theta_0},\partial_{\xi_d}) \, \cU_1+L(\partial) \, \cU_0 +\cM (\cU_0,\cU_0) =0 \, ,\\
&{\rm (c)}\quad \cL(\partial_{\theta_0},\partial_{\xi_d}) \, \cU_2 +L(\partial) \, \cU_1 +\cM (\cU_0,\cU_1) +\cM (\cU_1,\cU_0) 
+\cN_1 (\cU_0,\cU_0) +\cN_2 (\cU_0,\cU_0,\cU_0) =0 \, ,\\
\end{split}
\end{equation}
where the differential operators $\cL,\cM,\cN_1,\cN_2$ are now defined by:
\begin{equation}
\label{3ap}
\begin{split}
&\cL(\partial_{\theta_0},\partial_{\xi_d}) := A_d(0) \, (\partial_{\xi_d} +i\, {\mathcal A} \tauetabar \, \partial_{\theta_0}) \, ,\\
&\cM (v,w) := \partial_j \varphi_0 \, ({\rm d}A_j(0) \cdot v) \, \partial_{\theta_0} w +({\rm d}A_d(0) \cdot v) \, \partial_{\xi_d} w \, ,\\
&\cN_1 (v,w) := ({\rm d}A_j(0) \cdot v) \, \partial_j w \, ,\\
&\cN_2 (v,v,w) := \dfrac{1}{2} \, \partial_j \varphi_0 \, ({\rm d}^2A_j(0) \cdot (v,v)) \, \partial_{\theta_0} w 
+\dfrac{1}{2} \, ({\rm d}^2A_d(0) \cdot (v,v)) \, \partial_{\xi_d} w \, .\\
\end{split}
\end{equation}
The equations \eqref{3p} in the domain $(-\infty,T] \times \R^d_+ \times \R_{\theta_0} \times \R^+_{\xi_d}$ are 
supplemented with the boundary conditions obtained by plugging \eqref{BKWp} in the boundary conditions of 
\eqref{0}, which yields (recall $B ={\rm d}b(0)$):
\begin{equation}
\label{4p}
\begin{split}
&{\rm (a)}\quad B\, \cU_0 =0 \, ,\\
&{\rm (b)}\quad B\, \cU_1 +\dfrac{1}{2} \, {\rm d}^2b(0) \cdot (\cU_0,\cU_0) =G(t,y,\theta_0) \, ,\\
\end{split}
\end{equation}
where functions on the left hand side of \eqref{4p} are evaluated at $x_d=\xi_d=0$. In order to get 
$u_\eps|_{t<0}=0$, as required in \eqref{0}, we also look for profiles ${\mathcal U}_k$ that vanish for $t<0$.

The construction of profiles in the wavetrain setting was accomplished by first assuming that solutions exist within 
the class of periodic functions of $\theta$, and then using that assumption to actually construct periodic solutions. 
The construction of profiles in the pulse setting has the difficulty that it is not so clear at first in what function space(s) 
solutions should be sought. Construction of each successive pulse corrector involves an additional integration over a 
noncompact set. Thus, each corrector $\cU_j$ is ``worse" than the previous one $\cU_{j-1}$, and successive correctors 
must be sought in successively larger spaces. In fact we will find that correctors beyond $\cU_2$ are useless; they 
grow at least linearly in $(\theta_0,\xi_d)$, and are thus too large to be considered correctors.

We now define spaces $\cV_F\subset\cV_H\subset \cC^1_b$, where $\cC^1_b$ is the space of $\cC^1$ functions 
$K(t,x,\theta_0,\xi_d)$, valued in $\R^3$, and bounded with their first-order derivatives. The variables $(t,x)$ lie in 
$(-\infty,T] \times \R^d_+$ while the variables $(\theta_0,\xi_d)$ lie in $\R \times \R^+$. In the subsequent discussion 
we will assume $\cU_0\in\cV_F$, $\cU_1\in \cV_H$, and $\cU_2\in \cC^1_b$ are solutions to \eqref{3p}, \eqref{4p}, 
and then construct profiles with those properties.

\begin{defn}
\label{functions}
(a)  Let $\cV_F$ denote the space of functions
\begin{equation*}
F(t,x,\theta_0,\xi_d)=\sum_{i=1}^3 F_i(t,x,\theta_0,\xi_d) \, A_d(0) \, r_i \, ,
\end{equation*}
where each function $F_i$ is a finite sum of real-valued functions of the form
\begin{align}
\label{aa}
\begin{split}
&f(t,x,\theta_0+\uomega_k \, \xi_d) \, ,\qquad 
f(t,x,\theta_0+\uomega_\ell \, \xi_d) \, g(t,x,\theta_0+\uomega_m \, \xi_d) \, ,\\
&\text{ and } \qquad f(t,x,\theta_0+\uomega_p \, \xi_d) \, g(t,x,\theta_0+\uomega_q \, \xi_d) 
\, h(t,x,\theta_0+\uomega_r \, \xi_d) \, ,
\end{split}
\end{align}
where the indices $k,\ell,\dots,r$ lie in $\{1,2,3\}$.\footnote{Elements of $\cV_F$ with terms involving no triple products 
were called functions of type $F$ in \cite{CW4,CW5}.} The functions $f(t,x,\theta)$, $g(t,x,\theta)$, etc., in \eqref{aa} 
are $\cC^1$ and decay with their first-order partials at the rate $O(\langle\theta\rangle^{-2})$ uniformly with respect 
to $(t,x)$. We refer to these functions as the ``constituent functions"  of $F \in \cV_F$.

(b) Define a \emph{transversal interaction integral} to be a function $I^i_{\ell,m}(t,x,\theta_0,\xi_d)$ of the form
\begin{equation}
\label{a0}
I^i_{\ell,m}(t,x,\theta_0,\xi_d) =-\int_{\xi_d}^{+\infty} \sigma(t,x,\theta_0+\uomega_i \, \xi_d 
+(\uomega_\ell-\uomega_i)\, s) \, \tau(t,x,\theta_0+\uomega_i \, \xi_d +(\uomega_m-\uomega_i)\, s) \, {\rm d}s \, ,
\end{equation}
where $\uomega_\ell$, $\uomega_m$, and $\uomega_i$ are mutually distinct. The functions $\sigma(t,x,\theta)$, 
$\tau(t,x,\theta)$ are real-valued, $\cC^1$ and decay with their first-order partials at the rate $O(\langle\theta\rangle^{-3})$ 
uniformly with respect to $(t,x)$.

(c) Let $\cV_H$ denote the space of functions 
\begin{equation*}
H(t,x,\theta_0,\xi_d)=\sum_{i=1}^3 H_i(t,x,\theta_0,\xi_d) \, A_d(0) \, r_i \, ,
\end{equation*}
where each $H_i$ is the sum of an element of $\cV_F$ plus a finite sum of terms of the form 
\begin{equation}
\label{a1}
I^i_{\ell,m} (t,x,\theta_0,\xi_d) \quad \text{ or } \quad \alpha(t,x,\theta_0+\uomega_k \, \xi_d) \, 
J^n_{p,q}(t,x,\theta_0,\xi_d) \, ,
\end{equation}
where $I^i_{\ell,m}$ and $J^n_{p,q}$ are transversal interaction integrals and the indices $i,l,\dots,q$ lie in $\{1,2,3\}$. 
The function $\alpha(t,x,\theta)$ is real-valued, $\cC^1$ and decays with its first-order partials at the rate 
$O(\langle\theta\rangle^{-3})$ uniformly with respect to $(t,x)$.

(d) For $H\in\cV_H$ we can write $H=H_F+H_{I}$,  where $H_F\in\cV_F$ and $H_I\notin \cV_F$ has components 
in the span of terms of the form \eqref{a1}. The ``constituent functions" of $H$ include those of $H_F$ as well as 
the functions like $\alpha$, $\sigma$, $\tau$ as in \eqref{a0}, \eqref{a1} which constitute $H_I$.
\end{defn}

\begin{rem}\label{a2}
The same spaces of functions $\cV_F$ and $\cV_H$ are obtained if one begins by writing
\begin{equation*}
F(t,x,\theta_0,\xi_d) =\sum^3_{i=1} \widetilde{F}_i(t,x,\theta_0,\xi_d) \, r_i \, ,\quad 
H(t,x,\theta_0,\xi_d) =\sum^3_{i=1} \widetilde{H}_i(t,x,\theta_0,\xi_d) \, r_i \, ,
\end{equation*}
and imposes the conditions in (a) (resp. (c)) on the $\widetilde{F}_i$ (resp. $\widetilde{H}_i)$.
\end{rem}

\subsection{Averaging and solution operators}

\emph{\quad}To construct the profiles $\cU_0, \cU_1,\cU_2$, we must solve equations of the form 
$\cL(\partial_{\theta_0},\partial_{\xi_d}) \, \cU=H$, see \eqref{3p}, where $H$ lies in $\cV_H$ and 
sometimes in $\cV_F$. In this section we define averaging operators $E_P$ and $E_Q$ and a 
solution operator $R_\infty$ (involving integration on a noncompact set) that provide a systematic 
way to study such equations. Henceforth we write $\cL:=\cL(\partial_{\theta_0},\partial_{\xi_d})$. The 
following simple lemma  implies the existence of most of the limits and integrals that appear below.

\begin{lem}\label{integrals}
Let $\sigma(t,x,\theta)$, $\tau(t,x,\theta)$ be continuous functions such that 
\begin{equation*}
|\sigma(t,x,\theta)| +|\tau(t,x,\theta)| \le C \, \langle\theta\rangle^{-N} \, ,
\end{equation*}
for some $N \ge 2$, uniformly with respect to $(t,x)$. Let $i,\ell,m,q$ lie in $\{ 1,2,3 \}$ and suppose $i,\ell,m$ 
are mutually distinct. Then the following estimates hold true uniformly with respect to all parameters:
\begin{align*}
&(a) \, \int^{+\infty}_{\xi_d} |\sigma(t,x,\theta_0+\uomega_i \, \xi_d +(\uomega_\ell-\uomega_i) \, s)| \, 
{\rm d}s \lesssim 1 \, ,\\
&(b) \, \int^{+\infty}_{\xi_d} |\sigma(t,x,\theta_0+\uomega_i \, \xi_d +(\uomega_\ell-\uomega_i) \, s) \, 
\tau(t,x,\theta_0+\uomega_i \, \xi_d +(\uomega_m-\uomega_i) \, s)| \, {\rm d}s \lesssim 
\langle \xi_d \rangle^{-N+1} \, ,\\
&(c) \text{ for } N \geq 3\, , \, \int^{+\infty}_{\xi_d} \! \! \int^{+\infty}_{s} 
|\sigma(t,x,\theta_0+\uomega_q \, \xi_d +(\uomega_i-\uomega_q) \, s +(\uomega_\ell-\uomega_i)\, r) \cdot \\
&\qquad \qquad \qquad \qquad 
\tau(t,x,\theta_0+\uomega_q \, \xi_d +(\uomega_i-\uomega_q) \, s +(\uomega_m-\uomega_i) \, r)| \, 
{\rm d}r \, {\rm d}s \lesssim \langle \xi_d \rangle^{-N+2} \, .
\end{align*}
\end{lem}

\begin{proof}
Part (a) is immediate. To prove (b), we use the Peetre inequality:
\begin{multline*}
\int^{+\infty}_{\xi_d} |\sigma(t,x,\theta_0+\uomega_i \, \xi_d +(\uomega_\ell-\uomega_i) \, s) \, 
\tau(t,x,\theta_0+\uomega_i \, \xi_d +(\uomega_m-\uomega_i) \, s)| \, {\rm d}s \\
\lesssim \int^{+\infty}_{\xi_d} \langle \theta_0+\uomega_i \, \xi_d +(\uomega_\ell-\uomega_i) \, s \rangle^{-N} 
\, \langle \theta_0+\uomega_i \, \xi_d +(\uomega_m-\uomega_i) \, s \rangle^{-N} \, {\rm d}s 
\lesssim \int^{+\infty}_{\xi_d} \langle (\uomega_\ell-\uomega_m) \, s \rangle^{-N} \, {\rm d}s \, ,
\end{multline*}
and the result follows (recall $\xi_d \ge 0$). Part (c) is also proved by using the Peetre inequality:
\begin{multline*}
\int^{+\infty}_{\xi_d} \! \! \int^{+\infty}_{s} 
|\sigma(t,x,\theta_0+\uomega_q \, \xi_d +(\uomega_i-\uomega_q) \, s +(\uomega_\ell-\uomega_i)\, r) \cdot \\
\tau(t,x,\theta_0+\uomega_q \, \xi_d +(\uomega_i-\uomega_q) \, s +(\uomega_m-\uomega_i) \, r)| \, 
{\rm d}r \, {\rm d}s \lesssim \int^{+\infty}_{\xi_d} \! \! \int^{+\infty}_{s} 
\langle (\uomega_\ell-\uomega_m) \, r \rangle^{-N} \, {\rm d}r \, {\rm d}s \lesssim \langle \xi_d \rangle^{-N+2} \, .
\end{multline*}
%
%
\end{proof}

\noindent The definition of $\cV_H$ and Lemma \ref{integrals} imply that the limits and integrals in the next 
definition are well-defined.

\begin{defn}[$E_P$, $E_Q$, $R_\infty$]
\label{ops}
For $H\in \cV_H$, define averaging operators
\begin{align*}
&E_Q \, H(t,x,\theta_0,\xi_d) :=\sum_{j=1}^3 \left( \lim_{T\to\infty} \dfrac{1}{T} \, \int^T_0 
\ell_j \cdot H(t,x,\theta_0 +\uomega_j \, (\xi_d-s),s) \, {\rm d}s \right) \, A_d(0) \, r_j \, ,\\
&E_P \, H(t,x,\theta_0,\xi_d) :=\sum_{j=1}^3 \left( \lim_{T\to\infty} \dfrac{1}{T} \, \int^T_0 
\ell_j \cdot A_d(0) \, H(t,x,\theta_0 +\uomega_j \, (\xi_d-s),s) \, {\rm d}s \right) \, r_j \, .
\end{align*}
For $H\in \cV_H$ such that $E_Q \, H=0$, define the solution operator
\begin{equation*}
R_\infty \, H(t,x,\theta_0,\xi_d) :=-\sum_{j=1}^3 \left(\int_{\xi_d}^{+\infty} \ell_j \cdot 
H(t,x,\theta_0 +\uomega_j \, (\xi_d-s),s) \, {\rm d}s \right) \, r_j \, .
\end{equation*}
\end{defn}

\begin{rem}\label{a4}
(a) Suppose $F \in \cV_F$ reads
\begin{equation*}
F(t,x,\theta_0,\xi_d) =\sum_{i=1}^3 F_i(t,x,\theta_0,\xi_d) \, A_d(0) \, r_i \, ,
\end{equation*}
where each $F_i$ has the form
\begin{equation*}
F_i(t,x,\theta_0,\xi_d) =\sum_{k=1}^3 a^i_k \, f^i_k(t,x,\theta_0 +\uomega_k \, \xi_d) 
+\sum_{\ell,m=1}^3 b^i_{\ell,m} \, g^i_\ell (t,x,\theta_0 +\uomega_\ell \, \xi_d) \, 
h^i_m(x,\theta_0 +\uomega_m \, \xi_d),\\
\end{equation*}
Then
\begin{equation*}
E_Q \, F =\sum^3_{i=1} \Big( a^i_i \, f^i_i(t,x,\theta_0 +\uomega_i \, \xi_d) 
+b^i_{i,i} \, g^i_i(t,x,\theta_0 +\uomega_i \, \xi_d) \, h^i_i(t,x,\theta_0 +\uomega_i \, \xi_d) \Big) \, A_d(0) \, r_i \, .
\end{equation*}
The obvious analogues hold when $F\in\cV_F$ involves triple products, or when $E_P$ is used in place of $E_Q$.

(b) Let $H\in\cV_H$ and write $H=H_F+H_{I}$,  where $H_F\in\cV_F$ and $H_I\notin \cV_F$ as in Definition 
\ref{functions}(d). Then Lemma \ref{integrals} implies $E_Q \, H_I =E_P \, H_I=0$.
\end{rem}

\noindent The following Proposition gives the main properties of the operators $E_P$, $E_Q$, $R_\infty$.

\begin{prop}\label{a3}
Suppose $H\in\cV_H$ and recall the notation $\cL=\cL(\partial_{\theta_0},\partial_{\xi_d})$. Then there holds:

(a) $E_Q \, \cL \, H =\cL \, E_P \, H=0$.  

(b) If $E_Q \, H=0$, then $R_\infty \, H$ is bounded and $\cL \, R_\infty \, H =H =(I-E_Q) \, H$. 

(c) If $E_P \, H=0$ ,then $R_\infty \, \cL \, H =H =(I-E_P) \, H$. 
\end{prop}

\begin{proof}
(a) $\cL \, E_P \, H=0$ follows directly from Remark \ref{a4} (a),(b). To show $E_Q \, \cL \, H=0$, we write 
$H=\sum^3_{i=1}H_i \, r_i$, and note that 
\begin{equation*}
\cL \, (H_i \, r_i) =(\partial_{\xi_d}-\uomega_i \, \partial_{\theta_0}) \, H_i \, A_d(0) \, r_i \, .
\end{equation*}
Thus, the integrals $\int^T_0$ in the definition of $E_Q \, \cL \, H$ can be evaluated and are uniformly bounded 
with respect to $T$. After dividing by $T$, the limit as $T$ goes to infinity vanishes.

(b) Boundedness of $R_\infty H$ follows from Lemma \ref{integrals}. Writing $H =\sum^3_{i=1}H_i \, A_d(0) \, r_i$, 
a direct computation using $L({\rm d}\varphi_i) \, r_i=0$ shows $\cL \, R_\infty \, H=H$.

(c) With $\beta=(\utau,\ueta)$, define $A(\beta)$ by $L({\rm d}\varphi_k) =A(\beta) +\uomega_k \, A_d(0)$, 
and observe that
\begin{equation}
\label{a5}
\ell_k \, A(\beta) \, r_j =-\uomega_k \, \delta_{jk} \, .
\end{equation}
Writing $H =\sum^3_{j=1} \widetilde{H}_j \,  r_j$ and using \eqref{a5}, a direct computation yields
\begin{equation*}
R_\infty \, \cL \, H =-\sum^3_{j=1} \left( \int_{\xi_d}^{+\infty} (\partial_{\xi_d}-\uomega_j \, \partial_{\theta_0}) 
\, \widetilde{H}_j (t,x,\theta_0 +\uomega_j \, (\xi_d-s),s) \, {\rm d}s \right) \, r_j =\sum_j \widetilde{H}_j \, r_j \, ,
\end{equation*}
since $\lim_{M \to+\infty} \widetilde{H}_j(t,x,\theta_0 +\uomega_j \, (\xi_d-M),M)=0$ when $E_P \, H=0$.
\end{proof}

The next proposition, which extends \cite[Proposition 1.22]{CW4}, is used repeatedly in constructing 
$\cU_0$, $\cU_1$ and $\cU_2$ hereafter.

\begin{prop}\label{properties}
Suppose $H\in\cV_H$.

(a) The equation $\cL \, \cU=H$ has a bounded $\cC^1$ solution if and only if $E_Q \, H=0$.\footnote{The proof 
actually shows that $\cL \, U=H$ has a $\cC^1$ solution that is {\sl sublinear} with respect to $\xi_d$ if and only if 
$E_Q \, H=0$.}

(b) When $E_Q \, H=0$, every bounded $\cC^1$ solution to the equation $\cL \, \cU=H$ has the form
\begin{equation*}
\cU =\sum^3_{m=1} \tau_m(t,x,\theta_0+\uomega_m \, \xi_d) \, r_m +R_\infty H \, .
\end{equation*}

(c) When $H \in \cV_F$ satisfies $E_Q \, H=0$, solutions to $\cL \, \cU=H$ satisfy
\begin{equation*}
E_P \, \cU=\sum^3_{m=1} \tau_m (t,x,\theta_0 +\uomega_m \, \xi_d) \, r_m \, ,\quad (I-E_P) \, \cU=R_\infty H \, .
\end{equation*}
\end{prop}

\begin{proof}
(a) The direction $\Leftarrow$ is given by Proposition \ref{a3}(b). ($\Rightarrow$) Suppose there is a bounded 
$\cC^1$ solution to $\cL \, U=H$, and write $H=H_F+H_I$ as in Remark \ref{a4}(b). Since $E_Q \, H_I=0$, 
there is a bounded $\cC^1$ solution to $\cL \, U=H_I$, so we conclude there is a bounded $\cC^1$ solution 
to $\cL \, U=H_F$, and similarly to $\cL \, U=E_Q H_F$. From the explicit form of $E_Q \, H_F$ given in Remark 
\ref{a4}(a), we see that there is no bounded $\cC^1$ solution of $\cL \, U=E_Q \, H_F$ if $E_Q \, H_F\neq 0$ 
(for otherwise solutions display a linear growth with respect to $\xi_d$). Thus $E_Q \, H_F=0$ and hence 
$E_Q \, H=0$.

(b) By Proposition \ref{a3}(b), $R_\infty \, H$ is a bounded $\cC^1$ solution of $\cL \, U=H$; moreover, the 
general $\cC^1$ solution of $\cL \, U=0$ has the form $\sum^3_{m=1} \tau_m(t,x,\theta_0+\uomega_m \, 
\xi_d) \, r_m$.

(c) Lemma \ref{integrals} implies $E_P \, R_\infty \, H=0$ when $H \in \cV_F$ and $E_Q \, H=0$.
\end{proof}

\subsection{Profile construction and proof of Theorem \ref{theopulses}}
\label{pconstruction}

\emph{\quad}  For $\Omega_T:=(-\infty,T]\times \R^{d}_x \times \R_\theta$, we define the weighted Sobolev 
spaces:
\begin{equation*}
\Gamma^k(\Omega_T):= \left\{ u \in L^2(\Omega_T) \, : \, \theta^\alpha \, \partial_{t,x,\theta}^\beta u \in 
L^2(\Omega_T) \quad \text{\rm if }Ê\alpha +|\beta| \le k \right\} \, .
\end{equation*}
This is a Hilbert space for the norm
\begin{equation*}
\| u \|_{\Gamma^k(\Omega_T)}^2 :=
\sum_{\alpha +|\beta| \le k} \|Ê\theta^\alpha \, \partial_{t,x,\theta}^\beta u \|_{L^2 (\Omega_T)}^2 \, ,
\end{equation*}
and $\Gamma^k(\Omega_T)$ is an algebra for $k>(d+2)/2$. For $L\in\N$, when $k>\frac{d+2}{2}+L+1$, 
elements of $\Gamma^k(\Omega_T)$ are $\cC^1$ and decay with their first-order partials at the rate 
$\langle \theta \rangle^{-L}$ uniformly with respect to $(t,x)$.

\begin{defn}\label{VFK}
(a)  Suppose $k>\frac{d+2}{2}+3$.  Let $\cV^k_F\subset \cV_F$ be the subspace consisting of elements whose 
constituent functions lie in $\Gamma^k(\Omega_T)$.

(b) Suppose $k>\frac{d+2}{2}+4$.  Let $\cV^k_H\subset \cV_H$ be the subspace consisting of elements whose 
constituent functions lie in $\Gamma^k(\Omega_T)$. 
\end{defn}

For $K_0>8+\frac{d+2}{2}$, we assume now that the equations \eqref{3p}, \eqref{4p} have solutions 
$\cU_0 \in \cV_F$, $\cU_1 \in \cV_H$, $\cU_2 \in \cC^1_b$ such that 
\begin{align}\label{a9}
\begin{split}
&(a)\; \cU_0 \in \cV^{K_0}_F \, , \; \cU_0 =E_P \, \cU_0 =\sum^3_{m=1} 
\sigma_m(t,x,\theta_0 +\uomega_m \, \xi_d) \, r_m \, , \text{ where } 
\underline{\sigma}_m (t,x) :=\int_\R \sigma_m(t,x,\theta_m) \, {\rm d}\theta_m =0 \, ,\\
&(b)\; E_P \, \cU_1 \in \cV_F^{K_0-3} \text{ and } (I-E_P) \, \cU_1 \in \cV_H^{K_0-2} \, ,
\end{split}
\end{align}
and then construct solutions with those properties.
\bigskip

\noindent \underline{Step 1: the leading profile ${\mathcal U}_0$ has no outgoing component.}

This step justifies one of the causality arguments used in \cite{MR}. Equation \ref{3p}(a) and Proposition 
\ref{properties}(b) imply that the expression of ${\mathcal U}_0$ reduces to:
\begin{equation*}
{\mathcal U}_0(t,x,\theta_0,\xi_d) =\sum_{m=1}^3 \sigma_m (t,x,\theta_0 +\underline{\omega}_m \, \xi_d) \, r_m \,
\end{equation*}
for functions $\sigma_m$ to be determined.  The last variable of $\sigma_m$ is denoted $\theta_m$ in 
what follows.

Since $\cU_1$ is a bounded solution to $\cL \, \cU_1=\cF_0$, where $\cF_0=-L(\partial)\cU_0-\cM(\cU_0,\cU_0) 
\in \cV_F^{K_0-1}$, Proposition \ref{properties}(a) implies $E_Q \, \cF_0=0$; that is,
\begin{equation}
\label{Burgersmp}
\partial_t \sigma_m +{\bf v}_m \cdot \nabla_x \sigma_m +c_m \, \sigma_m \, \partial_{\theta_m} \sigma_m =0 \, ,\quad 
m=1,2,3 \, ,\quad c_m :=\dfrac{\partial_j \varphi_m \, \ell_m \, ({\rm d}A_j(0) \cdot r_m) \, r_m}{\ell_m \, r_m} \, .
\end{equation}
Since $\varphi_2$ is outgoing, this implies $\sigma_2 \equiv 0$, and the boundary condition \eqref{4p}(a) 
gives, as in Paragraph \ref{sect2example}, the existence of a scalar function $a$ such that
\begin{equation}
\label{defa}
\sigma_1 (t,y,0,\theta_0) \, r_1 =a(t,y,\theta_0) \, e_1 \, ,\quad \sigma_3 (t,y,0,\theta_0) \, r_3 =a(t,y,\theta_0) \, e_3 \, .
\end{equation}
\bigskip

\noindent\underline{Step 2: showing $(I-E_P) \, \cU_1 \in \cV_H^{K_0-2}$.}

At this stage, we know that the leading profile ${\mathcal U}_0$ reads
\begin{equation}
\label{decompU0p2}
{\mathcal U}_0(t,x,\theta_0,\xi_d) =\sigma_1 (t,x,\theta_0 +\underline{\omega}_1 \, \xi_d) \, r_1 
+\sigma_3 (t,x,\theta_0 +\underline{\omega}_3 \, \xi_d) \, r_3 \, ,
\end{equation}
where $\sigma_1,\sigma_3$ satisfy \eqref{Burgersmp} and their traces satisfy \eqref{defa}. We thus compute
\begin{equation}
\label{decompF0}
{\mathcal F}_0 =-L(\partial) \, (\sigma_1 \, r_1+\sigma_3 \, r_3) 
-\sigma_1 \, \partial_{\theta_1} \sigma_1 \, R_{1,1} -\sigma_3 \, \partial_{\theta_3} \sigma_3 \, R_{3,3} 
-\sigma_3 \, \partial_{\theta_1} \sigma_1 \, R_{1,3} -\sigma_1 \, \partial_{\theta_3} \sigma_3 \, R_{3,1} \, ,
\end{equation}
with
\begin{equation*}
\forall \, m_1,m_2=1,3 \, ,\quad R_{m_1,m_2} := \partial_j \varphi_{m_1} \, ({\rm d}A_j(0) \cdot r_{m_2}) \, r_{m_1} \, ,
\end{equation*}
and the functions $(\sigma_1,\partial_{\theta_1} \sigma_1)$, resp. $(\sigma_3,\partial_{\theta_3} \sigma_3)$, 
in \eqref{decompF0} are evaluated at $(t,x,\theta_0 +\underline{\omega}_1 \, \xi_d)$, resp. $(t,x,\theta_0 
+\underline{\omega}_3 \, \xi_d)$. By Proposition \ref{properties}(b), we have
\begin{align*}
\begin{split}
&\cU_1 =\sum^3_{m=1} \tau_m(t,x,\theta_0 +\uomega_m \, \xi_d) \, r_m +R_\infty \, \cF_0 \, ,\\
&(I-E_P) \, \cU_1 =R_\infty \, \cF_0 =-\sum^3_{m=1} \left( \int_{\xi_d}^{+\infty} 
F_m(t,x,\theta_0+\uomega_m \, (\xi_d-s),s) \, {\rm d}s \right) \, r_m \, .
\end{split}
\end{align*}
Since we have $E_Q \, \cF_0=0$, the latter integrand $F_m(t,x,\theta_0+\uomega_m \, (\xi_d-s),s)$ reads
\begin{align}\label{a6}
\begin{split}
&\quad \sum_{k\neq m} V^m_k \, \sigma_k (t,x,\theta_0+\uomega_m \, \xi_d +(\uomega_k-\uomega_m) \, s) +\\
&\quad \sum_{k\neq m} c^m_k \, \sigma_k (t,x,\theta_0+\uomega_m \, \xi_d +(\uomega_k-\uomega_m) \, s) \, 
\partial_{\theta_k} \sigma_k (t,x,\theta_0+\uomega_m \, \xi_d +(\uomega_k-\uomega_m) \, s) +\\
&\quad \sum_{\ell \neq m} d^m_{m,\ell} \, \sigma_m(t,x,\theta_0+\uomega_m \, \xi_d) \, 
\partial_{\theta_\ell} \sigma_\ell (t,x,\theta_0+\uomega_m \, \xi_d +(\uomega_\ell-\uomega_m) \, s) +\\
&\quad \sum_{\ell \neq m} d^m_{\ell,m} \, \sigma_\ell (t,x,\theta_0+\uomega_m \, \xi_d +(\uomega_\ell-\uomega_m) \, s) 
\, \partial_{\theta_m} \sigma_m (t,x,\theta_0+\uomega_m \, \xi_d) +\\
&\quad \sum_{\ell \neq k,\ell \neq m,k\neq m} d^m_{\ell,k} \, 
\sigma_\ell(t,x,\theta_0+\uomega_m \, \xi_d+(\uomega_\ell-\uomega_m) \, s) \, 
\partial_{\theta_k} \sigma_k(t,x,\theta_0+\uomega_m \, \xi_d +(\uomega_k-\uomega_m) \, s) \, ,
\end{split}
\end{align}
where the $c^m_k$, $d^m_{\ell,k}$ are real constants, and  $V^m_k$ is the tangential vector field given by
\begin{equation*}
V^m_k \, \sigma_k :=\ell_m \, L(\partial) \, (\sigma_k r_k) \, .
\end{equation*}
After integrating with respect to $s$ on $[\xi_d,+\infty[$ the second and third lines of \eqref{a6}, constituent functions 
in $\cV^{K_0}_F$ are obtained. The first and fourth lines can be integrated using the moment zero property of the 
$\sigma_m$'s to yield constituent functions in $\cV^{K_0-2}_F$ and $\cV^{K_0-1}_F$ respectively.\footnote{Without 
this moment zero property, these integrals and thus $\cU_1$ would be no better than bounded; consequently, 
$\cU_2$ would be unbounded.} The integral of the fifth line is a linear combination of  transversal interaction 
integrals with constituent functions in $\cV_F^{K_0-1}$.
\bigskip

\noindent \underline{Step 3: equations for $E_P \, \cU_1$.}

Since $\cU_2\in \cC^1_b$ is a solution to $\cL \, \cU_2=\cF_1$, where 
\begin{equation}
\label{F1}
{\mathcal F}_1 := -\Big[ L(\partial) \, \cU_1 +\cM (\cU_0,\cU_1) +\cM (\cU_1,\cU_0) 
+\cN_1 (\cU_0,\cU_0) +\cN_2 (\cU_0,\cU_0,\cU_0) \Big] \in \cV_H^{K_0-4} \, ,
\end{equation}
Proposition \ref{properties}(a) implies $E_Q \, \cF_1=0$.  We rewrite this and include the boundary condition 
on $\cU_1$ to obtain
\begin{align}\label{a7}
\begin{split}
&(a)\; E_Q \, \Big[ L(\partial) \, E_P \, \cU_1 +\cM(\cU_0,E_P \, \cU_1) +\cM(E_P \, \cU_1,\cU_0) \Big]=\\
&\qquad \qquad -E_Q \, \Big[ L(\partial) \, (I-E_P) \, \cU_1 +\cM(\cU_0,(I-E_P) \, \cU_1) 
+\cM((I-E_P) \, \cU_1,\cU_0) +\cN_1+\cN_2 \Big] \, ,\\
&(b)\; B \, E_P \, \cU_1 =G -\dfrac{1}{2} \, {\rm d}^2b(0) \cdot (\cU_0,\cU_0) -B \, (I-E_P) \, \cU_1 \, ,\quad 
\text{ on } x_d=\xi_d=0 \, .
\end{split}
\end{align}
Letting $V:=(I-E_P) \, \cU_1$ (which has been constructed from $\cU_0$ in Step 2 above), and decomposing 
$V=V_F+V_I$ as in Definition \ref{functions}(d), we can use Lemma \ref{integrals}(c) to see that
\begin{equation*}
E_Q \, \Big[ L(\partial) \, V_I +\cM(\cU_0,V_I) +\cM(V_I,\cU_0) \Big] =0 \, .
\end{equation*}
Thus, \eqref{a7}(a) simplifies to 
\begin{multline}
\label{a8}
E_Q \, \Big[ L(\partial) \, E_P \, \cU_1 +\cM(\cU_0,E_P \, \cU_1) +\cM(E_P \, \cU_1,\cU_0) \Big] \\
=-E_Q \, \Big[ L(\partial) \, V_F +\cM(\cU_0,V_F) +\cM(V_F,\cU_0) +\cN_1 +\cN_2 \Big] \in \cV_F^{K_0-3} \, .
\end{multline}
The components of $E_P \, \cU_1$ will be determined from equations \eqref{a8} and \eqref{a7}(b).
\bigskip

\noindent \underline{Step 4: determining the component of  $E_P \, {\mathcal U}_1$ on $r_2$.}

We now show $\tau_2=0$, thereby justifying another causality argument in \cite{MR}. Since $\cU_0$ is 
"purely incoming" (no $(\theta_0 +\uomega_2 \, \xi_d)$-dependence), using Remark \ref{a4}(a), we see 
that the component on $r_2$ of the right side of \eqref{a8} is zero. This is rather clear for the terms 
$E_Q \, (\cN_1 +\cN_2)$, because $\cN_1 +\cN_2$ is a linear combination of terms of the form
\begin{equation*}
f_1(t,x,\theta_0 +\uomega_1 \, \xi_d) \, ,\quad f_3(t,x,\theta_0 +\uomega_3 \, \xi_d) \, ,\quad 
g(t,x,\theta_0 +\uomega_1 \, \xi_d) \, h(t,x,\theta_0 +\uomega_3 \, \xi_d) \, ,
\end{equation*}
and it is also true for the terms arising in $E_Q \, [L(\partial) \, V_F +\cM(\cU_0,V_F) +\cM(V_F,\cU_0)]$ 
after examining the form of $V_F =(R_\infty \, \cF_0)_F$.

Similarly, the component on $r_2$ of $E_Q \, [\cM(\cU_0,E_P \, \cU_1) +\cM(E_P \, \cU_1,\cU_0)]$ is zero, 
whatever the value of the constituent functions $\tau_m$ in $E_P \, \cU_1$ (which remains to be determined). 
This forces the function $\tau_2$ in the decomposition of $E_P \, \cU_1$ to satisfy the homogeneous transport 
equation
\begin{equation*}
(\partial_t +{\bf v}_2 \cdot \nabla_x) \, \tau_2 =0 \, ,
\end{equation*}
Since $\varphi_2$ is an outgoing phase, $\tau_2$ is identically zero.

\begin{rem}
The justification of $\tau_2 =0$  relies on the assumption that \eqref{3p}(c) admits a bounded solution 
$\cU_2$. Boundedness of $\cU_2$ makes the expression $\eps^3 \, \cU_2$ meaningful as a corrector to the 
approximate solution $\eps \, \cU_0 +\eps^2 \, \cU_1$. However, we shall see in Appendix \ref{appB} that 
assuming boundedness of $\cU_2$ has a major consequence on the leading order profile $\cU_0$. In 
particular, the governing equation \eqref{eqMachStems'} below for the evolution of $a$ on the boundary 
will not coincide with the equation obtained by considering \eqref{eqa2} in Part \ref{part1} in the regime 
$\Theta \rightarrow +\infty$.
\end{rem}

\noindent \underline{Step 5: the Mach stem equation for $a$.}

From the previous step of the analysis, the trace of the first corrector $\cU_1$ satisfies
\begin{equation*}
\cU_1 (t,y,0,\theta_0,0) = \star \, r_1 -\int_0^{+\infty} \ell_2 \, {\mathcal F}_0 
(t,y,0,\theta_0 -\underline{\omega}_2 \, X,X) \, {\rm d}X \, r_2 +\star \, r_3 \, ,
\end{equation*}
where $\star$ denotes a coefficient whose expression is not useful for what follows, and ${\mathcal F}_0$ is given 
by \eqref{decompF0}. Plugging the latter expression in \eqref{4p}(b) and applying the row vector $\underline{b}$, 
we get
\begin{equation}
\label{ampli1}
-\underline{b} \, B \, r_2 \, \int_0^{+\infty} \ell_2 \, {\mathcal F}_0 (t,y,0,\theta_0 -\underline{\omega}_2 \, X,X) 
\, {\rm d}X +\dfrac{1}{2} \, \underline{b} \, {\rm d}^2b(0) \cdot (e,e) \, a^2 =\underline{b} \, G \, ,
\end{equation}
which is the solvability condition for $E_P \, \cU_1|_{x_d=\xi_d=0}$ in \eqref{a7}. It remains to differentiate 
\eqref{ampli1} with respect to $\theta_0$ and to identify the first term on the left hand side of \eqref{ampli1}. 
More precisely, we compute
\begin{align*}
\ell_2 \, {\mathcal F}_0 (t,y,0,\theta_0,\xi_d) =&-\ell_2 \, L_{\rm tan}(\partial) \, 
(a(t,y,\theta_0+\underline{\omega}_1 \, \xi_d) \, e_1 +a(t,y,\theta_0+\underline{\omega}_3 \, \xi_d) \, e_3) \\
&-\dfrac{1}{2} \, \ell_2 \, E_{1,1} \, \partial_{\theta_0} (a^2) (t,y,\theta_0+\underline{\omega}_1 \, \xi_d) 
-\dfrac{1}{2} \, \ell_2 \, E_{3,3} \, \partial_{\theta_0} (a^2) (t,y,\theta_0+\underline{\omega}_3 \, \xi_d) \\
&-\ell_2 \, E_{1,3} \, (\partial_{\theta_0} a) (t,y,\theta_0+\underline{\omega}_1 \, \xi_d) \, 
a(t,y,\theta_0+\underline{\omega}_3 \, \xi_d) \\
&-\ell_2 \, E_{3,1} \, a(t,y,\theta_0+\underline{\omega}_1 \, \xi_d) \, 
(\partial_{\theta_0} a) (t,y,\theta_0+\underline{\omega}_3 \, \xi_d) \, ,
\end{align*}
with $E_{1,1},E_{3,3},E_{1,3},E_{3,1}$ as in \eqref{defEm1m2}, and
\begin{equation*}
L_{\rm tan}(\partial) :=\partial_t +\sum_{j=1}^{d-1} A_j(0) \, \partial_j \, .
\end{equation*}
Using the expression of the matrices $R_1,R_3$ in Paragraph \ref{notation}, we thus find that \eqref{ampli1} 
reduces to
\begin{equation}
\label{eqMachStems'}
\upsilon \, \partial_{\theta_0} (a^2) -X_{\rm Lop} a +\partial_{\theta_0} Q_{\rm pul}[a,a] 
=\underline{b} \, \partial_{\theta_0} G \, ,
\end{equation}
with $\upsilon$ as in \eqref{defupsilon}, $X_{\rm Lop}$ as in \eqref{defXLop}, and\footnote{The variables $(t,y)$ enter 
as parameters in the definition of $Q_{\rm pul}$ so we omit them.}
\begin{multline}
\label{defQpul}
Q_{\rm pul}[a,\widetilde{a}] (\theta_0) := (\underline{b} \, B \, r_2) \, \ell_2 \, E_{1,3} \, \int_0^{+\infty} 
\partial_{\theta_0} a (\theta_0+(\underline{\omega}_1-\underline{\omega}_2)\, X) \, 
\widetilde{a}(\theta_0+(\underline{\omega}_3-\underline{\omega}_2)\, X) \, {\rm d}X \\
+(\underline{b} \, B \, r_2) \, \ell_2 \, E_{3,1} \, \int_0^{+\infty} a(\theta_0+(\underline{\omega}_1-\underline{\omega}_2)\, X) \, 
\partial_{\theta_0} \widetilde{a} (\theta_0+(\underline{\omega}_3-\underline{\omega}_2)\, X) \, {\rm d}X \, .
\end{multline}
\bigskip

\noindent \underline{Step 6:  completing the construction of $\cU_0$, $\cU_1$, $\cU_2$.}

It is proved in Corollary \ref{cor3} of Section \ref{sect7} that, with $K_0,K_1 \in \N$ and $G$ as in Theorem 
\ref{theopulses},  there exists $T>0$ and  a unique solution
\begin{equation*}
a \in 	\cap_{\ell =0}^{K_0} \, {\mathcal C}^\ell ((-\infty,T];\Gamma^{K_1-1-\ell}(\R^d_{y,\theta})) \, ,
\end{equation*}
to \eqref{eqMachStems'}, \eqref{defQpul}. From \eqref{defa}, this determines the boundary data of $\sigma_1$, 
$\sigma_3$. Corollary \ref{cor4} of section \ref{sect7} yields $\sigma_1, \sigma_3 \in \Gamma^{K_0} (\Omega_T)$ 
satisfying \ref{Burgersmp} (up to restricting the time $T>0$). Moreover, $a$ and thus $\sigma_1,\sigma_3$ are 
shown there to have moment zero. That completes the construction of $\cU_0$ with the properties in \eqref{a9}.

Assuming these results for now, we complete the construction of $\cU_1$ and $\cU_2$. With $\cU_0$ determined, 
$(I-E_P) \, \cU_1\in \cV_H^{K_0-2}$ is now constructed as in Step 2. To determine $E_P \, \cU_1$ we return to 
\eqref{a8} and \eqref{a7}(b), noting that the right side of \eqref{a8} is now determined. Writing $(I-E_P) \, \cU_1 
=V_F+V_I$ as before, we have\footnote{We use self-explanatory notation here.} $V_F|_{x_d=\xi_d=0} \in 
\Gamma^{K_0-3}$, and the same holds for the trace of $V_I$ as a consequence of  Corollary \ref{corproof}. 
The right hand side of \eqref{a7}(b) satisfies the required solvability condition, so \eqref{a7}(b) uniquely 
determines the trace $E_P \, \cU_1|_{x_d=\xi_d=0} \in \Gamma^{K_0-3}$, taking its value in $\check \E^s (\utau,\ueta)$; 
recall \eqref{decomposition3}. Equations \eqref{a8} and \eqref{a7}(b) determine decoupled transport equations 
for the components $\tau_1$, $\tau_3$ of $E_P \, \cU_1$, and we obtain $\tau_1,\tau_3 \in \Gamma^{K_0-3}$ 
and hence $E_P \, \cU_1\in \cV_F^{K_0-3}$.

The profile $\cU_2$ satisfies $\cL \, \cU_2 =\cF_1$, where $\cF_1$ as in \eqref{F1} satisfies $E_Q \, \cF_1=0$. 
Thus, Proposition \ref{properties} yields a solution $\cU_2 =R_\infty \, \cF_1 \in \cC^1_b$.

Apart from the results proved in section \ref{sect7} that were used in this step, this completes the proof of 
Theorem \ref{theopulses} in the $3\times 3$ strictly hyperbolic case. The profile $\cU_0$ satisfies \eqref{3p}(a), 
\eqref{4p}(a), and the correctors $\cU_1$, $\cU_2$ satisfy \eqref{3p}(b), (c) and \eqref{4p}(b).

\section{Analysis of the amplitude equation}
\label{sect7}

\subsection{Preliminary reductions}

\emph{\quad} Our goal in this section is to prove a well-posedness result for the ``Mach stem equation" 
\eqref{eqMachStems'}. We focus on the case of $3 \times 3$ strictly hyperbolic systems, and leave the 
minor modifications for the extension to $N \times N$ systems to the interested reader (see Paragraph 
\ref{extensionpulses} for the derivation of the corresponding amplitude equation). Up to dividing by nonzero 
constants, and using the shorter notation $\theta$ instead of $\theta_0$, Equation \eqref{eqMachStems'} 
takes the form
\begin{equation}
\label{eqMachStems1}
\partial_t a +{\bf w} \cdot \nabla_y a +c \, a \, \partial_\theta a +\partial_\theta Q_{\rm pul}[a,a] =g \, ,
\end{equation}
with ${\bf w} \in \R^{d-1}$, $c \in \R$, and the quadratic operator $Q_{\rm pul}$ is defined by
\begin{multline*}
Q_{\rm pul}[a,\widetilde{a}] (\theta) := \mu_1 \, \int_0^{+\infty} 
\partial_\theta a (\theta+(\underline{\omega}_1-\underline{\omega}_2)\, X) \, 
\widetilde{a}(\theta+(\underline{\omega}_3-\underline{\omega}_2)\, X) \, {\rm d}X \\
+\mu_3 \, \int_0^{+\infty} a(\theta+(\underline{\omega}_1-\underline{\omega}_2)\, X) \, \partial_\theta \widetilde{a} 
(\theta+(\underline{\omega}_3-\underline{\omega}_2)\, X) \, {\rm d}X \, ,
\end{multline*}
where $\mu_1,\mu_3 \in \R$ and the $\underline{\omega}_m$'s are pairwise distinct. The latter operator 
only acts on the $\theta$-variable, and $(t,y)$ only enter as parameters, which we do not write for simplicity.

We first reduce the expression of $Q_{\rm pul}$ by recalling that $\underline{\omega}_1,\underline{\omega}_3$ 
are the two incoming modes while $\underline{\omega}_2$ is the outgoing mode. There is no loss of generality 
in assuming $\underline{\omega}_3>\underline{\omega}_1$. Then we define the two nonzero parameters
\begin{equation*}
\delta_1 :=\dfrac{\underline{\omega}_1-\underline{\omega}_2}{\underline{\omega}_3-\underline{\omega}_1} \, ,\quad 
\delta_3 :=\dfrac{\underline{\omega}_3-\underline{\omega}_2}{\underline{\omega}_3-\underline{\omega}_1} \, ,
\end{equation*}
that satisfy $\delta_3=1+\delta_1$. Changing variables in the expression of $Q_{\rm pul}$ and redefining the 
constants $\mu_{1,3}$, we obtain
\begin{equation*}
Q_{\rm pul}[a,\widetilde{a}] (\theta) =\mu_1 \, \int_0^{+\infty} \partial_\theta a (\theta+\delta_1 \, X) \, 
\widetilde{a}(\theta+\delta_3 \, X) \, {\rm d}X 
+\mu_3 \, \int_0^{+\infty} a(\theta+\delta_1 \, X) \, \partial_\theta \widetilde{a} (\theta+\delta_3 \, X) \, {\rm d}X \, .
\end{equation*}

For later use, we define the following bilinear operator $\F_{\rm pul}$ acting on functions that depend on the 
variable $\theta \in \R$ (whenever the formula below makes sense):
\begin{equation}
\label{defFpul}
\F_{\rm pul}(u,v) (\theta) :=\int_0^{+\infty} u(\theta+\delta_1 \, X) \, v(\theta+\delta_3 \, X) \, {\rm d}X \, .
\end{equation}
The operator $\F_{\rm pul}$ satisfies the properties:
\begin{align}
\text{\rm (Differentiation)} \quad &\partial_\theta (\F_{\rm pul}(u,v)) =\F_{\rm pul}(\partial_\theta u,v) 
+\F_{\rm pul}(u,\partial_\theta v) \, ,\label{propF1} \\
\text{\rm (Integration by parts)} \quad &\F_{\rm pul}(u,\partial_\theta v) =-\dfrac{1}{\delta_3} \, u\, v 
-\dfrac{\delta_1}{\delta_3} \, \F_{\rm pul}(\partial_\theta u,v) \, .\label{propF2}
\end{align}
Using the properties \eqref{propF1}, \eqref{propF2}, we can rewrite Equation \eqref{eqMachStems1} as
\begin{equation}
\label{eqMachStems}
\partial_t a +{\bf w} \cdot \nabla_y a +c\, a \, \partial_\theta a +\mu \, \F_{\rm pul} (\partial_\theta a,\partial_\theta a) 
=g \, ,
\end{equation}
with suitable constants that are denoted $c$ and $\mu$ for simplicity and whose exact expression is useless. 
Our goal is to solve Equation \eqref{eqMachStems} by a standard fixed point argument. The main ingredient 
in the proof is to show that the nonlinear term $\F_{\rm pul} (\partial_\theta a, \partial_\theta a)$ acts as a 
{\sl semilinear} term in a suitable scale of Sobolev regularity.

\subsection{Boundedness of the bilinear operator $\F_{\rm pul}$}

The operator $\F_{\rm pul}$ is not symmetric but changing the roles of $\delta_1$ and $\delta_3$, the roles of the first 
and second argument of $\F_{\rm pul}$ in the estimates below can be exchanged. This will be used in several places.

Let us first recall the definition of weighted Sobolev spaces:
\begin{equation*}
\Gamma^k (\R^d):= \left\{ u \in L^2(\R^{d-1}_y \times \R_\theta) \, : \, \theta^\alpha \, \partial_{y,\theta}^\beta u \in 
L^2 (\R^d) \quad \text{\rm if }Ê\alpha +|\beta| \le k \right\} \, .
\end{equation*}
This is a Hilbert space for the norm
\begin{equation*}
\| u \|_{\Gamma^k(\R^d)}^2 := \sum_{\alpha +|\beta| \le k} \|Ê\theta^\alpha \, \partial_{y,\theta}^\beta u \|_{L^2 (\R^d)}^2 \, .
\end{equation*}
Following the same integration by parts arguments as in \cite[Proposition 3.3]{CW4}, there holds

\begin{lem}
For all integer $k$, the space $\Gamma^k(\R^d)$ coincides with
\begin{equation*}
\left\{ u \in H^k(\R^{d-1}_y \times \R_\theta) \, : \, \theta^k \, u \in L^2(\R^d) \right\} \, ,
\end{equation*}
and the norm of $\Gamma^k(\R^d)$ is equivalent to the norm
\begin{equation*}
\|Ê\theta^k \, u \|_{L^2(\R^d)} +\| u \|_{H^k (\R^d)} \, .
\end{equation*}
\end{lem}

\noindent Our main boundedness result for the operator $\F_{\rm pul}$ reads as follows.

\begin{prop}
\label{propFpul}
Let $k_0$ denote the smallest integer satisfying $k_0>(d+1)/2$. Then for all $k \ge 2\, k_0+1$, there 
exists a constant $C_k$ satisfying
\begin{equation}
\label{estimpropF}
\forall \, u,v \in \Gamma^k (\R^d)\, ,\quad \| \F_{\rm pul}(\partial_\theta u,\partial_\theta v) \|_{\Gamma^k(\R^d)} 
\le C_k\, \| u \|_{\Gamma^k(\R^d)} \, \| v \|_{\Gamma^k(\R^d)} \, .
\end{equation}
\end{prop}

The estimate \eqref{estimpropF} of Proposition \ref{propFpul} is {\sl not tame}, but it will be sufficient for our 
purpose since we shall only construct finitely many terms in the WKB expansion of the solution $u_\eps$ 
to \eqref{0} (opposite to what we did in Part \ref{part1} where we constructed approximate solutions of 
arbitrarily high order).

\begin{proof}
Let us first observe that when $u,v$ belong to the Schwartz space ${\mathcal S}(\R^d)$, $\F_{\rm pul}(u,v)$ 
also belongs to ${\mathcal S}(\R^d)$. By a density/continuity argument, we are thus reduced to proving 
the estimate \eqref{estimpropF} for $u,v \in {\mathcal S}(\R^d)$. The decay and regularity of $u,v$ will 
justify all the manipulations below.

{\bf 1)} We start with the basic $L^2$ estimate of the function $\F_{\rm pul}(u,v)$. Using Cauchy-Schwartz inequality, 
we have
\begin{align*}
\Big| \int_0^{+\infty} u(y,\theta+\delta_1 \, X) \, & v(y,\theta+\delta_3 \, X) \, {\rm d}X \Big|^2 \\
&\le \int_0^{+\infty} |u(y,\theta+\delta_1 \, X)| \, {\rm d}X \, 
\int_0^{+\infty} |u(y,\theta+\delta_1 \, X)| \, |v(y,\theta+\delta_3 \, X)|^2 \, {\rm d}X \\
&\le \dfrac{1}{|\delta_1|} \, \int_\R |u(y,\theta')| \, {\rm d}\theta' \, 
\int_\R |u(y,\theta+\delta_1 \, X)| \, |v(y,\theta+\delta_3 \, X)|^2 \, {\rm d}X \, .
\end{align*}
Integrating with respect to $(y,\theta)$, and changing variables (use $\delta_3-\delta_1=1$), we get
\begin{align*}
\| \F_{\rm pul}(u,v) \|_{L^2(\R^d)}^2 &\le \dfrac{1}{|\delta_1|} \, \int_{\R^{d-1}} \left( \int_\R |u(y,\theta)| \, {\rm d}\theta \right)^2 
\, \left( \int_\R |v(y,\theta)|^2 \, {\rm d}\theta \right) \, {\rm d}y \\
&\le \dfrac{\pi}{|\delta_1|} \, \int_{\R^{d-1}} \left( \int_\R (1+\theta^2) \, |u(y,\theta)|^2 \, {\rm d}\theta \right) 
\, \left( \int_\R |v(y,\theta)|^2 \, {\rm d}\theta \right) \, {\rm d}y \\
&\le \dfrac{\pi}{|\delta_1|} \, \left( \sup_{y \in \R^{d-1}} \int_\R (1+\theta^2) \, |u(y,\theta)|^2 \, {\rm d}\theta \right) 
\, \| v \|_{L^2(\R^d)}^2 \, .
\end{align*}
Since $k_0-1>(d-1)/2$, we have
\begin{equation*}
|u(y,\theta)|^2 \le C \, \sum_{|\alpha| \le k_0-1} \int_{\R^{d-1}} |\partial_y^\alpha u(y,\theta)|^2 \, {\rm d}y \, ,
\end{equation*}
by Sobolev's inequality, and we thus get (with a possibly larger constant $C$)
\begin{equation}
\label{Festim1}
\| \F_{\rm pul}(u,v) \|_{L^2(\R^d)} \le C\, \| u \|_{\Gamma^{k_0}(\R^d)} \, \| v \|_{L^2(\R^d)} \, .
\end{equation}
The "symmetric" inequality
\begin{equation}
\label{Festim2}
\| \F_{\rm pul}(u,v) \|_{L^2(\R^d)} \le C\, \| u \|_{L^2(\R^d)} \, \| v \|_{\Gamma^{k_0}(\R^d)} \, ,
\end{equation}
is obtained by exchanging the roles of $\delta_1$ and $\delta_3$ as explained earlier.

{\bf 2)} Let us now estimate the $H^k$-norm of $\F_{\rm pul}(\partial_\theta u,\partial_\theta v)$ with $k \ge 2\, k_0+1$. 
We first apply the estimate \eqref{Festim1} for the $L^2$-norm:
\begin{equation*}
\| \F_{\rm pul}(\partial_\theta u,\partial_\theta v) \|_{L^2(\R^d)} \le C\, \| u \|_{\Gamma^{k_0+1}(\R^d)} \, \| v \|_{H^1(\R^d)} 
\le C \, \| u \|_{\Gamma^k(\R^d)} \, \| v \|_{\Gamma^k(\R^d)}\, .
\end{equation*}
Using Plancherel's Theorem, it is sufficient to estimate the $k$-th derivatives of $\F_{\rm pul}(\partial_\theta u,\partial_\theta v)$ 
in order to estimate all derivatives of order less than $k$. Let us therefore consider a multiinteger $\alpha$ of length 
$k$, and apply the Leibniz formula (this is justified because the differentiation formula \eqref{propF1} holds not only 
for the $\theta$-derivative but also for the $y$-derivatives):
\begin{equation}
\label{Festim3}
\partial^\alpha \, \F_{\rm pul}(\partial_\theta u,\partial_\theta v) =\sum_{\beta \le \alpha} \star \, \, 
\F_{\rm pul} (\partial^\beta \, \partial_\theta u,\partial^{\alpha-\beta} \, \partial_\theta v) \, ,
\end{equation}
where $\star$ denotes harmless numerical coefficients, and $\partial^\beta,\partial^{\alpha-\beta}$ stand for 
possibly mixed $y,\theta$ derivatives.

We begin with the extreme terms in \eqref{Festim3}. If $|\beta|=0$, we need to estimate the term 
$\F_{\rm pul} (\partial_\theta u,\partial^\alpha \, \partial_\theta v)$ which, using \eqref{propF2}, we write as
\begin{equation*}
-\dfrac{1}{\delta_3} \, \partial_\theta u \, \partial^\alpha v -\dfrac{\delta_1}{\delta_3} \, 
\F_{\rm pul} (\partial^2_{\theta \theta} u,\partial^\alpha v) \, .
\end{equation*}
We get the estimate
\begin{equation*}
\|Ê\F_{\rm pul} (\partial_\theta u,\partial^\alpha \, \partial_\theta v) \|_{L^2(\R^d)} \le C \, \| \partial_\theta u \|_{L^\infty(\R^d)} 
\, \| v \|_{H^k(\R^d)} +C \, \| u \|_{\Gamma^{k_0+2}(\R^d)} \, \| v \|_{H^k(\R^d)} \, ,
\end{equation*}
where we have used \eqref{Festim1}. Since $k>d/2+1$ (this follows from the assumption $k \ge 2\, k_0+1$), 
we can apply Sobolev's inequality and get
\begin{equation*}
\|Ê\F_{\rm pul} (\partial_\theta u,\partial^\alpha \, \partial_\theta v) \|_{L^2(\R^d)} \le C \, \| u \|_{\Gamma^k(\R^d)} \, 
\| v \|_{\Gamma^k(\R^d)}\, .
\end{equation*}
The second extreme term $\F_{\rm pul}(\partial^\alpha \, \partial_\theta u,\partial_\theta v)$ is dealt with in 
the same way.

Using the assumption $k \ge 2\, k_0+1$, we verify that for $\beta \le \alpha$, and $\beta \neq 0$, $\beta 
\neq \alpha$, one of the following two properties is satisfied
\begin{equation*}
(|\beta| \ge 1 \quad \text{\rm and } \quad 1+|\beta| \le k-k_0) \quad 
\text{\rm or } \quad (|\alpha|-|\beta| \ge 1 \quad \text{\rm and } \quad |\beta| \ge k_0 +1) \, .
\end{equation*}
In the first case, we use \eqref{Festim1} and get
\begin{equation*}
\|Ê\F_{\rm pul} (\partial^\beta \, \partial_\theta u,\partial^{\alpha-\beta} \, \partial_\theta v) \|_{L^2(\R^d)} \le C \, 
\|Ê\partial^\beta \, \partial_\theta u \|_{\Gamma^{k_0}(\R^d)} \, \|Ê\partial^{\alpha-\beta} \, \partial_\theta v  \|_{L^2(\R^d)} 
\le C \, \| u \|_{\Gamma^k(\R^d)} \, \| v \|_{\Gamma^k(\R^d)} \, ,
\end{equation*}
and the second case is dealt with in a symmetric way (using \eqref{Festim2} rather than \eqref{Festim1}).

{\bf 3)} It remains to estimate the $L^2$-norm of $\theta^k \, \F_{\rm pul}(\partial_\theta u,\partial_\theta v)$. We write 
(use $\delta_3-\delta_1=1$ again)
\begin{equation*}
\theta =\delta_3 \, (\theta +\delta_1 \, s) -\delta_1 \, (\theta +\delta_3 \, s) \, ,
\end{equation*}
which gives
\begin{equation}
\label{Festim4}
\theta^k \, \F_{\rm pul}(\partial_\theta u,\partial_\theta v) =\sum_{j=0}^k \star \, \, \F_{\rm pul} (\theta^j \, \partial_\theta u, 
\theta^{k-j} \, \partial_\theta v) \, ,
\end{equation}
with, again, harmless binomial coefficients that are denoted by $\star$. Let us first consider the extreme 
terms in the latter sum and focus on $\F_{\rm pul}(\partial_\theta u,\theta^k \, \partial_\theta v)$. We write
\begin{equation*}
\theta^k \, \partial_\theta v =\partial_\theta (\theta^k \, v) -k \, \theta^{k-1} \, v \, ,
\end{equation*}
and use the property \eqref{propF2} to get
\begin{equation*}
\F_{\rm pul}(\partial_\theta u,\theta^k \, \partial_\theta v) =-\dfrac{1}{\delta_3} \, \partial_\theta u \, (\theta^k \, v) 
-\dfrac{\delta_1}{\delta_3} \, \F_{\rm pul}(\partial^2_{\theta \theta} u,\theta^k \, v) 
-k \, \F_{\rm pul} (\partial_\theta u,\theta^{k-1} \, v) \, .
\end{equation*}
The $L^2$-estimate follows from \eqref{Festim1} and from the Sobolev imbedding Theorem:
\begin{equation*}
\| \F_{\rm pul}(\partial_\theta u,\theta^k \, \partial_\theta v) \|_{L^2(\R^d)} \le C \, \| u \|_{\Gamma^k(\R^d)} 
\, \| v \|_{\Gamma^k(\R^d)} \, .
\end{equation*}
The estimate for $\F_{\rm pul}(\theta^k \, \partial_\theta u,\partial_\theta v)$ is similar.

For $1 \le j \le k-1$, we observe again that there holds either $j \le k-k_0-1$ or $j \ge k_0+1$, so that we 
can directly estimate all intermediate terms in the sum \eqref{Festim4} by applying either \eqref{Festim1} 
or \eqref{Festim2}. The proof of Proposition \ref{propFpul} is now complete.
\end{proof}

\noindent As an immediate corollary of the proof we have:

\begin{cor}
\label{corproof}
Let $k_0$ denote the smallest integer satisfying $k_0>(d+1)/2$. Then for all $k \ge 2\, k_0$, there 
exists a constant $C_k$ satisfying
\begin{equation*}
\forall \, u,v \in \Gamma^k (\R^d)\, ,\quad \| \F_{\rm pul}(u, v) \|_{\Gamma^k(\R^d)} 
\le C_k\, \| u \|_{\Gamma^k(\R^d)} \, \| v \|_{\Gamma^k(\R^d)} \, .
\end{equation*}
\end{cor}

\subsection{The iteration scheme}

In view of the boundedness property proved in Proposition \ref{propFpul}, Equation \eqref{eqMachStems} is a 
semilinear perturbation of the Burgers equation (the transport term ${\bf w} \cdot \nabla_y$ is harmless), and 
it is absolutely not surprising that we can solve \eqref{eqMachStems} by using the standard energy method 
on a fixed point iteration. This well-posedness result can be summarized in the following Theorem.

\begin{theo}
\label{thmMachStems}
Let $k_0$ be defined as in Proposition \ref{propFpul}, and let $k \ge 2\, k_0+1$. Then for all $R>0$, there exists 
a time $T(k,R)>0$ such that for all data $a_0 \in \Gamma^k(\R^d)$ satisfying $\| a_0 \|_{\Gamma^k(\R^d)} \le R$, 
there exists a unique solution $a \in {\mathcal C}([0,T];\Gamma^k(\R^d))$ to the Cauchy problem:
\begin{equation*}
\begin{cases}
\partial_t a +{\bf w} \cdot \nabla_y a +c\, a \, \partial_\theta a +\mu \, 
\F_{\rm pul} (\partial_\theta a,\partial_\theta a) =0 \, ,& \\
a|_{t=0} =a_0 \, . &
\end{cases}
\end{equation*}
\end{theo}

Of course, one can also incorporate a nonzero source term $g \in L^2([0,T_0];\Gamma^k(\R^d))$ in the equation, 
and obtain a unique solution $a \in {\mathcal C}([0,T];\Gamma^k(\R^d))$ on a time interval $[0,T]$ with $T \le T_0$. 
We omit the details that are completely standard.

\begin{proof}
We follow the standard energy method for quasilinear symmetric systems, see for instance \cite[chapter 10]{BS}, 
and solve the Cauchy problem by the iteration scheme
\begin{equation*}
\begin{cases}
\partial_t a^{n+1} +{\bf w} \cdot \nabla_y a^{n+1} +c\, a^n \, \partial_\theta a^{n+1} +\mu \, \F_{\rm pul} 
(\partial_\theta a^n,\partial_\theta a^n) =0 \, ,& \\
a^{n+1}|_{t=0} =a_{0,n+1} \, , &
\end{cases}
\end{equation*}
where $(a_{0,n})$ is a sequence of, say, Schwartz functions that converges towards $a_0$ in $\Gamma^k(\R^d)$, 
and the scheme is initialized with the choice $a^0 \equiv a_{0,0}$. Given the radius $R$ for the ball in $\Gamma^k(\R^d)$, 
we can choose some time $T>0$, that only depends on $R$ and $k$, such that the sequence $(a^n)$ is bounded 
in ${\mathcal C}([0,T];\Gamma^k(\R^d))$. The uniform bound in ${\mathcal C}([0,T];H^k(\R^d))$ is proved by following 
the exact same ingredients as in the case of the Burgers equation, and the weighted $L^2$ bound is proved by 
computing
\begin{equation*}
\partial_t (\theta^k \, a^{n+1}) +{\bf w} \cdot \nabla_y (\theta^k \, a^{n+1}) +c \, a^n \, \partial_\theta (\theta^k \, a^{n+1}) 
=-\mu \, \theta^k \, \F_{\rm pul} (\partial_\theta a^n,\partial_\theta a^n) +k\, c \, a^n \, (\theta^{k-1} \, \partial_\theta a^{n+1}) \, ,
\end{equation*}
and performing the usual $L^2$-estimate for the transport equation.

It remains to show that the iteration contracts in the ${\mathcal C}([0,T];L^2(\R^d))$-topology for $T$ small enough. 
This is proved by defining $r^{n+1} :=a^{n+1}-a^n$ and computing
\begin{equation*}
\partial_t r^{n+1} +{\bf w} \cdot \nabla_y r^{n+1} +c \, a^n \, \partial_\theta r^{n+1} =-c \, r^n \, \partial_\theta a^n 
-\mu \, \F_{\rm pul}(\partial_\theta r^n,\partial_\theta a^n) -\mu \, \F_{\rm pul}(\partial_\theta a^{n-1},\partial_\theta r^n) \, . 
\end{equation*}
The error terms on the right hand-side are written as
\begin{align*}
\F_{\rm pul}(\partial_\theta r^n,\partial_\theta a^n) &=-\dfrac{1}{\delta_1} \, r^n \, \partial_\theta a^n 
-\dfrac{\delta_3}{\delta_1} \, \F_{\rm pul}(r^n,\partial^2_{\theta \theta} a^n) \, ,\\
\F_{\rm pul}(\partial_\theta a^{n-1},\partial_\theta r^n) &=-\dfrac{1}{\delta_3} \, r^n \, \partial_\theta a^{n-1} 
-\dfrac{\delta_1}{\delta_3} \, \F_{\rm pul}(\partial^2_{\theta \theta} a^{n-1},r^n) \, ,
\end{align*}
and we then apply the $L^2$-estimates \eqref{Festim1} and \eqref{Festim2}. This gives contraction of the iteration 
scheme in ${\mathcal C}([0,T];L^2(\R^d))$ and, passing to the limit, we obtain a solution $a \in L^\infty([0,T];\Gamma^k 
(\R^d))$ to the Cauchy problem. Continuity in $\Gamma^k(\R^d)$ is recovered by the same arguments as in 
\cite[chapter 10]{BS}, using the norm in $\Gamma^k(\R^d)$ rather than the $H^k$-norm. We feel free to skip 
the details.
\end{proof}

We do not claim that the regularity assumption in Theorem \ref{thmMachStems} is optimal as far as local existence 
of smooth solutions is concerned, but it is sufficient for our purpose. The global existence of weak and/or smooth 
solutions is, of course, a wide open problem. Numerical simulations reported in \cite{MR2} tend to indicate that 
finite time breakdown of smooth solutions should be expected, and Proposition \ref{propFpul} is clearly a first 
step towards a rigorous justification of this fact. We postpone the study of such finite time breakdown to a future 
work.

\subsection{Construction of the leading profile}

Corollary \ref{cor3} below is based on Theorem \ref{thmMachStems} and is entirely similar to Corollary \ref{cor1} 
for the wavetrain problem. The only difference is that we restrict here to some finite regularity since the estimate 
provided by Theorem \ref{thmMachStems} is not tame. Let us recall that the smoothness assumption for the 
source term $G$ in \eqref{0} was made in Theorem \ref{theopulses}.

\begin{cor}
\label{cor3}
With $K_0,K_1 \in \N$ and $G$ as in Theorem \ref{theopulses}, there exists $0<T \le T_0$, and there exists 
a unique 
\begin{equation*}
a \in 	\cap_{\ell =0}^{K_0} \, {\mathcal C}^\ell ((-\infty,T];\Gamma^{K_1-1-\ell}(\R^d_{y,\theta})) \, ,
\end{equation*}
solution to \eqref{eqMachStems}, with $a|_{t<0}=0$. Furthermore, the integral of $a$ with respect to the variable 
$\theta \in \R$ vanishes.
\end{cor}

\begin{proof}
The proof is rather straightforward. Since $G \in  {\mathcal C}^0((-\infty,T];\Gamma^{K_1}(\R^d))$, Theorem 
\ref{thmMachStems} yields a unique solution $a \in  {\mathcal C}^0((-\infty,T];\Gamma^{K_1-1}(\R^d))$ to 
\eqref{eqMachStems}, with $a|_{t<0}=0$. Furthermore, \eqref{eqMachStems} automatically yields $a \in 
{\mathcal C}^1((-\infty,T];\Gamma^{K_1-2}(\R^d))$ thanks to Proposition \ref{propFpul} and the fact that 
$\Gamma^{K_1-2}(\R^d)$ is an algebra. Then the standard bootstrap argument yields
\begin{equation*}
a \in 	\cap_{\ell =0}^{K_0} \, {\mathcal C}^\ell ((-\infty,T];\Gamma^{K_1-1-\ell}(\R^d_{y,\theta})) \, ,
\end{equation*}
by differentiating \eqref{eqMachStems} sufficiently many times with respect to time.

Let us observe that the requirement $K_1 -K_0-1 \ge 2\, k_0+1$ in Theorem \ref{theopulses} is used here to 
apply Proposition \ref{propFpul} for the term $\partial_{\theta_0} Q_{\rm pul}[a,a]$ in \eqref{eqMachStems1} 
and its successive time derivatives.

Thanks to the property $a \in {\mathcal C}^1(\Gamma^{K_1-2})$, $a$ is integrable with respect to $\theta \in \R$, 
and integration of \eqref{eqMachStemsstrict} yields
\begin{equation*}
X_{\rm Lop} \, \underline{a} =0 \, ,\quad \underline{a}|_{t<0} =0 \, ,
\end{equation*}
where $\underline{a}$ denotes the integral of $a$ with respect to $\theta$. Hence $\underline{a}$ is identically 
zero.
\end{proof}

\begin{cor}
\label{cor4}
Up to restricting $T>0$ in Corollary \ref{cor3}, for all $m \in {\mathcal I}$, there exists a unique solution
\begin{equation*}
\sigma_m \in \cap_{\ell =0}^{K_0} \, {\mathcal C}^\ell ((-\infty,T];\Gamma^{K_0-\ell}(\R^d_+ \times \R_\theta))
\end{equation*}
to \eqref{Burgersmp} with $\sigma_m|_{t<0}=0$ and $\sigma_m|_{x_d=0} =\mathfrak{e}_m \, a$, where the 
real number $\mathfrak{e}_m$ is defined by $e_m =\mathfrak{e}_m \, r_m$. Furthermore, each $\sigma_m$ 
has zero integral with respect to the variable $\theta_m \in \R$.
\end{cor}

\begin{proof}
From Corollary \ref{cor3}, we get $a \in \Gamma^{K_0}((-\infty,T]_t \times \R^d_{y,\theta})$, with $K_0>1+(d+1)/2$ 
(here we use $K_1-1 \ge K_0$). Then we solve the Burgers equation \eqref{Burgersmp} with prescribed boundary 
condition $\sigma_m|_{x_d=0} =\mathfrak{e}_m \, a$. The theory is similar to that in the standard Sobolev space 
$H^{K_0}$, and we feel free to use well-posedness in this weighted Sobolev space framework. This yields a solution 
\begin{equation*}
\sigma_m \in \cap_{\ell =0}^{K_0} \, {\mathcal C}^\ell ((-\infty,T];\Gamma^{K_0-\ell}(\R^d_+ \times \R_\theta)) \, ,
\end{equation*}
to \eqref{Burgersmp} that vanishes in the past. Integration of \eqref{Burgersmp} with respect to $\theta_m$ shows 
that the integral of $\sigma_m$ with respect to the variable $\theta_m \in \R$ satisfies
\begin{equation*}
\partial_t \underline{\sigma}_m +{\bf v}_m \cdot \nabla_x \underline{\sigma}_m =0 \, ,\quad 
\underline{\sigma}_m|_{x_d=0} =0 \, ,
\end{equation*}
and therefore vanishes.
\end{proof}

\begin{rem}
\label{remsect7}
We observe that there has been a rather big loss of regularity in passing from the trace $a$ to the interior 
functions $\sigma_m$. This is due to the following fact: the trace $a$ is obtained by solving a Cauchy problem, 
where tangential regularity with respect to the time slices $\{ t=\text{\rm constant} \}$ is propagated. However, 
constructing a local in time smooth solution to \eqref{Burgersmp} requires a full regularity for the trace $a$, 
that is, regularity of both $(y,\theta)$ and time partial derivatives. This is the reason why we also need to control 
time derivatives of $a$, which means controlling time derivatives of $\partial_{\theta_0} Q_{\rm pul}[a,a]$. In 
Corollary \ref{cor3}, we have considered sufficiently smooth initial data so that such time derivatives are 
controlled by an easy application of Proposition \ref{propFpul}. Again, we do not aim at an optimal regularity 
result.
\end{rem}

\subsection{Extension to more general $N\times N$ systems}
\label{extensionpulses}

\emph{\quad} The extension of the definitions ($\cV_F, \cV_H, E_P, E_Q, R_\infty$, etc.) and results for pulses 
in the strictly hyperbolic $3\times 3$ case to the strictly hyperbolic $N\times N$ case is straightforward. We first 
show that the leading profile $\cU_0$ reads
\begin{equation*}
{\mathcal U}_0(t,x,\theta_0,\xi_d) =\sum_{m \in {\mathcal I}} \sigma_m (t,x,\theta_0 +\underline{\omega}_m \, \xi_d) 
\, r_m \, ,
\end{equation*}
with functions $\sigma_m$ that are expected to decay at infinity with respect to $\theta_m$. Moreover, the 
$\sigma_m$'s satisfy \eqref{Burgersmp} in the domain $\{ x_d >0\}$, and \eqref{4p}(a) yields
\begin{equation*}
\forall \, m \in {\mathcal I} \, ,\quad \sigma_m (t,y,0,\theta_0) \, r_m =a(t,y,\theta_0) \, e_m \, ,
\end{equation*}
for a suitable function $a$.

The existence of a bounded corrector $\cU_2$ solution to \eqref{3p}(c) implies that the first corrector reads
\begin{align*}
{\mathcal U}_1(t,x,\theta_0,\xi_d) =& \sum_{m \in {\mathcal O}} -\int_{\xi_d}^{+\infty} \ell_m \, {\mathcal F}_0 
(t,x,\theta_0 +\underline{\omega}_m \, (\xi_d-X),X) \, {\rm d}X \, r_m \\
&+\sum_{m \in {\mathcal I}} \left( \tau_m(t,x,\theta_0 +\underline{\omega}_m \, \xi_d) -\int_{\xi_d}^{+\infty} 
\ell_m \, {\mathcal F}_0 (t,y,0,\theta_0 +\underline{\omega}_m \, (\xi_d-X),X) \, {\rm d}X \right) \, r_m \, .
\end{align*}
Plugging these expressions in \eqref{4p}(b), we derive the following amplitude equation that governs the 
evolution of $a$ on the boundary:
\begin{equation}
\label{eqMachStemsstrict}
\upsilon \, \partial_{\theta_0} (a^2) -X_{\rm Lop} a +\partial_{\theta_0} Q_{\rm pul}[a,a] 
=\underline{b} \, \partial_{\theta_0} G \, ,
\end{equation}
with $\upsilon$ and $X_{\rm Lop}$ defined in \eqref{defupsilonstrict}, and $Q_{\rm pul}$ defined by:
\begin{align}
\label{defQpulstrict}
\begin{split}
&Q_{\rm pul}[a,\widetilde{a}] (\theta_0) :=\sum_{m \in {\mathcal O}} \underline{b} \, B \, r_m \, 
\sum_{\underset{m_1, m_2 \in {\mathcal I}}{m_1 <m_2}} \ell_m \, E_{m_1,m_2} \, \int_0^{+\infty} 
\partial_{\theta_0} a (\theta_0+(\underline{\omega}_{m_1}-\underline{\omega}_m)\, X) \, 
\widetilde{a}(\theta_0+(\underline{\omega}_{m_2}-\underline{\omega}_m)\, X) \, {\rm d}X \\
&\qquad +\sum_{m \in {\mathcal O}} \underline{b} \, B \, r_m \, 
\sum_{\underset{m_1, m_2 \in {\mathcal I}}{m_1 < m_2}} \ell_m \, E_{m_2,m_1} \, \int_0^{+\infty} 
a (\theta_0+(\underline{\omega}_{m_1}-\underline{\omega}_m)\, X) \, 
\partial_{\theta_0} \widetilde{a}(\theta_0+(\underline{\omega}_{m_2}-\underline{\omega}_m)\, X) \, {\rm d}X \, ,
\end{split}
\end{align}
with $E_{m_1,m_2}$ as in \eqref{defEm1m2}.
\bigskip

As in the wavetrain case some care is needed in the extension to hyperbolic systems with constant multiplicities. 
Instead of introducing bases $\{\ell_{m,k}\}$, $\{r_{m,k}\}$, $m=1,\dots, M$, $k=1,\dots, \nu_{k_m}$ of left and 
right eigenvectors, we now define the spaces $\cV_F$ and $\cV_H$ and operators $E_P$, $E_Q$, $R_\infty$ 
as follows.

\begin{defn}\label{functions2}
(a) Define the set of ``constituent functions" $\cC=\cup^M_{m=1} \cC^m$, where $\cC^m$ is the set of $\R^N$-valued 
functions of $(t,x,\theta_0,\xi_d)$ of the form $F(t,x,\theta_0+\uomega_m\xi_d)$,  where $F(t,x,\theta_m)$  is $\cC^1$ 
and decays with its first-order partials at the rate $O(\langle \theta \rangle^{-2})$ uniformly with respect to $(t,x)$. 
Setting $\cV_\cC :=\oplus^M_{m=1} \cC^m$, we can write any $F\in \cV_\cC$ as $F =\sum^M_{m=1} F^m$ with 
$F^m\in\cC^m$.

(b) Define $\cV_F$ to be the space of $\R^N$-valued functions of $(t,x,\theta_0,\xi_d)$ that may be written as a finite 
sum of terms of the form
\begin{equation*}
F \, ,\quad \cB_\alpha(G,H) \, ,\quad \cT_\beta(K,L,M) \, ,
\end{equation*}
where $F$, $G, \dots, M$ lie in $\cV_\cC$, $\cB_\alpha : \R^N \times \R^N \to \R^N$ is any bilinear map, and 
$\cT_\beta : \R^N \times \R^N \times \R^N \to \R^N$ is any trilinear map.

(c) Define a \emph{transversal interaction integral} to be a function $I^i_{\ell,m}$ of the form
\begin{equation}
\label{b1}
I^i_{\ell,m}(t,x,\theta_0,\xi_d) =-\int_{\xi_d}^{+\infty} A_d(0)^{-1} \, Q_i \, \cB_\gamma (F^\ell,G^m) 
(t,x,\theta_0+\uomega_i \, (\xi_d-s),s) \, {\rm d}s \, ,
\end{equation}
where $i,\ell,m$ lie in $\{1,\dots,M\}$ and are mutually distinct, $\cB_\gamma : \R^N \times \R^N \to \R^N$ is any 
bilinear map, and $F^\ell l\in \cC^\ell$, $G^m \in \cC^m$ are required to decay with their first-order partials at the 
rate $O(\langle \theta \rangle^{-3})$ uniformly with respect to $(t,x)$.

(d) Define $\cV_H$ to be the space of $\R^N$-valued functions of $(t,x,\theta_0,\xi_d)$  that may be written as 
the sum of an element of $\cV_F$ plus a finite sum of terms of the form
\begin{equation}
\label{b2}
I^i_{\ell,m} (t,x,\theta_0,\xi_d) \quad \text{ or } \quad \cB_\delta (A^j, J^n_{p,q}(t,x,\theta_0,\xi_d)) \, ,
\end{equation}
where $I^i_{\ell,m}$ and $J^n_{p,q}$ are transversal interaction integrals,  $\cB_\delta : \R^N \times \R^N \to \R^N$ 
is any bilinear map, and $A^j \in \cC^j$ is required to decay with its first-order partials at the rate $O(\langle \theta 
\rangle^{-3})$ uniformly with respect to $(t,x)$.

(e) For  $H \in \cV_H$ we can write $H=H_F+H_I$, where $H_F \in \cV_F$ and $H_I \notin \cV_F$ is a finite sum 
of terms of the form \eqref{b2}. The ``constituent functions" of $H$ include those of $H_F$ as well as the functions 
like $A^j$, $F^\ell$, $G^m$ as in \eqref{b1}, \eqref{b2} which constitute $H_I$.

(f) With these definitions of $\cV_F$, $\cV_H$, and ``constituent functions", the subspaces $\cV_F^k$ and $\cV_H^k$ 
may be defined just as in Definition \ref{VFK}.
\end{defn}

\begin{defn}[$E_P$,$E_Q$,$R_\infty$]
\label{ops2}
For $H\in \cV_H$, define averaging operators
\begin{align*}
\begin{split}
&E_Q \, H(t,x,\theta_0,\xi_d) :=\sum_{m=1}^M \lim_{T \to \infty} \dfrac{1}{T} \, \int^T_0 Q_m \, 
H(x,\theta_0+\uomega_m \, (\xi_d-s),s) \, {\rm d}s \, ,\\
&E_P \, H(t,x,\theta_0,\xi_d) :=\sum_{m=1}^M \lim_{T \to \infty} \dfrac{1}{T} \, \int^T_0 P_m \, 
H(x,\theta_0+\uomega_m \, (\xi_d-s),s) \, {\rm d}s \, .
\end{split}
\end{align*}
For $H\in \cV_H$ such that $E_Q \, H=0$, define the solution operator
\begin{equation*}
R_\infty \, H(t,x,\theta_0,\xi_d) :=-\sum_{m=1}^M \int_{\xi_d}^{+\infty} A_d(0)^{-1} \, Q_m \, 
H(t,x,\theta_0+\uomega_m \, (\xi_d-s),s) \, {\rm d}s \, .
\end{equation*}
\end{defn}

With these definitions, Propositions \ref{a3} and \ref{properties} hold with the obvious minor changes. 
For example, in Proposition \ref{properties}(c), we now have
\begin{equation*}
E_P \, \cU =\sum^M_{m=1} \cU^m \, , \text{ where } \; \cU^m \in \cC^m \; \text{ and } \; P_m \, \cU^m=\cU^m \, .
\end{equation*}
The construction of profiles is carried out assuming $\cU_2 \in \cC^1_b$ and that $\cU_0$, $\cU_1$ satisfy 
\eqref{a9}, where \eqref{a9}(a)  is modified to
\begin{equation*}
\cU_0 =\sum^M_{m=1} \cU^m_0 \, ,\quad \cU^m_0=P_m \, \cU_0^m \in \cC^m \cap \cV^{K_0}_F \, ,\quad 
\int_\R \cU_0^m (t,x,\theta_m) \, {\rm d}\theta_m=0 \, .
\end{equation*}
Again the interior leading profile equations for the $Q_m\cU^m_0$ take the form \eqref{eq:Burgerssyst}, 
which allows the moment zero property of the $\cU^m_0$ to be deduced from that of the amplitude $a$ 
as before. The solvability in $\Gamma^k$ spaces of \eqref{eq:Burgerssyst} with boundary conditions 
\eqref{bcs} follows from Lemma \ref{sh} via $L^2$ estimates proved in the standard way.\footnote{Such 
an argument is given in detail in Propositions 3.5 and 3.6 of \cite{CW4}.} One finds as before that $\cU_0$ 
is purely incoming \eqref{polar}.

The consequence \eqref{symm} of the conservative structure assumption is used again in step 2 of the 
profile construction when performing the integral that now replaces the integral on $[\xi_d,+\infty[$ of the 
second line of \eqref{a6}. This integral now reads
\begin{equation*}
-\sum_{k\neq m} \int_{\xi_d}^{+\infty} A_d(0)^{-1} \, Q_m \, \cM(\cU^k_0,\cU^k_0) \, (t,x,\theta_0+\uomega_m 
\, (\xi_d-s),s) \, {\rm d}s \, .
\end{equation*}
The replacements for the other lines of \eqref{a6} are treated essentially as before; the last line now yields 
transversal interaction integrals of the form \eqref{b1}.

Parallel to \eqref{qoper} the nonlocal operator in the equation for $a$ now has the form
\begin{multline*}
Q_{\rm pul}[a,\widetilde{a}] (\theta_0) := \\
\sum_{m \in {\mathcal O}}
\sum_{\underset{m_1, m_2 \in {\mathcal I}}{m_1 <m_2}} \underline{b} \, B \, A_d(0)^{-1} \, Q_m \, E_{m_1,m_2} \, 
\int_0^{+\infty} \partial_{\theta_0} a (\theta_0+(\underline{\omega}_{m_1}-\underline{\omega}_m)\, X) \, 
\widetilde{a}(\theta_0+(\underline{\omega}_{m_2}-\underline{\omega}_m)\, X) \, {\rm d}X \\
+\sum_{m \in {\mathcal O}} \sum_{\underset{m_1, m_2 \in {\mathcal I}}{m_1 < m_2}} 
\underline{b} \, B \, A_d(0)^{-1} \, Q_m \, E_{m_2,m_1} \, \int_0^{+\infty} 
a (\theta_0+(\underline{\omega}_{m_1}-\underline{\omega}_m)\, X) \, 
\partial_{\theta_0} \widetilde{a}(\theta_0+(\underline{\omega}_{m_2}-\underline{\omega}_m)\, X) \, {\rm d}X \, ,
\end{multline*}
with $E_{m_1,m_2}$ as in \eqref{defEm1m2}.

Repetition of step 6 of the profile construction yields $\cU_0$, $\cU_1$, $\cU_2$ with the same regularity 
and decay properties as before.


\newpage
\part{Appendices}
\appendix
\section{Example: the two-dimensional isentropic Euler equations}
\label{appA}

In this first appendix, we discuss how our main results apply to the two-dimensional Euler equations in a fixed half-plane. 
Once again, we refer to \cite{MR,AM,wangyu} for the derivation of such amplified high frequency expansions in various 
free boundary problems. Our discussion here will mainly deal with the verification of the small divisor condition, that is, 
Assumption \ref{assumption5}. In quasilinear form, the isentropic Euler equations read
\begin{equation}
\label{euler}
\begin{cases}
\partial_t v +u_1\, \partial_1 v +u_2\, \partial_2 v -v \, (\partial_1 u_1 +\partial_2 u_2) =0\, ,& \\
\partial_t u_1 +u_1\, \partial_1 u_1 +u_2\, \partial_2 u_1 -\dfrac{c(v)^2}{v} \, \partial_1 v =0 \, ,& \\
\partial_t u_2 +u_1\, \partial_1 u_2 +u_2\, \partial_2 u_2 -\dfrac{c(v)^2}{v} \, \partial_2 v =0 \, ,&
\end{cases}
\end{equation}
where $v>0$ denotes the specific volume of the fluid, $c(v)>0$ denotes the sound speed and $(u_1,u_2)$ the velocity field. 
We consider a fixed reference volume $\underline{v}$, and a fixed velocity field $(0,\underline{u})$ with
\begin{equation*}
\underline{v}>0 \, ,\quad 0<\underline{u}<\underline{c} :=c(\underline{v}) \, ,
\end{equation*}
which corresponds to an {\sl incoming subsonic} fluid. The above theory applies when looking for WKB solutions to 
\eqref{euler} of the form
\begin{equation*}
\begin{pmatrix}
v \\
u_1 \\
u_2 \end{pmatrix}_\eps \sim \begin{pmatrix}
\underline{v} \\
0 \\
\underline{u} \end{pmatrix} +\eps \, \sum_{n \ge 0} \eps^n \, {\mathcal U}_n \left( t,x,\dfrac{\Phi (t,x)}{\eps} \right) \, ,
\end{equation*}
provided that the linearization of \eqref{euler} (with appropriate boundary conditions) at the reference state $(\underline{v}, 
0,\underline{u})$ satisfies all assumptions of Section \ref{intro}. Let us therefore consider the linearization of \eqref{euler} 
at $(\underline{v},0,\underline{u})$, which corresponds, in the notation of Section \ref{intro}, to
\begin{align*}
&A_1(0) := \begin{pmatrix}
0 & -\underline{v} & 0 \\
-\underline{c}^2/\underline{v} & 0 & 0 \\
0 & 0 & 0 \end{pmatrix} \, ,\quad 
A_2(0) := \begin{pmatrix}
\underline{u} & 0 & -\underline{v} \\
0 & \underline{u} & 0 \\
-\underline{c}^2/\underline{v} & 0 & \underline{u} \end{pmatrix} \, ,\\
&{\rm d}A_1(0) \cdot \begin{pmatrix}
v \\
u_1 \\
u_2 \end{pmatrix} :=\begin{pmatrix}
u_1 & -v & 0 \\
(-2 \, \underline{c} \, \underline{c}'/\underline{v}+\underline{c}^2/\underline{v}^2) \, v & u_1 & 0 \\
0 & 0 & u_1 \end{pmatrix} \, ,\\
&{\rm d}A_2(0) \cdot \begin{pmatrix}
v \\
u_1 \\
u_2 \end{pmatrix} := \begin{pmatrix}
u_2 & 0 & -v \\
0 & u_2 & 0 \\
(-2 \, \underline{c} \, \underline{c}'/\underline{v}+\underline{c}^2/\underline{v}^2) \, v & 0 & u_2 \end{pmatrix} \, ,
\end{align*}
where $\underline{c}'$ stands for $c'(\underline{v})$. The operator $\partial_t +A_1(0)\, \partial_1 +A_2(0) \, \partial_2$ 
is strictly hyperbolic with eigenvalues
\begin{equation*}
\lambda_1(\xi_1,\xi_2) :=\underline{u} \, \xi_2 -\underline{c} \, \sqrt{\xi_1^2+\xi_2^2} \, ,\quad
\lambda_2(\xi_1,\xi_2) :=\underline{u} \, \xi_2 \, ,\quad
\lambda_3(\xi_1,\xi_2) := \underline{u} \, \xi_2 +\underline{c} \, \sqrt{\xi_1^2+\xi_2^2} \, ,
\end{equation*}
so Assumptions \ref{assumption1} and \ref{assumption1'} are satisfied. There are two incoming characteristics and 
one outgoing characteristic, so the matrix $B$ encoding the boundary conditions for the linearized equations should 
be a $2 \times 3$ matrix of maximal rank. One possible choice for $B$ is made precise below. The hyperbolic region 
${\mathcal H}$ is given by
\begin{equation*}
{\mathcal H} = \left\{ (\tau,\eta) \in \R \times \R \, / \, |\tau| > \sqrt{\underline{c}^2-\underline{u}^2} \, |\eta| \right\} \, .
\end{equation*}
We focus on the verification of Assumption \ref{assumption5}. For concreteness, let $(\tau,\eta) \in {\mathcal H}$ 
with $\tau>0$ and $\eta>0$, and consider the boundary phase
\begin{equation*}
\varphi_0(t,x_1) :=\tau \, t+\eta \, x_1 \, .
\end{equation*}
The three (distinct) eigenvalues of ${\mathcal A}(\tau,\eta)$ are
\begin{equation*}
\omega_1 :=\dfrac{\underline{u} \, \tau -\underline{c} \, \sqrt{\tau^2 -(\underline{c}^2-\underline{u}^2) \, \eta^2}}
{\underline{c}^2-\underline{u}^2} \, ,\quad 
\omega_2 :=\dfrac{\underline{u} \, \tau +\underline{c} \, \sqrt{\tau^2 -(\underline{c}^2-\underline{u}^2) \, \eta^2}}
{\underline{c}^2-\underline{u}^2} \, ,\quad 
\omega_3 := -\dfrac{\tau}{\underline{u}} \, ,
\end{equation*}
and they satisfy 
\begin{equation*}
\tau +\lambda_1 (\eta,\omega_1) =\tau +\lambda_1 (\eta,\omega_2) =\tau +\lambda_2 (\eta,\omega_3) = 0 \, .
\end{equation*}
The associated group velocities are
\begin{equation*}
{\bf v}_1 :=\dfrac{1}{\tau +\underline{u} \, \omega_1} \, \begin{pmatrix}
-\underline{c}^2 \, \eta \\
\underline{c} \, \sqrt{\tau^2 -(\underline{c}^2-\underline{u}^2) \, \eta^2} \end{pmatrix} \, ,\quad {\bf v}_2 := 
\dfrac{1}{\tau +\underline{u} \, \omega_2} \, \begin{pmatrix}
-\underline{c}^2 \, \eta \\
-\underline{c} \, \sqrt{\tau^2 -(\underline{c}^2-\underline{u}^2) \, \eta^2} \end{pmatrix} \, ,\quad {\bf v}_3 :=\begin{pmatrix}
0 \\
\underline{u} \end{pmatrix} \, ,
\end{equation*}
hence the phase $\varphi_2$ is outgoing while $\varphi_1,\varphi_3$ are incoming (as in the framework of Paragraph 
\ref{sect2example}). The nonresonance condition in Assumption \ref{assumption5}, that is,
\begin{equation*}
\forall \, \alpha \in \Z^3 \setminus \Z^{3;1} \, ,\quad \det L({\rm d} (\alpha \cdot \Phi)) \neq 0 \, ,
\end{equation*}
holds if and only if the (dimensionless) quantity
\begin{equation*}
\dfrac{\underline{u}}{\underline{c}} \, \sqrt{1- (\underline{c}^2-\underline{u}^2) \, \dfrac{\eta^2}{\tau^2}} \, ,
\end{equation*}
is not a rational number, as follows from a straightforward calculation. We thus introduce a positive irrational parameter 
$\zeta$ defined by
\begin{equation}
\label{defzeta}
\dfrac{\underline{u}}{\underline{c}} \, \sqrt{\tau^2 -(\underline{c}^2-\underline{u}^2) \, \eta^2} =\zeta \, \tau \, .
\end{equation}
Our choice of parameters gives $\zeta \in (0,1)$. For all $\alpha_1,\alpha_3 \in \Z \setminus \{ 0\}$, we compute
\begin{align*}
\det L({\rm d} (\alpha_1 \, \varphi_1+\alpha_3 \, \varphi_3)) &=-\alpha_1 \, \alpha_3 \, \underline{c}^2 \, 
(\tau +\underline{u} \, \omega_1) \, \big[ 2 \, \alpha_1 \, (\eta^2 +\omega_1 \, \omega_3) +\alpha_3 \, 
(\eta^2 +\omega_3^2) \big] \\
&= \alpha_1 \, \alpha_3 \, (\tau +\underline{u} \, \omega_1) \, \dfrac{\underline{c}^4 \, \tau^2}{\underline{u}^2 \, 
(\underline{c}^2-\underline{u}^2)} \, (\zeta -1) \, \big( 2 \, \alpha_1 \, \zeta +\alpha_3 \, (\zeta+1) \big) \, .
\end{align*}
It appears from the latter expression that the verification of Assumption \ref{assumption5} only depends on the arithmetic 
properties of the parameter $\zeta$ in \eqref{defzeta}. In particular, we have the following result.

\begin{lem}
\label{lemeuler}
For all $\gamma>1$, there exists a set of full measure ${\mathcal M}_\gamma$ in $(0,1)$ such that for all $M \in 
{\mathcal M}_\gamma$ and $\underline{c}>0$, if $\underline{u} =M \, \underline{c}$ and $\tau=\underline{c} \, \eta$, 
then Assumption \ref{assumption5} is satisfied for some constant $c>0$.
\end{lem}

Of course, one could also fix the parameters $\underline{u},\underline{c}$ and make $\tau$ vary in the hyperbolic 
region, and obtain a similar result (meaning that Assumption \ref{assumption5} would be satisfied except possibly for 
$\tau$ in a negligible set).

\begin{proof}
Choosing $\tau=\underline{c} \, \eta$ with $\eta>0$, one has $(\tau,\eta) \in {\mathcal H}$, and the parameter $\zeta$ 
in \eqref{defzeta} equals $M^2$ with $M :=\underline{u}/\underline{c}$. Using the above expression for the determinant 
of $L({\rm d} (\alpha_1 \, \varphi_1+\alpha_3 \, \varphi_3))$, we see that the small divisors condition of Assumption 
\ref{assumption5} will be satisfied provided that there exists a constant $c>0$ such that
\begin{equation}
\label{smalldivisors}
\forall \, (p,q) \in \Z^2 \setminus \{ (0,0) \} \, ,\quad |p+q \, M^2| \ge c \, |(p,q)|^{-\gamma} \, .
\end{equation}

For $c \in (0,1)$ and $\gamma>1$, let us define
\begin{equation*}
{\mathcal N}_c := \Big\{ m \in (0,1) \, / \, \exists \, (p,q) \in \Z^2 \, ,\quad q \neq 0 \quad \text{\rm and } \quad 
|p+q\, m| <\dfrac{c}{|(p,q)|^\gamma} \Big\} \, .
\end{equation*}
The set ${\mathcal N}_c$ is the countable union, as $p$ varies in $\Z$ and $q$ in $\Z^*$, of intervals of width 
at most $2\, c/(|q| \, |(p,q)|^\gamma)$. Hence the Lebesgue measure of ${\mathcal N}_c$ is $O(c)$, which 
means that the intersection $\cap_{c>0} {\mathcal N}_c$ is negligible. Consequently, for every fixed $\gamma>1$, 
the set of parameters $M^2$ such that \eqref{smalldivisors} is satisfied for some constant $c>0$ has full measure 
$1$. The claim of Lemma \ref{lemeuler} follows.
\end{proof}

We assume from now on that $\tau,\eta$ are fixed such that $\tau=\underline{c} \, \eta$, and the Mach number $M$ 
is chosen in such a way that Assumption \ref{assumption5} is satisfied (which is some kind of a generic condition on 
$M$). Then we compute
\begin{equation*}
r_1 :=\begin{pmatrix}
\underline{v} \\
\underline{c} \\
0 \end{pmatrix} \, ,\quad r_2 :=\begin{pmatrix}
\dfrac{1+M^2}{1-M^2} \, \underline{v} \\
\underline{c} \\
\dfrac{2 \, \underline{u}}{1-M^2} \end{pmatrix} \, ,\quad r_3 :=\begin{pmatrix}
0 \\
\underline{c} \\
\underline{u} \end{pmatrix} \, ,
\end{equation*}
and
\begin{align*}
\ell_1 &:=\begin{pmatrix}
\dfrac{1}{2\, \underline{v} \, \underline{u}} & \dfrac{1}{2\, \underline{u} \, \underline{c}} & 0 \end{pmatrix} \, ,\quad 
\ell_2 :=-\dfrac{1}{1+M^2} \, \begin{pmatrix}
\dfrac{1+M^2}{2\, \underline{v} \, \underline{u}} & \dfrac{1-M^2}{2\, \underline{u} \, \underline{c}} & 
\dfrac{1}{\underline{c}^2} \end{pmatrix} \, ,\\
\ell_3 &:=\dfrac{1}{1+M^2} \, \begin{pmatrix}
0 & \dfrac{1}{\underline{u} \, \underline{c}} & \dfrac{1}{\underline{c}^2} \end{pmatrix} \, .
\end{align*}

We now make the choice of $B$ precise. As in \cite[Appendix B]{CGW3}, we choose
\begin{equation*}
B := \begin{pmatrix}
0 & \underline{v} & 0 \\
\underline{u} & 0 & \underline{v} \end{pmatrix} \, ,
\end{equation*}
which does not have any physical interpretation but makes the following calculations rather easy. The reader 
can check that Assumptions \ref{assumption2} and \ref{assumption3} are satisfied, and we can choose $e :=r_1 -r_3$ 
as the vector that spans $\ker B \, \cap \, \E^s \tauetabar$. The one-dimensional space $B \, \E^s \tauetabar$ 
can be written as the orthogonal of the kernel of the linear form $\underline{b}:=(\underline{u},-\underline{c})$. 
We can then compute the bilinear Fourier multiplier $Q_{\rm per}$ defined in \eqref{defQper}, and get:
\begin{equation*}
Q_{\rm per}[a,\widetilde{a}] =\underline{v} \, \underline{u} \, \underline{c} \, \sum_{k \in \Z} \left( 
\sum_{\underset{k_1 \, k_3 \neq 0}{k_1+k_3=k,}} \left( 1+\dfrac{2\, (1-M^2) \, k_1}{k_3 +(2\, k_1 +k_3) \, M^2} \right) 
\, a_{k_1} \, \widetilde{a}_{k_3} \right) \, {\rm e}^{2 \, i \, \pi \, k \, \theta_0/\Theta} \, ,
\end{equation*}
whose kernel is unbounded and depends, as expected, on the arithmetic properties of $M^2$.

To conclude this Appendix, let us observe that in three space dimensions, the isentropic Euler equations are no longer 
strictly hyperbolic but they enjoy a conservative structure (in the physical variables $\rho, \rho \, \vec{u}$). Similarly, 
the full Euler equations with the energy conservation law also have a conservative structure. Hence Assumption 
\ref{assumption1'} is satisfied, and one can perform a similar derivation as above for the leading amplitude equation.

\section{Formal derivation of the large period limit: from wavetrains to pulses}
\label{appB}

\subsection{The large period limit of the amplitude equation \eqref{eqa2}}

In this Appendix, we study the relationship between the quadratic operators arising in the leading amplitude 
equations for the wavetrains and pulses problems which we have considered. For the sake of clarity, we focus 
on the easiest case $N=3$, $p=2$, that was considered in paragraph \ref{sect2example} and section \ref{sect6}. 
The Fourier multiplier $Q_{\rm per}$ is then defined by \eqref{defQper}, while the bilinear operator $Q_{\rm pul}$ 
is defined by \eqref{defQpul}. From these expressions, we can decompose both operators as
\begin{align*}
&Q_{\rm per}[a,a] =(\underline{b} \, B \, r_2) \, \ell_2 \, E_{1,3} \, \widetilde{\F}_{\rm per}[\partial_\theta a,a] 
+(\underline{b} \, B \, r_2) \, \ell_2 \, E_{3,1} \, \widetilde{\F}_{\rm per}[a,\partial_\theta a] \, ,\\
&Q_{\rm pul}[a,a] =(\underline{b} \, B \, r_2) \, \ell_2 \, E_{1,3} \, \widetilde{\F}_{\rm pul}[\partial_\theta a,a] 
+(\underline{b} \, B \, r_2) \, \ell_2 \, E_{3,1} \, \widetilde{\F}_{\rm pul}[a,\partial_\theta a] \, ,
\end{align*}
with (observe the slightly different normalizations with respect to \eqref{defFper} and \eqref{defFpul}):
\begin{align}
\widetilde{\F}_{\rm per}[u,v](\theta) &:= \dfrac{i\, \Theta}{2\, \pi} \, \sum_{k \in \Z} \left( 
\sum_{\underset{k_1 \, k_3 \neq 0}{k_1+k_3=k,}} 
\dfrac{u_{k_1} \, v_{k_3}}{k_1 \, (\underline{\omega}_1-\underline{\omega}_2) 
+k_3 \, (\underline{\omega}_3-\underline{\omega}_2)} \right) \, 
{\rm e}^{2 \, i \, \pi \, k \, \theta/\Theta} \, ,\label{defF1per} \\
\widetilde{\F}_{\rm pul}[u,v] (\theta) &:=\int_0^{+\infty} u(\theta+(\underline{\omega}_1-\underline{\omega}_2)\, s) 
\, v(\theta+(\underline{\omega}_3-\underline{\omega}_2)\, s) \, {\rm d}s \, .\label{defF1pul}
\end{align}
In \eqref{defF1per}, both functions $u,v$ are assumed to be $\Theta$-periodic and $u_k,v_k$ stand for their $k$-th 
Fourier coefficient, while in \eqref{defF1pul}, $u,v$ are defined on $\R$ and have sufficiently fast decay at infinity 
(so that the integral makes sense).

Our goal is to explain how one can compute the (formal) limit of $\widetilde{\F}_{\rm per}$ when the period $\Theta$ 
becomes infinitely large and to make the link with the expression of $\widetilde{\F}_{\rm pul}$ in \eqref{defF1pul}. The 
additional variables $(t,y)$ play the role of parameters here, so we focus on $\widetilde{\F}_{\rm per},\widetilde{\F}_{\rm pul}$ 
as operators acting on functions that depend on a single variable $\theta$. We pick two functions $u,v$ in the Schwartz 
class ${\mathcal S}(\R)$, and define
\begin{equation*}
\forall \, \theta \in \R \, ,\quad u_\Theta (\theta) := \sum_{n \in \Z} u(\theta +n\, \Theta) \, ,\quad 
v_\Theta (\theta) := \sum_{n \in \Z} v(\theta +n\, \Theta) \, .
\end{equation*}
The functions $u_\Theta,v_\Theta$ are $\Theta$-periodic and converge, uniformly on compact sets, towards 
$u,v$ as $\Theta$ tends to infinity. Moreover, the Poisson summation formula gives the Fourier coefficients of 
$u_\Theta,v_\Theta$ in terms of the Fourier transform of $u,v$:
\begin{equation*}
u_\Theta (\theta) =\dfrac{1}{\Theta} \, \sum_{k \in \Z} \widehat{u} \left( \dfrac{2\, k \, \pi}{\Theta} \right) \, 
{\rm e}^{2 \, i \, \pi \, k \, \theta/\Theta} \, ,\quad 
v_\Theta (\theta) =\dfrac{1}{\Theta} \, \sum_{k \in \Z} \widehat{v} \left( \dfrac{2\, k \, \pi}{\Theta} \right) \, 
{\rm e}^{2 \, i \, \pi \, k \, \theta/\Theta} \, .
\end{equation*}
We compute
\begin{align*}
\widetilde{\F}_{\rm per}[u_\Theta,v_\Theta] (\theta) &=\dfrac{i}{2\, \pi \, \Theta} \, \sum_{k \in \Z} \left( 
\sum_{\underset{k_1 \, k_3 \neq 0}{k_1+k_3=k,}} 
\dfrac{\widehat{u} (2\, k_1 \, \pi/\Theta) \, \widehat{v} (2\, k_3 \, \pi/\Theta)}{k_1 \, (\underline{\omega}_1-\underline{\omega}_2) 
+k_3 \, (\underline{\omega}_3-\underline{\omega}_2)} \right) \, {\rm e}^{2 \, i \, \pi \, k \, \theta/\Theta} \\
&=\dfrac{i}{4\, \pi^2} \, \dfrac{4\, \pi^2}{\Theta^2} \, \sum_{k \in \Z} \sum_{\underset{k_1 \, k_3 \neq 0}{k_1+k_3=k,}} 
\dfrac{\widehat{u} (2\, k_1 \, \pi/\Theta) \, \widehat{v} (2\, k_3 \, \pi/\Theta)}{(2\, k_1 \, \pi/\Theta) \, 
(\underline{\omega}_1-\underline{\omega}_2) +(2\, k_3 \, \pi/\Theta) \, (\underline{\omega}_3-\underline{\omega}_2)} 
\, {\rm e}^{2 \, i \, \pi \, k \, \theta/\Theta} \, .
\end{align*}
The latter expression suggests that $\widetilde{\F}_{\rm per}[u_\Theta,v_\Theta] (\theta)$ is the approximation by a 
Riemann sum, of the double integral\footnote{It is not obvious at first sight that this double integral makes any sense, 
but our goal here is to identify formally the large period limit so let us pretend that all manipulations on the integrals 
and limits are valid.}
\begin{equation*}
\dfrac{i}{4\, \pi^2} \, \int_\R \int_\R \dfrac{\widehat{u} (\eta) \, \widehat{v} (\xi-\eta)}{\eta \, 
(\underline{\omega}_1-\underline{\omega}_2) +(\xi-\eta) \, (\underline{\omega}_3-\underline{\omega}_2)} \, 
{\rm e}^{i \, \xi \, \theta} \, {\rm d}\eta \, {\rm d}\xi \, .
\end{equation*}
Formally, this means that, as $\Theta$ tends to infinity, $\widetilde{\F}_{\rm per}[u_\Theta,v_\Theta]$ tends towards 
a function whose Fourier transform is given by
\begin{equation}
\label{Fourierlimit}
\xi \in \R \longmapsto -\dfrac{1}{2\, i\, \pi} \, \int_\R \dfrac{\widehat{u} (\eta) \, \widehat{v} (\xi-\eta)}{\eta \, 
(\underline{\omega}_1-\underline{\omega}_2) +(\xi-\eta) \, (\underline{\omega}_3-\underline{\omega}_2)} \, {\rm d}\eta \, ,
\end{equation}
assuming of course that the latter expression makes any sense. We wish to compare \eqref{Fourierlimit} with the 
Fourier transform of $\widetilde{\F}_{\rm pul}[u,v]$, whose expression is given by the following (rigorous!) result.

\begin{lem}
\label{lemFourier}
Let $u,v \in {\mathcal S}(\R)$. Then \eqref{defF1pul} defines a function $\widetilde{\F}_{\rm pul}[u,v] \in {\mathcal S}(\R)$ 
whose Fourier transform is given by
\begin{multline}
\label{expressionFourier}
\xi \in \R \longmapsto \dfrac{1}{2\, |\underline{\omega}_3-\underline{\omega}_1|} \, \widehat{u} 
\left( \dfrac{\underline{\omega}_3-\underline{\omega}_2}{\underline{\omega}_3-\underline{\omega}_1} \, \xi \right) \, \widehat{v} 
\left( \dfrac{\underline{\omega}_2-\underline{\omega}_1}{\underline{\omega}_3-\underline{\omega}_1} \, \xi \right) \\
-\dfrac{1}{2\, i\, \pi} \, \text{\rm p.v.} \int_\R 
\dfrac{\widehat{u} (\eta) \, \widehat{v} (\xi-\eta)}{\eta \, (\underline{\omega}_1-\underline{\omega}_2) 
+(\xi-\eta) \, (\underline{\omega}_3-\underline{\omega}_2)} \, {\rm d}\eta \, ,
\end{multline}
where $\text{\rm p.v.}$ stands for the principal value of the integral.
\end{lem}

\begin{proof}
That $\widetilde{\F}_{\rm pul}[u,v]$ belongs to ${\mathcal S}(\R)$ is done in two steps. One first proves differentiability 
by applying the classical differentiation Theorem for integrals with a parameter. Then the formula
\begin{equation*}
\widetilde{\F}_{\rm pul}[u,v]'=\widetilde{\F}_{\rm pul}[u',v] +\widetilde{\F}_{\rm pul}[u,v'] \, ,
\end{equation*}
and a straightforward induction argument yields infinite differentiability. Given an integer $N$, we can apply the 
Peetre inequality and get
\begin{equation*}
(1+\theta^2)^N \, |\widetilde{\F}_{\rm pul}[u,v](\theta)| \le C \, \sup_{t \in \R} \, ((1+t^2)^N \, |v(t)|) \, 
\int_0^{+\infty} (1+(\theta+(\underline{\omega}_1-\underline{\omega}_2)\, s)^2)^N \, 
|u(\theta+(\underline{\omega}_1-\underline{\omega}_2)\, s)| \, {\rm d}s \, ,
\end{equation*}
so $(1+\theta^2)^N \, \widetilde{\F}_{\rm pul}[u,v]$ is bounded. The previous formula for $\widetilde{\F}_{\rm pul}[u,v]'$ 
shows again by induction that all functions $(1+\theta^2)^{N_1} \, \widetilde{\F}_{\rm pul}[u,v]^{(N_2)}$ are bounded. 
We may thus compute the Fourier transform of $\widetilde{\F}_{\rm pul}[u,v]$.

Let us first write
\begin{equation}
\label{decompF1pul}
\widetilde{\F}_{\rm pul}[u,v](\theta) =F_1(\theta) +F_2(\theta) \, ,
\end{equation}
with
\begin{align*}
F_1(\theta) &:= \dfrac{1}{2} \, \int_\R u(\theta+(\underline{\omega}_1-\underline{\omega}_2)\, s) \, 
v(\theta+(\underline{\omega}_3-\underline{\omega}_2)\, s) \, {\rm d}s \, ,\\
F_2(\theta) &:= \dfrac{1}{2} \, \int_\R \text{\rm sgn}(s) \, u(\theta+(\underline{\omega}_1-\underline{\omega}_2)\, s) \, 
v(\theta+(\underline{\omega}_3-\underline{\omega}_2)\, s) \, {\rm d}s \, .
\end{align*}
Here $\text{\rm sgn}$ denotes the sign function. The Fourier transform of $F_1$ is computed by applying an 
elementary change of variables and the Fubini Theorem:
\begin{align*}
\widehat{F_1}(\xi) &=\dfrac{1}{2} \, \int_{\R^2} {\rm e}^{-i\, \xi \, \theta} \, 
u(\theta+(\underline{\omega}_1-\underline{\omega}_2)\, s) \, 
v(\theta+(\underline{\omega}_3-\underline{\omega}_2)\, s) \, {\rm d}s \, {\rm d}\theta \\
&=\dfrac{1}{2} \, \int_{\R^2} 
{\rm e}^{-i\, \xi \, \dfrac{X\, (\underline{\omega}_3-\underline{\omega}_2)-Y\, 
(\underline{\omega}_1-\underline{\omega}_2)}{\underline{\omega}_3-\underline{\omega}_1}} 
\, u(X) \, v(Y) \, \dfrac{{\rm d}X \, {\rm d}Y}{|\underline{\omega}_3-\underline{\omega}_1|} \, .
\end{align*}
The contribution of $\widehat{F_1}(\xi)$ therefore gives the term in the first line of \eqref{expressionFourier}. 
It remains to compute the Fourier transform of $F_2$.

The Fubini Theorem gives
\begin{equation*}
\widehat{F_2}(\xi) =\dfrac{1}{2} \int_\R \text{\rm sgn}(s) \left( \int_\R {\rm e}^{-i\, \xi \, \theta} \, 
u(\theta+(\underline{\omega}_1-\underline{\omega}_2)\, s) \, v(\theta+(\underline{\omega}_3-\underline{\omega}_2)\, s) \, 
{\rm d}\theta \right) {\rm d}s 
=\dfrac{1}{4\, \pi} \int_\R \text{\rm sgn}(s) \, (\widehat{U_s} \star \widehat{V_s})(\xi) \, {\rm d}s ,
\end{equation*}
with
\begin{equation*}
U_s(\theta) := u(\theta+(\underline{\omega}_1-\underline{\omega}_2)\, s) \, ,\quad 
V_s(\theta) := v(\theta+(\underline{\omega}_3-\underline{\omega}_2)\, s) \, .
\end{equation*}
We thus get
\begin{equation*}
\widehat{F_2}(\xi) =\dfrac{1}{4\, \pi} \, \int_\R \text{\rm sgn}(s) \left( 
{\rm e}^{i\, (\underline{\omega}_1-\underline{\omega}_2) \, s \, \xi} 
\, \int_\R {\rm e}^{i\, (\underline{\omega}_3-\underline{\omega}_1) \, s \, \eta} \, \widehat{u} (\xi-\eta) \, 
\widehat{v} (\eta) \, {\rm d}\eta \right) \, {\rm d}s \, ,
\end{equation*}
which is an expression of the form
\begin{equation*}
\dfrac{1}{4\, \pi} \, \int_\R \text{\rm sgn}(s) \, \widehat{S_\xi}(s) \, {\rm d}s \, ,
\end{equation*}
for a suitable Schwartz function $S_\xi$ ($\xi$ is a parameter here). We can therefore transform the expression of 
$\widehat{F_2}(\xi)$ by using the Fourier transform of the sign function, which yields
\begin{equation*}
\widehat{F_2}(\xi) =\dfrac{1}{2\, i\, \pi} \, \text{\rm p.v.} \int_\R \dfrac{1}{\eta} \, S_\xi (\eta) \, {\rm d}\eta \, .
\end{equation*}
The function $S_\xi$ is given by
\begin{equation*}
S_\xi (\eta) =\dfrac{1}{|\underline{\omega}_3-\underline{\omega}_1|} \, 
\widehat{u} \left( \dfrac{\eta +(\underline{\omega}_3-\underline{\omega}_2) \, \xi}{\underline{\omega}_3-\underline{\omega}_1} 
\right) \, 
\widehat{v} \left( \dfrac{\eta +(\underline{\omega}_1-\underline{\omega}_2) \, \xi}{\underline{\omega}_1-\underline{\omega}_3} 
\right) \, ,
\end{equation*}
and a last change of variable gives, as claimed in \eqref{expressionFourier}:
\begin{equation*}
\widehat{F_2}(\xi) =-\dfrac{1}{2\, i\, \pi} \, \text{\rm p.v.} \int_\R 
\dfrac{\widehat{u} (\eta) \, \widehat{v} (\xi-\eta)}{\eta \, (\underline{\omega}_1-\underline{\omega}_2) 
+(\xi-\eta) \, (\underline{\omega}_3-\underline{\omega}_2)} \, {\rm d}\eta \, .
\end{equation*}
\end{proof}

In view of the decomposition \eqref{decompF1pul}, and the previous computations of Fourier transforms, 
we observe that the formal limit \eqref{Fourierlimit} for the Fourier transform of $\widetilde{\F}_{\rm per} 
[u_\Theta,v_\Theta]$ does not coincide with the Fourier transform of $\widetilde{\F}_{\rm pul}[u,v]$ but rather 
coincides with the Fourier transform of the function $F_2$. In other words, we have {\sl formally} obtained that 
in the large period limit $\Theta \rightarrow +\infty$, $\widetilde{\F}_{\rm per}[u_\Theta,v_\Theta]$ tends towards
\begin{equation}
\label{symappB}
\dfrac{1}{2} \, \int_\R \text{\rm sgn}(s) \, u(\theta+(\underline{\omega}_1-\underline{\omega}_2)\, s) \, 
v(\theta+(\underline{\omega}_3-\underline{\omega}_2)\, s) \, {\rm d}s \, ,
\end{equation}
which is a "symmetrized" version of the operator $\widetilde{\F}_{\rm pul}$ in \eqref{defF1pul}. In particular, 
for any function $a \in {\mathcal S}(\R)$, the $\Theta$-periodic function $\widetilde{\F}_{\rm per}[\partial_\theta 
a_\Theta,a_\Theta]$ formally converges, as $\Theta$ tends to infinity, towards
\begin{equation*}
\dfrac{1}{2} \, \int_\R \text{\rm sgn}(s) \, \partial_\theta a(\theta+(\underline{\omega}_1-\underline{\omega}_2)\, s) \, 
a(\theta+(\underline{\omega}_3-\underline{\omega}_2)\, s) \, {\rm d}s \, .
\end{equation*}

\begin{rem}
It might be surprising that the small divisor condition in Assumption \ref{assumption5} does not seem to 
play any role in the analysis of pulses. However, there remains some kind of trace of this Assumption but 
it is hidden in the functional framework. More precisely, we have shown that for some large enough index 
$\nu$, the bilinear operator $\F_{\rm pul}$ satisfies a bound of the form
\begin{equation*}
\| \F_{\rm pul} (\partial_\theta u,\partial_\theta v) \|_{\Gamma^\nu} \le C \, \| u \|_{\Gamma^\nu} \, \| v \|_{\Gamma^\nu} \, ,
\end{equation*}
where $\Gamma^\nu$ is a weighted Sobolev space. Omitting the variables $(t,y)$ for simplicity, it is also possible 
to prove that for all integer $\nu$, there exists a bounded sequence $(a_n)_{n \in \N}$ in $H^\nu(\R)$ such that
\begin{equation*}
\left\langle a_n,\F_{\rm pul} (\partial_\theta a_n,\partial_\theta a_n) \right\rangle_{H^\nu(\R)} \rightarrow +\infty \, .
\end{equation*}
In other words, the space $\Gamma^\nu$ is well-suited for studying boundedness of $\F_{\rm pul}$ and the 
standard Sobolev space $H^\nu$ is not. This is not so surprising because on the Fourier side, the kernel of 
$\F_{\rm pul}$ is unbounded, which is a trace of the small divisors in the wavetrains problem. Unboundedness 
of the kernel requires introducing a principal value in Lemma \ref{lemFourier}, which relies on some continuity 
of the integrand. It is therefore not surprising that continuity of $\F_{\rm pul}$ holds in a functional space where 
Fourier transforms have some extra regularity properties (which is another way to formulate polynomial decay 
in the physical variable).
\end{rem}

As repeatedly claimed, this paragraph does not aim at giving a {\sl rigorous} justification of the large period 
limit. An open question that is raised by the formal arguments given above is: assuming that the small 
divisor condition in Assumption \ref{assumption5} holds, proving that the $\Theta$-periodic function 
$\widetilde{\F}_{\rm per}[u_\Theta,v_\Theta]$ does indeed converge (for instance, uniformly on compact 
sets), as $\Theta$ tends to infinity, towards \eqref{symappB} for $u,v$ in the Schwartz class. It is also likely 
that the amplitude equation \eqref{eqMachStems1} is ill-posed in $H^\nu(\R^d)$ for any large integer $\nu$, 
though it is well-posed in $\Gamma^\nu$. We refer to \cite{BCT} for similar ill-posedness issues on nonlocal 
versions of the Burgers equation.

\subsection{What is the correct amplitude equation for Mach stem formation ?}

It is surprising that the formal limit of the amplitude equation \eqref{eqa2} as $\Theta$ tends to infinity does 
not coincide with \eqref{eqMachStems'} if the operator $Q_{\rm pul}$ is defined by \eqref{defQpul}. In several 
problems of nonlinear geometric optics, namely
\begin{itemize}
 \item Quasilinear hyperbolic Cauchy problems \cite{HMR,JMR1,AR1,GR},
 
 \item Uniformly stable quasilinear hyperbolic boundary value problems \cite{williams4,W1,CGW1,CW4},
 
 \item Weakly stable semilinear hyperbolic boundary value problems \cite{CGW3,CW5},
\end{itemize}
the limit of the amplitude equation for wavetrains does coincide with the amplitude equation for pulses. 
Of course, nonlinear geometric optics has received a more complete description for these three problems 
than for the one we consider here because in all three above problems, {\sl exact} solutions on a {\sl fixed} 
time interval have been shown to exist and to be close to {\sl approximate} WKB solutions (which justifies 
the relevance of the corresponding evolution equation for the leading order amplitudes).

In view of all available references in the literature, it is therefore natural to wonder whether the amplitude 
equation \eqref{eqMachStems'}, \eqref{defQpul}, which is the one derived in \cite{MR}, does give an accurate 
description for pulse-like solutions to \eqref{0}. Rephrasing the question, is it possible to construct another 
family of approximate solutions that would be determined by solving an amplitude equation that is obtained 
as the large period limit of \eqref{eqa2}, \eqref{defQper} ? The answer is yes, but the price to pay is to 
forget about boundedness of the second corrector $\cU_2$.

Let us quickly review the analysis of the WKB cascade \eqref{3p}, \eqref{4p} for pulses. In Step 1 of Paragraph 
\ref{pconstruction}, we have derived the expression \eqref{decompU0p2} of the leading order profile $\cU_0$ 
by assuming that the {\sl first} corrector $\cU_1$ is bounded. The amplitudes $\sigma_1,\sigma_3$ solve the 
decoupled Burgers equations \eqref{Burgersmp} and satisfy the trace relation \eqref{defa}. In Step 4 of 
Paragraph \ref{pconstruction}, we have obtained the expression of the component of $E_P \, \cU_1$ on $r_2$ 
by assuming that the second corrector $\cU_2$ is bounded. However, this boundedness assumption is not 
necessarily supported by a clear physical interpretation, and it might very well be that $\cU_2$ is not uniformly 
bounded in the stretched variables $(\theta_0,\xi_d)$.

One can indeed construct (infinitely) many families of approximate solutions to \eqref{0}. Let us choose a 
parameter $s \in [0,1]$. Then, given the leading order profile $\cU_0$ in \eqref{decompU0p2} with functions 
$\sigma_1,\sigma_3$ that are expected to decay sufficiently fast at infinity, one can construct the first 
corrector $\cU_1$ as a particular solution to \eqref{3p}(b). One possible choice is:
\begin{multline*}
\widetilde{\cU}_{1,s} :=\sum_{m=1}^3 \left\{ -s\, \int_{\xi_d}^{+\infty} 
\ell_m \, {\mathcal F}_0 (t,x,\theta_0+\underline{\omega}_m \, (\xi_d-X),X) \, {\rm d}X \right. \\
\left. +(1-s) \, \int_{-\infty}^{\xi_d} \ell_m \, {\mathcal F}_0 (t,x,\theta_0+\underline{\omega}_m \, (\xi_d-X),X) 
\, {\rm d}X  \right\} \, r_m \, ,
\end{multline*}
and the general solution to \eqref{3p}(b) is of the form:
\begin{equation*}
\widetilde{\cU}_{1,s} +\sum_{m=1}^3 \tau_m^s(t,x,\theta_0 +\underline{\omega}_m \, \xi_d) \, r_m \, .
\end{equation*}
The choice in \cite{MR} corresponds, as in Paragraph \ref{pconstruction}, to $s=1$ and causality is invoked 
to set the function $\tau^1_2$ equal to zero (see Equation (4.10) in \cite{MR}). In Paragraph \ref{pconstruction}, 
we have explained why $\tau^1_2=0$ can de deduced from the assumption that there exists a bounded corrector 
solution to \eqref{3p}(c).

If we do not assume existence of a bounded corrector to \eqref{3p}(c), there does not seem to be a clear option 
for constructing the outgoing part $\tau_2^s$, and therefore, given $s \in [0,1]$, the corrector we have at our 
disposal satisfies
\begin{multline*}
\widetilde{\cU}_{1,s} (t,y,0,\theta_0,0) = \star \, r_1 -s\, \int_0^{+\infty} \ell_2 \, {\mathcal F}_0 
(t,y,0,\theta_0 -\underline{\omega}_2 \, X,X) \, {\rm d}X \, r_2 \\
+(1-s) \, \int^0_{-\infty} \ell_2 \, {\mathcal F}_0 (t,y,0,\theta_0 -\underline{\omega}_2 \, X,X) \, {\rm d}X \, r_2 
+\star \, r_3 \, .
\end{multline*}
Plugging the latter expression in \eqref{4p}(b), applying the row vector $\underline{b}$ and differentiating 
with respect to $\theta_0$, we obtain Equation \eqref{eqMachStems'} with the following new definition of 
the operator $Q_{\rm pul}$:
\begin{align}
Q_{\rm pul}[a,\widetilde{a}] (\theta_0) := \, &(\underline{b} \, B \, r_2) \, \ell_2 \, E_{1,3} \, s \, \int_0^{+\infty} 
\partial_{\theta_0} a (\theta_0+(\underline{\omega}_1-\underline{\omega}_2)\, X) \, 
\widetilde{a}(\theta_0+(\underline{\omega}_3-\underline{\omega}_2)\, X) \, {\rm d}X \notag \\
&+(\underline{b} \, B \, r_2) \, \ell_2 \, E_{3,1} \, s \, \int_0^{+\infty} 
a(\theta_0 +(\underline{\omega}_1-\underline{\omega}_2)\, X) \, 
\partial_{\theta_0} \widetilde{a} (\theta_0+(\underline{\omega}_3-\underline{\omega}_2)\, X) \, {\rm d}X \notag \\
&-(\underline{b} \, B \, r_2) \, \ell_2 \, E_{1,3} \, (1-s) \, \int^0_{-\infty} 
\partial_{\theta_0} a (\theta_0+(\underline{\omega}_1-\underline{\omega}_2)\, X) \, 
\widetilde{a}(\theta_0+(\underline{\omega}_3-\underline{\omega}_2)\, X) \, {\rm d}X \notag \\
&-(\underline{b} \, B \, r_2) \, \ell_2 \, E_{3,1} \, (1-s) \, \int^0_{-\infty} 
a(\theta_0 +(\underline{\omega}_1-\underline{\omega}_2)\, X) \, 
\partial_{\theta_0} \widetilde{a} (\theta_0+(\underline{\omega}_3-\underline{\omega}_2)\, X) \, {\rm d}X \, .
\label{defQpuls}
\end{align}
Unlike all other terms in \eqref{eqMachStems'}, the bilinear term $Q_{\rm pul}$ does depend on the value of 
$s$ that is chosen for constructing the particular solution $\widetilde{\cU}_{1,s}$ to \eqref{3p}(b), and therefore 
invoking causality to discard the $\tau_2^s$ term has an impact on the leading order amplitude equation 
\eqref{eqMachStems'}. The choice $s=1/2$ is the only that is compatible with the large period limit 
$\Theta \rightarrow +\infty$.
\bigskip

Our overall assessment is the following. In the wavetrain problem, the construction of approximate WKB 
solutions in Part \ref{part1} is well-understood. In particular Theorem \ref{theowavetrains} shows that 
the full WKB cascade has a {\sl unique} solution $(\cU_n)_{n \in \N}$ in a suitable functional space. There 
remains of course the problem of understanding the lifespan of exact solutions to \eqref{0} and whether 
approximate and exact solutions are close to each other. In the pulse problem, the situation is worse 
since the construction of approximate solutions in Part \ref{part2} may be questionable. There are basically 
two options (the infinitely many other ones seem to be a mathematical artifact). Either one hopes that the 
expansion of exact solutions will yield a second corrector $\cU_2$ that is bounded, and in that case the 
appropriate leading amplitude equation is given, as in \cite{MR}, by \eqref{eqMachStems'}, \eqref{defQpul}. 
Or one rather expects that the amplitude equation for pulses should coincide with the large period limit of the 
corresponding wavetrain equation, and in that case the appropriate leading amplitude equation is still given 
by \eqref{eqMachStems'} but with the bilinear operator $Q_{\rm pul}$ as in \eqref{defQpuls} with $s=1/2$. 
The two corresponding leading profiles differ in each of the two options. Without any precise understanding 
of the behavior of exact solutions, deciding between the two possible expansions seems hardly possible.

\section{Some remarks on the resonant case}
\label{appC}

In this Appendix, we explain why the analysis in Part \ref{part1} breaks down when resonances occur. Let us 
consider for simplicity the $3 \times 3$ strictly hyperbolic case of Paragraph \ref{sect2example}. A resonance 
corresponds to the existence of a triplet $(n_1,n_2,n_3) \in \Z^3$, with $n_1 \, n_2 \, n_3 \neq 0$, $\text{\rm gcd } 
(n_1,n_2,n_3) =1$, and
\begin{equation}
\label{resonance}
n_1 \, \varphi_1 =n_2 \, \varphi_2 +n_3 \, \varphi_3 \, .
\end{equation}
Incoming waves associated with the phases $\varphi_1,\varphi_3$ may then interact to produce an outgoing wave 
for the phase $\varphi_2$. More precisely, the source term $\eps^2 \, G(t,y,\varphi_0/\eps)$ in \eqref{0} is expected 
to give rise to incoming waves associated with the phases $\varphi_1,\varphi_3$ of amplitude $O(\eps)$. Then the 
products $A_j(u_\eps) \, \partial_j u_\eps$ in \eqref{0} ignites oscillations associated with the outgoing phase $\varphi_2$ 
of amplitude\footnote{This is a major scaling difference with our previous work \cite{CGW3} where nonlinear interaction 
was due to a zero order term so the outgoing oscillations produced by resonances had amplitude $O(\eps^2)$, 
instead of $O(\eps)$ here.} $O(\eps)$. These $\varphi_2$-oscillations are then expected to be amplified when 
reflected on the boundary, as in the linear analysis of \cite{CG}, but this would give rise to $\varphi_1,\varphi_3$-oscillations 
of amplitude $O(1)$ and would thus completely ruin the ansatz \eqref{BKW}. Resonances are therefore expected to 
produce dramatically different phenomena from the ones that are studied here.

Another hint that resonances should make the ansatz \eqref{BKW} break down is to try to solve the WKB cascade 
\eqref{BKWint}, \eqref{BKWbord} when the resonance \eqref{resonance} occurs between two incoming phases 
$\varphi_1,\varphi_3$ and one outgoing phase $\varphi_2$. Let us recall that in the $3 \times 3$ strictly hyperbolic 
case considered in Paragraph \ref{sect2example}, $B$ is a $2 \times 3$ matrix of maximal rank. Its kernel has 
dimension $1$ and is therefore spanned by the vector $e=e_1+e_3 \in \E^s \tauetabar$. In other words, we can 
choose a vector $\check{e} \in \E^s \tauetabar$ such that
\begin{equation*}
\R^2 =\text{\rm Span } (B \, \check{e},B\, r_2) \, .
\end{equation*}

We now follow the analysis of Paragraph \ref{sect2example} and try to analyze the WKB cascade in the presence 
of a resonance. Equation \eqref{3}(a) shows again that the leading profile can be decomposed as
\begin{equation*}
\cU_0 (t,x,\theta_1,\theta_2,\theta_3) =\underline{\cU}_0(t,x) +\sum_{m=1}^3 \sigma_m(t,x,\theta_m) \, r_m \, .
\end{equation*}
Then Equation \eqref{3}(b) shows that the mean value $\underline{\cU}_0$ satisfies \eqref{eq:moyenne1}, 
\eqref{eq:moyenne2}, and therefore vanishes. The interior equation satisfied by each $\sigma_m$ exhibits 
the resonance between the phases, see \cite{rauch}. Namely, the $\sigma_m$'s satisfy the coupled Burgers-type 
equations
\begin{equation}
\label{systresonant}
\begin{cases}
\partial_t \sigma_1 +{\bf v}_1 \cdot \nabla_x \sigma_1 +c_1 \, \sigma_1 \, \partial_{\theta_1} \sigma_1 
=B_1 (\sigma_2,\sigma_3) \, ,& \\
\partial_t \sigma_2 +{\bf v}_2 \cdot \nabla_x \sigma_2 +c_2 \, \sigma_2 \, \partial_{\theta_2} \sigma_2 
=B_2 (\sigma_1,\sigma_3) \, ,& \\
\partial_t \sigma_3 +{\bf v}_3 \cdot \nabla_x \sigma_3 +c_3 \, \sigma_3 \, \partial_{\theta_3} \sigma_3 
=B_3 (\sigma_1,\sigma_2) \, ,
\end{cases}
\end{equation}
where the constants $c_m$ are defined as in \eqref{eq:Burgersm} and, for instance:
\begin{align*}
B_1 (\sigma_2,\sigma_3) (t,x,\theta_1) :=& \dfrac{2\, i \, \pi \, \ell_1 \, \partial_j \varphi_3 \, ({\rm d}A_j(0) \cdot r_2) 
\, r_3}{\Theta \, \ell_1 \, r_1} \sum_{k \in \Z^*} c_{k\, n_2}(\sigma_2) (t,x) \, (k\, n_3) \, c_{k\, n_3}(\sigma_3) (t,x) \, 
{\rm e}^{2\, i \, \pi \, k \, n_1 \, \theta_1/\Theta} \\
&+\dfrac{2\, i \, \pi \, \ell_1 \, \partial_j \varphi_2 \, ({\rm d}A_j(0) \cdot r_3) \, r_2}{\Theta \, \ell_1 \, r_1} 
\sum_{k \in \Z^*} (k\, n_2) \, c_{k\, n_2}(\sigma_2) (t,x) \, c_{k\, n_3}(\sigma_3) (t,x) \, 
{\rm e}^{2\, i \, \pi \, k \, n_1 \, \theta_1/\Theta} \, .
\end{align*}
The definitions of $B_2(\sigma_1,\sigma_3)$ and $B_3 (\sigma_1,\sigma_2)$ are similar, and correspond to 
the so-called {\sl interaction integrals} in \cite{CGW1}.

The boundary conditions for the $\sigma_m$'s are given by \eqref{4}(a). Using the basis $(e,\check{e})$ of 
$\text{\rm Span } (r_1,r_3)$, we decompose
\begin{equation*}
\cU_0(t,y,0,\theta_0,\theta_0,\theta_0) =a(t,y,\theta_0) \, e +\check{a}(t,y,\theta_0) \, \check{e} 
+\sigma_2(t,y,0,\theta_0) \, r_2 \, ,
\end{equation*}
so \eqref{4}(a) gives
\begin{equation*}
\check{a} \equiv 0 \, ,\quad \sigma_2|_{x_d=0} \equiv 0 \, .
\end{equation*}
However, this boundary condition on  $\sigma_2$ does not seem to be compatible with \eqref{systresonant} 
because $\sigma_2$ satisfies an outgoing transport equation with a nonzero source term (it is proved in 
\cite{CGW1} that the operators $B_1,B_2,B_3$ in \eqref{systresonant} act like {\sl semilinear} terms and 
do not contribute to the leading order part of the differential operators in \eqref{systresonant}). Even if we 
manage to isolate an amplitude equation for determining the traces of the incoming amplitudes $\sigma_1, 
\sigma_3$, the overall cascade seems to give rise to an overdetermined problem for $\cU_0$.

We do not claim of course that the above arguments give a rigorous justification of the ill-posedness 
of \eqref{BKWint}, \eqref{BKWbord} when a resonance occurs, but it clearly indicates that the ansatz 
\eqref{BKW} does not seem to be appropriate anylonger.

When there is only one incoming phase, all the above discussion becomes irrelevant, and this may again 
be one reason why any discussion on resonances is absent from \cite{AM} or \cite{wangyuJDE,wangyu}.

\newpage
\bibliographystyle{alpha}
\bibliography{MachStems}
\end{document}